%% file: models-new.tex
\documentclass[11pt]{article}
\usepackage{amssymb,amsmath,amsfonts,amsthm,epsfig,latexsym,color}

\makeatletter
\newcommand{\subsectionruninhead}{\@startsection{subsection}{2}{0mm}
{-\baselineskip}{-0mm}{\bf\large}}
\newcommand{\subsubsectionruninhead}{\@startsection{subsubsection}{3}{0mm}
{-\baselineskip}{-0mm}{\bf\normalsize}}
\makeatother

\newtheorem*{theorem*}{Theorem}
\newtheorem*{maintheorem}{Main Theorem}
\newtheorem*{addenduma}{Addendum A}
\newtheorem*{addendumb}{Addendum B}
\newtheorem*{strong}{Palis conjecture}
\newtheorem{theorem}{Theorem}
\newtheorem{proposition}{Proposition}[section]
\newtheorem*{proposition*}{Proposition}
\newtheorem*{corollary*}{Corollary}
\newtheorem{corollary}[proposition]{Corollary}
\newtheorem{lemma}[proposition]{Lemma}

\newtheorem{claim}[proposition]{Claim}
\newtheorem*{claim*}{Claim}
\theoremstyle{definition}
\newtheorem{definition}[proposition]{Definition}
\theoremstyle{remark}
\newtheorem{remark}[proposition]{Remark}
\newtheorem{remarks}[proposition]{Remarks}

\numberwithin{equation}{section}

 \def\RR{{\mathbb R}}  
   
 \def\ZZ{{\mathbb Z}}

    \def\cS{\mathcal{S}}
    
\def\cC{\mathcal{C}}   \def\cO{\mathcal{O}} \def\cU{\mathcal{U}}
\def\cD{\mathcal{D}}    \def\cV{\mathcal{V}}
    
   \def\cR{\mathcal{R}}

\newcommand{\Tang}{\operatorname{Tang}}
\newcommand{\Cycl}{\operatorname{Cycl}}

\newcommand{\supp}{\operatorname{Supp}}
\newcommand{\diff}{\operatorname{Diff}}

\newcommand{\diam}{\operatorname{Diam}}

\begin{document}
\title{Partial hyperbolicity far from homoclinic bifurcations}
\author{Sylvain Crovisier}
\date{\today}
\maketitle

\begin{abstract}
We prove that any diffeomorphism of a compact manifold can be
$C^1$-approximated by a diffeomorphism which exhibits a homoclinic bifurcation
(a homoclinic tangency or a heterodimensional cycle) or by a diffeomorphism
which is partially hyperbolic (its chain-recurrent set splits into partially hyperbolic
pieces whose centre bundles have dimensions less or equal to two).
We also study in a more systematic way the central models introduced in~\cite{crovisier-palis-faible}.
\begin{center}
\bf R\'esum\'e
\end{center}
\emph{Hyperbolicit\'e partielle loin des bifurcations homoclines.}
Nous montrons que tout dif\-f\'e\-o\-mor\-phis\-me d'une vari\'et\'e compacte peut \^etre approch\'e en topologie
$C^1$ par un diff\'eomorphisme qui pr\'esente une bifurcation homocline (une tangence homocline ou un
cycle h\'e\-t\'e\-ro\-di\-men\-si\-on\-nel) ou bien par un diff\'eomorphisme partiellement hyperbolique
(son ensemble r\'ecurrent par cha\^\i nes se d\'ecompose en pi\`eces partiellement hyperboliques dont les fibr\'es centraux sont de dimensions au plus \'egales \`a 2). Nous \'etudions \'egalement d'un point de vue plus syst\'ematique
les mod\`eles de dynamiques centrales introduits en~\cite{crovisier-palis-faible}.
\vskip 2mm

\begin{description}
\item[\bf Key words:] Homoclinic tangency, heterodimensional cycle, hyperbolic diffeomorphism, ge\-ne\-ric dynamics,
homoclinic class, partial hyperbolicity.
\item[\bf MSC 2000:] 37C05, 37C20, 37C29, 37C50, 37D25, 37D30.
\end{description}
\end{abstract}

\setcounter{section}{-1}
\tableofcontents

\section{Introduction}
\subsection{Towards a characterisation of non-hyperbolic systems}
One of the goal of the dynamics is the study of the evolution of most systems.
It appeared in the early sixties with Smale~\cite{smale}
that the dynamics on a compact manifold $M$
can be well described from topological and statistical points of view for
a large class of differentiable systems, the \emph{uniformly hyperbolic} ones:
these are the dynamics whose recurrence locus (the chain-recurrent set)
decomposes into a finite number of invariant pieces whose tangent spaces
split into two invariant subbundles $TM=E^s\oplus E^u$, the first being uniformly contracted and the
second uniformly expanded.  These dynamics have a nice property: they are stable (hence form an open set of systems). But they are not ``universal": there exist open sets of systems which are not hyperbolic.

Two simple obstructions for the hyperbolicity have been discovered.
On the one hand Abraham and Smale~\cite{abraham-smale} have build an open set of diffeomorphisms
such that two hyperbolic periodic orbits with different stable dimensions are linked by the dynamics:
by perturbation one gets a diffeomorphism having a \emph{heterodimensional cycle}, i.e.
such that the stable manifold of each of these orbits intersects the unstable manifold of the other.
On the other hand Newhouse~\cite{newhouse} has obtained a $C^2$-open set of diffeomorphisms
which possess a hyperbolic set whose stable and unstable laminations have non-transverse intersections:
by perturbation one gets hyperbolic periodic points whose stable and unstable manifolds
have a \emph{homoclinic tangency}. These two mechanisms have taken a major role in the study of non-hyperbolic
dynamics, see~\cite{PT,BDV}.
\medskip

Palis has proposed~\cite{palis-conjecture}
that these two obstructions characterise the non-hyperbolic systems.
\begin{strong}
Any $C^r$-diffeomorphism ($r\geq 1$) of a compact manifold
can be $C^r$-appro\-xi\-ma\-ted by a hyperbolic diffeomorphism or
by a diffeomorphism having a homoclinic tangency or a heterodimensional cycle.
\end{strong}
If this holds we would have on one side some systems whose global dynamics is
well described and stable, and on the other side a dense set of systems presenting
local bifurcations easy to detect (they involve periodic orbits and their local invariant manifolds)
and which generate important changes on the dynamics.

Due to the difficulties to perturb the orbits of a system in topology $C^r$ for $r>1$,
the main progresses have only been obtained for the $C^1$-dynamics.
On surfaces, the conjecture has been solved by E. Pujals and M. Sambarino~\cite{pujals-sambarino1}.
In higher dimensions, some partial results have been obtained, for instance~\cite{wen-conjecture,pujals2}.
\medskip

One approach consists in studying the generic dynamics in the set
$\diff^1(M)\setminus \overline{\Tang\cup \Cycl}$ of diffeomorphisms far from homoclinic tangencies
and in proving some hyperbolicity properties:
various forms of weak hyperbolicity have been already considered.
For instance, a splitting $TM=E\oplus F$ of the tangent bundle above an invariant set $K$
is \emph{dominated} if there exists $n\geq 1$ such that for any $x\in K$ and unitary vectors
$u\in E_x, v\in F_x$ one has $\|Df^n.u\|<\|Df^n.v\|$.
A set $K$ is \emph{partially hyperbolic} if the tangent bundle above $K$
admits a dominated splitting $E^s\oplus E^c\oplus E^u$ such that $E^s$ is uniformly contracted
and $E^u$ is uniformly expanded (one allows to one of the bundles $E^s$ and $E^u$ to be degenerated
but not both): the difference with the uniform hyperbolicity is the presence of a central bundle.
Generalising the definition of hyperbolic diffeomorphisms, one defines
the \emph{partially hyperbolic diffeomorphisms} as the diffeomorphisms whose chain-recurrent set
decomposes into finitely many invariant pieces, each of them being partially hyperbolic.

One can also benefit from the recent progresses towards the understanding of general $C^1$-generic
dynamics. In particular, with C. Bonatti~\cite{BC} we have raised the interest of Conley's decomposition
of the chain-recurrent set into (disjoint, invariant and compact) chain-recurrence classes.
For a $C^1$-generic diffeomorphism any chain-recurrence class which contains a hyperbolic periodic orbit $\cO$
coincides with the \emph{homoclinic class} $H(\cO)$ of $\cO$: this is the smallest compact set which
contains all the periodic orbits $\cO'$ such that $\cO,\cO'$ are \emph{homoclinically related},
that is such that the stable manifold of each of these periodic orbits intersects transversely
the unstable manifold of the other. The chain-recurrence classes without periodic points are called \emph{aperiodic classes}.
\medskip

\subsection{Main results}
In~\cite{wen-conjecture}, L. Wen
has shown that for any $C^1$-generic diffeomorphism far from homoclinic bifurcations
the minimally non-hyperbolic sets are partially hyperbolic
with centre bundles of dimensions smaller or equal to two.
Our main result is a global version of this work.

\begin{maintheorem}
There exists a dense G$_\delta$ subset $\cR$ of $\diff^1(M)\setminus \overline{\Tang\cup\Cycl}$
such that any diffeomorphism $f\in \cR$ is \emph{partially hyperbolic}:
\begin{itemize}
\item[--] each homoclinic class $H$ has a dominated splitting
$T_{H}M=E^s\oplus E^c_1\oplus E^c_2\oplus E^u$ which is partially hyperbolic
and such that each central bundle $E^c_1, E^c_2$ has dimension $0$ or $1$;
\item[--] each aperiodic class $\cC$ of $f$ has a partially hyperbolic structure
$T_\cC M=E^s\oplus E^c \oplus E^u$ with $\dim(E^c)=1$.
\end{itemize}
\end{maintheorem}
\medskip

\noindent
The Main theorem is a consequence of two results which provide
more precise informations on the chain-recurrence classes of $f$:

\begin{addenduma}
Each \emph{homoclinic class} $H=H(\cO)$ of $f\in \cR$ is ``weakly hyperbolic":
\begin{itemize}
\item[--] $H$ has a dominated splitting $T_{H}M=E\oplus F$ such that $\dim(E)$
is equal to the stable dimension of $\cO$.
\item[--] There is a continuous family of ``chain-stable" manifolds
$(\cD^{cs}_x)_{x\in H}$ of dimension $\dim(E)$:
for each point $x\in H$, $\cD^{cs}_x$ is a $C^1$-embedded disk tangent to $E_x$
contained in the chain stable set of $H$
and there exist arbitrarily small open neighbourhoods $U_x$
of $x$ in $\cD^{cs}_x$ such that the family $(U_x)_{x\in H}$ satisfies:
$$f(\overline{U_x})\subset U_{f(x)}.$$
\item[--] There is a continuous family of ``chain-unstable" manifolds
$(\cD^{cu}_x)_{x\in H}$ of dimension $\dim(F)$ satisfying the same properties for $f^{-1}$.
\end{itemize}

\noindent Moreover if the bundle $E$ is not uniformly contracted, then:
\begin{itemize}
\item[--] there is a dominated decomposition
$E=E^s\oplus E^c_1$ such that $E^s$ is uniformly contracted and $E^c_1$ has dimension $1$,
\item[--] the class $H$ contains some periodic points $q$ homoclinically related to $\cO$
and whose Lyapunov exponent along $E^c_1$ is arbitrarily close to $0$.
\end{itemize}
A similar property holds if $F$ is not uniformly expanded.
\end{addenduma}
\medskip

\begin{addendumb}
Each \emph{aperiodic class} $\cC$ of $f\in \cR$ is a minimal set
and its dynamics along the central bundle $E^c$ is ``neutral":
there is a continuous family of $C^1$-embedded curves
$(\cD^{c}_x)_{x\in \cC}$ tangent to $E^c_x$
and there exist arbitrarily small open neighbourhoods $U_x, V_x$ of $x$ in $\cD^{cs}_x$
such that the families $(U_x)_{x\in H(p)}$, $(V_x)_{x\in H(p)}$ satisfy:
$$f(\overline{U_x})\subset U_{f(x)}\quad \text{and} \quad
f^{-1}(\overline{V_x})\subset V_{f^{-1}(x)}.$$
In particular the Lyapunov exponent along $E^c$ for each invariant probability measure
supported on $\cC$ is equal to $0$.
\end{addendumb}

\begin{remark}
The proof will show that the extremal bundles $E^s,E^u$ of the aperiodic classes
are not degenerated. For the homoclinic classes (different from the sinks and sources)
the same result holds: this follows from~\cite{pujals-sambarino1,pujals-sambarino2}
if $\dim(M)=2$ or if the central bundle $(E^c_1\oplus E^c_2)$ has dimension $1$;
when the central bundle of the class is two-dimensional, this follows from~\cite{CP}.
\end{remark}
\begin{remark}
There are only finitely many homoclinic classes whose central bundle has dimension $2$
(see the end of section~\ref{s.conclusion}).
\end{remark}
\medskip

Let us explain how this is related to other works.
With Pujals we use these results for proving~\cite{CP}
that any $C^1$-generic diffeomorphism far from the homoclinic bifurcations
is \emph{essentially hyperbolic}: there exists a finite number of hyperbolic attractors
whose basin is dense in the manifold.
Recently J. Yang has announced~\cite{yang-aperiodic} that a part of the Main Theorem
can be extended to the diffeomorphisms far from the homoclinic tangencies:
any aperiodic class of a $C^1$-generic diffeomorphism
$f\in \diff^1(M)\setminus \overline{\Tang}$ has a partially hyperbolic structure
$TM=E^s\oplus E^c\oplus E^u$ with $\dim(E^c)=1$.
\medskip

The setting of our proofs is the $C^1$-generic dynamics far from the homoclinic tangencies.
Three ingredients are used.
\begin{itemize}
\item[--] Wen has shown~\cite{wen-tangence} that the splitting into stable and unstable
spaces above periodic orbits having a same stable dimension is dominated.
\item[--] One then uses Liao's selecting lemma in order to discuss the hyperbolicity properties of these
splittings: when a bundle is not uniformly contracted or expanded, the dynamics contains
a minimal set $K$ which possesses a partially hyperbolic structure $E^s\oplus E^c\oplus E^u$
with $\dim(E^c)=1$. The dynamics in the central direction $E^c$ is not well described by the dynamics
of the tangent map, since all the measures supported on $K$ have a Lyapunov exponent along $E^c$ equal to zero.
\item[--] One thus describes the dynamics in the central direction through a more topological approach:
this is done by using the central models introduced in~\cite{crovisier-palis-faible}.
\end{itemize}
The main part of the paper is devoted to a more systematic study of the central models.
The topological description of the central dynamics of a partially hyperbolic set
can be illustrated in the simple case of a germ of an orientation preserving diffeomorphism $f$ of $\RR$
that fixes a point $p$. The dynamics falls in one of the following types
(see figure~\ref{f.type}):
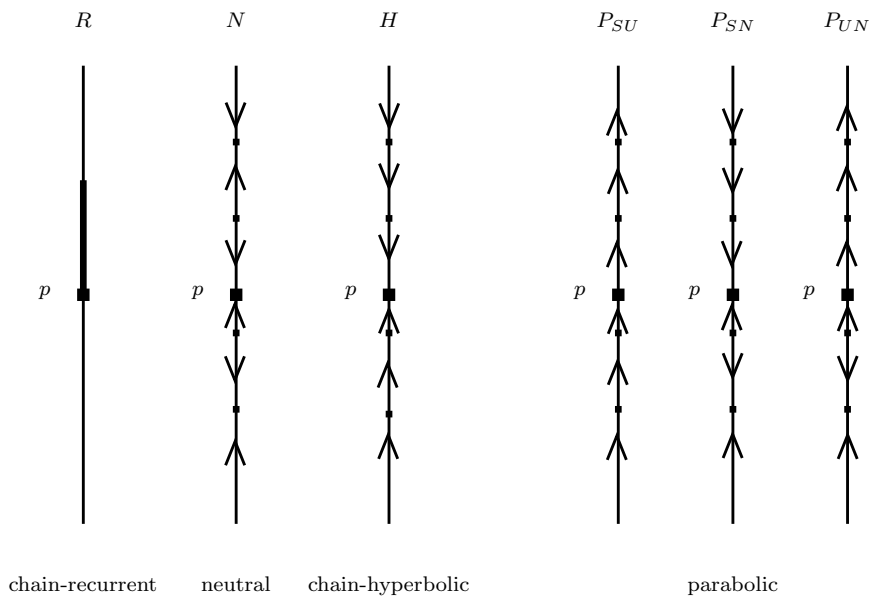
\begin{figure}[ht]
\begin{center}
\input{type.pstex_t}
\end{center}
\caption{Types for the central dynamics near a fixed point.\label{f.type}}
\end{figure}
\begin{itemize}
\item[--] \emph{Chain-recurrent}: a non-trivial interval of fixed points contains $p$.
\item[--] \emph{Neutral}: $p$ is accumulated on both sides by wandering intervals that are mapped to the right
and other that are mapped to the left.
\item[--] \emph{Chain-hyperbolic}: in the attracting case,
$p$ is accumulated on both sides by wandering intervals that are mapped towards $p$
and there is no wandering interval mapped away from $p$.
One defines similarly the chain-hyperbolic repelling type.
\item[--] \emph{Parabolic}: one side is chain-hyperbolic attracting and the other chain-hyperbolic repelling;
or one side is chain-hyperbolic and the other side is neutral.
\end{itemize}
If the diffeomorphism is $C^1$-generic,
we then prove that if the set $K$ has chain-recurrent or chain-hyperbolic type
or if it has parabolic type and is non-twisted, then
$K$ is contained in a homoclinic class.
If $K$ has parabolic type and a twisted geometry, one can get a heterodimensional cycle by perturbation.
If $K$ has neutral type and is contained in a larger chain-transitive set $\Lambda$,
there exist in $\Lambda$ some pseudo-orbit connecting $K$ to a given point in $\Lambda\setminus K$:
the neutral property should imply that these pseudo orbits leave a small neighbourhood of $K$
along the strong stable and strong unstable leaves of $K$; it is then possible to create a
heterodimensional cycle.

\paragraph{Acknowledgements.}
\emph{I am gratefull to Christian Bonatti,
Rafael Potrie, Enrique Pujals, Martin Sambarino and Dawei Yang
for their comments about this work.}

\subsection{Notations and definitions}
We fix a diffeomorphism $f\in \diff^1(M)$ and recall briefly some definitions.
\smallskip

A sequence $(z_n)$ in $M$
is a $\varepsilon$-\emph{pseudo-orbit} if $d(f(z_n),z_{n+1})$ is smaller than $\varepsilon>0$
for each $n$.
An invariant compact set $K$ is \emph{chain-transitive} if
for any $\varepsilon>0$ it contains a periodic $\varepsilon$-pseudo-orbit
that is $\varepsilon$-dense in $K$.
The maximal chain-transitive sets are the \emph{chain-recurrence classes}:
two classes are disjoint or equal. The union of the chain-recurrence classes is a compact set
called the \emph{chain-recurrent set}.
If $K$ is an invariant compact set, its \emph{chain-stable set} is the set of points
$x$ that may be joint to $K$ by a $\varepsilon$-pseudo-orbit for any $\varepsilon>0$;
its \emph{chain-unstable set} is the chain-stable set of $K$ for $f^{-1}$.
\smallskip

The \emph{index} of a hyperbolic periodic orbit $\cO$ is the dimension of its stable space $E^s(\cO)$.
Its \emph{homoclinic class} $H(\cO)$ is the closure of the transverse intersection points between
the invariant manifolds $W^s(\cO)$ and $W^u(\cO)$. Two hyperbolic periodic orbits $\cO,\cO'$
are \emph{homoclinically related} if both $W^s(\cO)\cap W^u(\cO')$ and
$W^u(\cO)\cap W^s(\cO')$ contain transverse intersection points: the homoclinic class $H(\cO)$
also coincides with the closure of the set of hyperbolic periodic points whose orbit is homoclinically related
to $\cO$. A \emph{homoclinic tangency} associated to a hyperbolic periodic orbit $\cO$
is a non-transverse intersection point in $W^s(\cO)\cap W^u(\cO)$.
A \emph{cycle} associated to two hyperbolic periodic orbits $\cO,\cO'$ is the compact invariant set
obtained as union of $\cO$, $\cO'$
with two heteroclinic orbits: one in $W^s(\cO)\cap W^u(\cO')$ and one in $W^u(\cO)\cap W^s(\cO')$.
When the indices of $\cO,\cO'$ are different, the cycle is \emph{heterodimensional}.
\smallskip

If $E$ is a continuous invariant one-dimensional bundle over an invariant compact set $K$ and 
if $\mu$ is an invariant probability measure supported on $K$, the \emph{Lyapunov exponent of $\mu$
along $E$} is
$$\lambda(\mu,E^c)=\int \log \|Df_{|E}\|d\mu.$$
A measure is \emph{hyperbolic} if its Lyapunov exponents are all non-zero.
\smallskip

Let $K$ be an invariant set.
A decomposition $T_KM=E\oplus F$ of the tangent bundle over $K$ into
two linear subbundles is a \emph{dominated splitting} if there exists an integer $N\geq 1$ such that
for any point $x\in K$ and any unitary vectors $u\in E_x$ and $v\in F_x$, we have
$$2\|D_xf^N.u\|\leq \|D_xf^N.v\|.$$
(One says that the decomposition is \emph{$N$-dominated}.)

A linear subbundle $E$ of $T_{K}M$
is \emph{uniformly contracted} if there exists an integer $N\geq 1$
such that for any point $x\in K$ one has $\|D_xf^N\|\leq \frac{1}{2}$.
An invariant compact set $K$ is \emph{partially hyperbolic} if its tangent bundle
decomposes as the sum of three linear subbundles $E^{s}\oplus E^c\oplus E^{u}$
(at most one of the two extremal bundles $E^s,E^u$ is non-trivial) so that:
\begin{itemize}
\item The splittings $E^{s}\oplus (E^c\oplus E^{u})$ and $(E^{s}\oplus E^c)\oplus E^{u}$
are dominated.
\item The bundle $E^{s}$ is uniformly contracted by $f$ and the bundle $E^{u}$ is
uniformly contracted by $f^{-1}$.
\end{itemize}
The partially hyperbolic structure extends to any
invariant compact set contained in a small neighbourhood of $K$.
(We refer to~\cite{BDV} for a survey on the properties satisfied by partially hyperbolic
sets and dominated splittings.)
\smallskip

If $K$ is an invariant compact set and $E$ a continuous linear subbundle of $T_KM$,
a \emph{cone around $E$} is a cone field defined on a neighbourhood of $K$
induced by a neighbourhood of $E$ in the unitary tangent space of $M$.
\smallskip

Let $K$ be a compact set which has a partially hyperbolic structure
$E^s\oplus E^c\oplus E^u$ such that $E^c$ is one-dimensional.
A periodic orbit $\cO\subset K$ has a \emph{strong homoclinic intersection} if
$W^{ss}(\cO)\setminus \cO$ and $W^{uu}(\cO)\setminus \cO$ have an intersection.
Strong homoclinic intersections sometimes produce heterodimensional cycles by $C^1$-per\-tur\-ba\-tion.
The following proposition is proved in~\cite[proposition 2.4]{pujals1} and~\cite[proposition 2.5]{pujals2}.
\begin{proposition}\label{p.strong-cycle}
Let $f$ be a diffeomorphism, $\cU$ be a neighbourhood of $f$ in $\diff^1(M)$
and $K$ be a compact set which has a partially hyperbolic structure
$E^s\oplus E^c\oplus E^u$ such that $E^c$ is one-dimensional.
Then, there exists $\delta>0$ such that,
for any periodic orbit $\cO\subset K$  whose central exponent along $E^c$ belongs to $[-\delta,\delta]$ and which has a strong homoclinic connection
there exists $g\in \cU$ which has a heterodimensional cycle.

Moreover if  $U$ is a neighbourhood of $\cO$ and $V$ is an open set containing $\cO$ and the
homoclinic orbit, then  the heterodimensional cycle is contained in $V$ and associated to
periodic orbits contained in $U$.
\end{proposition}


\section{Non-hyperbolic dominated splittings}
In this section $f$ is a diffeomorphism that is not the $C^1$-limit of diffeomorphisms
exhibiting a homoclinic tangency.
We prove the existence of invariant compact sets
having dominated and partially hyperbolic splittings with a one-dimensional
central bundle.

\subsection{Existence of dominated splittings}
We first recall a theorem of Wen~\cite{wen-tangence}
(an improved version was proved by N. Gourmelon, see~\cite{gourmelon-tangence}).

\begin{theorem*}[Wen]
Let $f$ be a diffeomorphism in $\diff^1(M)\setminus \overline{\Tang}$ and $i\in\{1,\dots,\dim(M)-1\}$.
Then, the tangent bundle above the set of hyperbolic periodic points with index $i$ has a dominated decomposition
$E\oplus F$ with $dim(E)=i$.
\end{theorem*}
One can derive a uniform version of this result (see~\cite[lemmas 3.3 and 3.4]{wen-conjecture}):

\begin{corollary}[Wen]\label{c.wen}
Let $f$ be a diffeomorphism in $\diff^1(M)\setminus \overline{\Tang}$.
Then, there exist a $C^1$-neighbourhood $\cU$ of $f$,
some integers $N,\nu\geq 1$, some constants $\delta_0,C>0$ and $\lambda>1$ such that, for any $g\in \cU$ and any periodic orbit $\cO$ of $g$, the following property holds.

Let $E^s\oplus E^c\oplus E^u$ be the decomposition of the tangent bundle $TM$ over $\cO$
into characteristic spaces whose Lyapunov exponents belong to 
$]-\infty,-\delta_0[$, $[-\delta_0,\delta_0]$ and $]\delta_0, +\infty[$ respectively. Then,
\begin{enumerate}
\item[1)] $E^c$ has dimension at most $1$;
\item[2)] the splitting $E^s\oplus E^c\oplus E^u$ is $N$-dominated;
\item[3)] denoting by $\tau$ the period of $\cO$, we have for any $x\in \cO$,
$$\prod_{i=0}^{\tau/\nu-1}\|D f^{\nu}_{|E^s}(f^{i\nu}(x))\|\leq C.\lambda^{-\tau}
\text{ and } \prod_{i=0}^{\tau/\nu-1}\|D f^{-\nu}_{|E^u}(f^{-i\nu}(x))\|\leq C.\lambda^{-\tau}.$$
\end{enumerate}
\end{corollary}

\begin{remark}
Consider a diffeomorphism $f$ whose periodic orbits are hyperbolic.
By~\cite{crovisier-approximation} any chain-recurrence class of $f$ is the Hausdorff limit of hyperbolic periodic orbits associated to a sequence of diffeomorphisms that converges towards $f$ in $\diff^1(M)$.
A result of Pliss~\cite{pliss} implies that there exists only finitely many sinks or sources
whose exponents are all in the complement set of $[\delta_0,\delta_0]$.
One deduces that there exists $N\geq 1$ such that
all chain-recurrence classes of $f$ but maybe finitely many sinks and sources have a non-trivial
dominated splitting that is $N$-dominated.
\end{remark}

Using an improved version of Ma\~n\'e's ergodic closing lemma provided by~\cite{abc-mesure}, one can extend these results
to the ergodic measures. A similar result has been obtained independently by Yang~\cite{yang-mesure}.

\begin{corollary}\label{c.measure-tangency}
Let $f$ be a diffeomorphism in $\diff^1(M)\setminus \overline{\Tang}$.
Then, there exist a $C^1$-neighbourhood $\cU$ of $f$,
some integers $N,\nu\geq 1$ and some constants $\delta,\rho>0$
such that for any $g\in \cU$ and any ergodic measure $\mu$ of $g$ the following property holds.

Let $E^s\oplus E^c\oplus E^u$ be the Oseledets decomposition of the tangent bundle $TM$ at $\mu$-a.e. point
into spaces whose Lyapunov exponents belong to 
$]-\infty,-\delta[$, $[-\delta,\delta]$ and $]\delta, +\infty[$ respectively. Then,
\begin{enumerate}
\item[1)] $E^c$ has dimension at most $1$;
\item[2)] the splitting $E^s\oplus E^c\oplus E^u$ is $N$-dominated (and hence extends to the
support of $\mu$),
\item[3)] for $\mu$-a.e. point $x$ we have
$$\lim_{k\to \infty}\frac 1 k \sum_{i=0}^{k-1} \log \|D f^{\nu}_{|E^s}(f^{i\nu}(x))\|
\leq -\rho \;\text{ and}\;
\lim_{k\to \infty}\frac 1 k \sum_{i=0}^{k-1} \log \|D f^{-\nu}_{|E^u}(f^{-i\nu}(x))\|
\leq -\rho.$$
\end{enumerate}
\end{corollary}
\begin{proof}
Let us consider the constants $N,\nu,\delta_0,\lambda$ given by Wen's corollary~\ref{c.wen} above
and choose $\delta\in (0,\delta_0)$ and $\rho\in (0,\log(\lambda))$.
By~\cite[proposition~6.1]{abc-mesure}, there exists a sequence of
diffeomorphisms $(f_n)$ that converge toward $f$ and for each $f_n$ there exists
a periodic orbit $\cO_n$ with the following properties.
The Lyapunov exponents of the orbits $\cO_n$ for $f_n$ converge
towards the Lyapunov exponents of $\mu$ for $f$;
moreover the $f_n$-invariant probability measures supported on the orbits $\cO_n$
converge toward $\mu$.
By Wen's corollary~\ref{c.wen}, each orbit $\cO_n$ admits an $N$-dominated
splitting $E_1\oplus E_2\oplus E_3$ such that the dimensions of
$E_1,E_2,E_3$ coincide with the dimensions of $E^s,E^c,E^u$.
This implies that $E^c$ is trivial or one-dimensional.

Since the domination extends to the closure, the support of $\mu$ also carries
an $N$-dominated splitting $E_1\oplus E_2\oplus E_3$. It coincides with
$E^s\oplus E^c\oplus E^u$ at $\mu$-a.e. point, by definition of Oseledets splitting.
This proves the first item.

The third item of Wen's corollary and a classical argument due to Pliss
now imply that a uniform proportion of points $x$ in each orbit $\cO_n$ satisfies
for any $k\geq 0$,
$$\prod_{i=0}^{k-1}\|D {f_n^{\nu}}_{|E^s}(f_n^{i\nu}(x))\|\leq e^{-k\rho}.$$
Since the measures supported on the $\cO_n$ approximate the measure $\mu$,
one deduces that the same property holds for $f$ and
the points $x$ in a set $X$ with positive
$\mu$-measure. By ergodicity, $\mu$-a.e. point $x$ has a forward iterate in $X$
and hence satisfies
$$\lim_{k\to \infty}\frac 1 k \sum_{i=0}^{k-1} \log \|D f^{\nu}_{|E^s}(f^{i\nu}(x))\|
\leq -\rho.$$
The same argument applies to the bundle $E^u$.
\end{proof}
\medskip

The classical Anosov closing lemma is still valid for hyperbolic ergodic measures of $C^1$-diffeo\-mor\-phisms,
provided that the Oseledets decomposition is dominated.
\begin{proposition}\label{p.anosov}
Let $f$ be a $C^1$-diffeomorphism and $\mu$ a hyperbolic ergodic measure
whose Oseledets decomposition $E^s\oplus E^u$ is dominated.
Then $\mu$ is supported on a homoclinic class: there exists a sequence of hyperbolic periodic orbits
$(\cO_n)$ whose index is equal to $\dim(E^s)$, that are homoclinically related together, that converge toward the support of $\mu$ for the Hausdorff topology
and such that the invariant measures supported on the $\cO_n$ converge toward $\mu$
for the weak-$*$ topology.
\end{proposition}
\begin{proof}
By~\cite[lemma 8.4]{abc-mesure}, there exists $\rho>0$ and an integer $\nu\geq 1$
such that for $\mu$-almost every point $x\in M$,
the Birkhoff averages
$$\frac{1}{\nu k}\sum_{i=0}^{k-1}\log\|Df^\nu_{|E^s}(f^{i\nu}(x))\|,\quad
\frac{1}{\nu k}\sum_{i=0}^{k-1}\log\|Df^{-\nu}_{|E^u}(f^{-i\nu}(x))\|$$
converge when $k\to +\infty$ towards a number less than $-\rho$.
One deduces that there exists $C>0$
and a set $X$ with positive $\mu$-measure such that for each point
$x\in X$ and each $k\geq 0$ one has
$$\prod_{i=0}^{k-1}\|D f^{\nu}_{|E^s}(f^{i\nu}(x))\|\leq C.e^{-k\rho}
\text{ and } \prod_{i=0}^{k-1}\|D f^{-\nu}_{|E^u}(f^{-i\nu}(x))\|\leq C.e^{-k\rho}.$$
This implies that one can find segments of orbits
$(x,\dots,f^m(x))$ in the support of $\mu$ such that $x,f^m(x)$ belong to $X$,
the distance $d(x, f^m(x))$ is arbitrarily small and
the (non-invariant) atomic measure
$\frac 1 m \sum_{i=0}^{m-1} \delta_{f^i(x)}$ is arbitrarily close to $\mu$.
In particular for each $0\leq k\leq m$ one has
\begin{equation}\label{e.1}
\prod_{i=0}^{k-1}\|D f^{\nu}_{|E^s}(f^{i\nu}(x))\|\leq C.e^{-k\rho}
\text{ and } \prod_{i=0}^{k-1}\|D f^{-\nu}_{|E^u}(f^{m-i\nu}(x))\|\leq C.e^{-k\rho}.
\end{equation}
This property and the domination $E^s\oplus E^u$ allow to apply Liao-Gan's shadowing lemma~\cite{gan}:
the segment of orbit $(x,\dots f^m(x))$ is $\varepsilon$-shadowed by a periodic orbit
$\cO=(y,\dots, f^m(y)=y)$ where $\varepsilon$ goes to zero when $d(x, f^ m(x))$ decreases.
In particular, $\cO$ is arbitrarily close to the support of $\mu$ for the Hausdorff topology
and it supports a periodic measure arbitrarily close to the measure $\mu$.
Note that the segment of orbit $(y,\dots,f^m(y))$ satisfies an estimate like~(\ref{e.1})
(with constants $\rho',C'$ close to $\rho,C$). Consequently the orbit $\cO$
is hyperbolic and has index $\dim(E^s)$.
Repeating this argument, one obtains
a sequence of such periodic points $(y_n)$ whose orbits $(\cO_n)$ converge toward $\mu$.
By the estimate~(\ref{e.1}), the size of the
local stable and local unstable manifolds at $y_n$ is uniform.
Hence the periodic orbits $\cO_n$ are homoclinically related for $n$ large.
\end{proof}

Let us sum up the previous discussions.
\begin{corollary}\label{c.dicho-mesure}
Let $f$ be a diffeomorphism in $\diff^1(M)\setminus \overline{\Tang}$ and $\mu$ be an ergodic measure.
Two cases are possible.
\begin{itemize}
\item[--] Either $\mu$ is hyperbolic.
Its support is then contained in a homoclinic class having index equal to the dimension of the stable spaces of $\mu$;
\item[--] or $\mu$ has an unique Lyapunov exponent equal to $0$ and its support
admits a dominated splitting $E^s\oplus E^c\oplus E^u$ which coincides $\mu$-a.e.
with the Oseledets decomposition into stable, central and unstable spaces.
\end{itemize}
\end{corollary}

\subsection{Non-uniformly contracted bundles}
Let us now give a consequence of Liao's selecting lemma that improves some results in~\cite{wen-conjecture}.
It allows to obtain minimal sets that have a partially hyperbolic splitting with a one-dimensional central bundle.

\begin{theorem}~\label{t.trichotomie}
Let $f$ be a diffeomorphism in $\diff^1(M)\setminus \overline{\Tang}$ and $K_0$ be an invariant compact set
having a dominated splitting $E\oplus F$. If $E$ is not uniformly contracted, then one of the following cases occurs.
\begin{enumerate}

\item $K_0$ intersects a homoclinic class whose index is strictly less than $dim(E)$.

\item $K_0$ intersects homoclinic classes $H$ whose index is equal to $dim(E)$
and which contain weak periodic orbits: for any $\delta>0$, there exists a sequence
of hyperbolic periodic orbits $(\cO_n)$ that are homoclinically related together,
that converge for the Hausdorff topology toward a compact subset $K$ of $K_0$, whose index is equal to $\dim(E)$ and whose maximal exponent along $E$ belongs to $(-\delta,0)$.

\item There exists an invariant compact set $K\subset K_0$ which has a partially hyperbolic structure
$E^{s}\oplus E^c\oplus E^{u}$. Moreover $\dim(E^{s})< \dim(E)$,
the central bundle $E^c$ is one-dimensional and any invariant measure supported on $K$ has a Lyapunov exponent along $E^c$ equal to $0$.

\end{enumerate}
\end{theorem}
\begin{remark}
In the second case there exists a dominated splitting $E'\oplus E^c\oplus F$
on the whole homoclinic class $H$ such that $\dim(E)=\dim(E'\oplus E^c)$
and $E^c$ is one-dimensional: indeed, by definition $H$ contains a dense set of periodic points
with index equal to $\dim(E)$ and whose largest Lyapunov exponent along $E$ is close to zero.
From corollary~\ref{c.wen}, one deduces that there exists on these periodic orbits a dominated
splitting $E'\oplus E^c\oplus F$ such that $\dim(E)=\dim(E'\oplus E^c)$
and $E^c$ is one-dimensional. The dominated splitting extend to the closure, hence to the whole $H$.
\end{remark}
\begin{proof}
We first state a characterisation of not uniformly contracted bundles.
If $E_0\subset T_{K}M$ is an invariant continuous subbundle over an invariant compact set $K\subset K_0$
and $\mu$ is an invariant measure supported on $K$, then
we denote by $\lambda^+(E_0,\mu)$ (resp. $\lambda^-(E_0,\mu)$)
the maximal (resp. minimal) Lyapunov exponent of $\mu$ along $E_0$.
If $E_0$ is one-dimensional, we just write $\lambda(E_0,\mu)$.
\begin{claim}\label{cl.0}
The bundle $E_0$ is not uniformly contracted if and only if there exists an invariant ergodic measure $\mu_0$ supported
on $K$ such that $\lambda^+(E_0,\mu_0)\geq 0$.
\end{claim}
\begin{proof}
Note that $E_0$ is uniformly contracted if and only if
for any unitary vector $v\in E_0$, there exists $n\geq 1$ such that
$\|Df^n.v\|<\|v\|$.

Let us first assume that $E_0$ is not uniformly contracted:
there exists a unitary vector $v\in E_0$ such that
for all $k\geq 0$ one has
$\|Df^k.v\|\geq \|v\|$.
One deduces that there exists a probability measure $\tilde \mu_0$ on the unitary bundle associated to $E_0$
which is invariant by $Df$ and whose evaluation on
the function $v\mapsto \log\left(\|Df.v\|/\|v\|\right)$ is non-negative.
This measure projects to an $f$-invariant measure $\mu_0$ on $K$ such that
$\lambda^+(E_0,\mu_0)\geq 0$ as required.
The converse part is clear.
\end{proof}

Let us now come to the proof of the theorem.
We will assume that the first case of the theorem does not occur.
\begin{claim}\label{cl.1}
If we are not in the first case of the theorem,
there exists an invariant closed subset $K\subset K_0$
which admits a dominated splitting $E^{s}\oplus E^c\oplus F'$
such that
\begin{itemize}
\item[--] $E^{s}$ is uniformly contracted and has dimension $\dim(E^{s})<\dim(E)$,
\item[--] $E^ c$ is one-dimensional, for any invariant measure $\mu$ supported
on $K$ the exponent $\lambda(E^ c,\mu)$ is non-positive, and
there exists a unitary vector $v\in E^c$ such that
\begin{equation}\label{e.2}
\forall k\geq 0,\quad \|Df^k.v\|\geq \|v\|.
\end{equation}
\end{itemize}
\end{claim}
\begin{proof}
Since $E$ is not uniformly contracted,
one considers an invariant ergodic measure $\mu_0$ supported on $K_0$
and such that $\lambda^+(E,\mu_0)\geq 0$ as given by claim~\ref{cl.0}.
By corollary~\ref{c.dicho-mesure}, the set $K=\supp(\mu_0)$
has a dominated decomposition $E_0\oplus F_0$ such that $\dim(E_0)<\dim(E)$,
$\lambda^+(E_0,\mu_0)<0$ and $\lambda^-(F_0,\mu_0)\geq 0$.
One may choose such a measure $\mu_0$ and a set $K$ such that $\dim(E_0)$ is minimal.
Using again the claim~\ref{cl.0}, this implies that the bundle $E^{s}:=E_0$ is uniformly contracted on $K$.

If $\mu_0$ is hyperbolic, then by corollary~\ref{c.dicho-mesure}
the set $K\subset K_0$ is contained in a homoclinic class with index equal to $\dim(E^{s})$
and we are in the first case of the theorem, contradicting our assumption.
We thus have $\lambda^-(F_0,\mu_0)=0$ so that by corollary~\ref{c.dicho-mesure},
the bundle $F_0$ decomposes as a dominated splitting $E^c\oplus F'$ where $E^c$ is one-dimensional
and satisfies $\lambda(E^c,\mu_0)=0$.
In particular $E^c$ is not uniformly contracted and as in the proof of the claim~\ref{cl.0},
there exists a unitary vector $v\in E^c$ such that~(\ref{e.2}) holds, as announced.

For any other invariant measure $\mu$ supported on $K$, the exponent $\lambda(\mu,E^ c)$ should be
non-positive: otherwise $\mu$ would be hyperbolic with stable dimension smaller than $\dim(E)$
and corollary~\ref{c.dicho-mesure} would imply that we are in the first case of the theorem,
contradicting our assumption.
\end{proof}
\medskip

Let $K$ be an invariant compact set as in the previous claim.
Let us consider the family of invariant closed subsets $K'\subset K$ such that (\ref{e.2})
holds for some unitary vector $v\in E^c_{K'}$. This family is invariant under decreasing
intersections, hence one can replace $K$ by some subset $K'$ in this family which is minimal for the inclusion;
by this argument, one can assume furthermore:
\begin{description}
\item[(*)] For any invariant closed proper subset $K'\subsetneq K$,
the bundle $E^c_{K'}$ is uniformly contracted.
\end{description}
(In the case $K$ holds a minimal dynamics, property (*) is obviously empty.)
\medskip

For any ergodic measure $\mu$ supported on $K$ the exponent $\lambda(E^c,\mu)$ is non-positive.
If $\lambda(E^ c,\mu)=0$ for any $\mu$,
then we are in the third case of the theorem.
It remains to consider the case, for some invariant ergodic measure $\mu$ supported on $K$ the exponent
$\lambda(E^c,\mu)$ is negative.
Let us define $L\leq 0$  as the supremum of the negative exponents $\lambda(E^c,\mu)$
over all such measures.
\medskip

In the case $L$ is zero, for each $\delta>0$, there exists an
ergodic measure $\mu$ supported on $K$ such that $\lambda(E^c,\mu)$ belongs to $(-\delta,0)$.
If $\delta$ is small enough, by corollary~\ref{c.dicho-mesure} this measure is hyperbolic
and its stable bundle coincides with $E'=E^{s}\oplus E^c$.
Consider as in proposition~\ref{p.anosov} a sequence of periodic orbits $\cO_n$
close to $K$ whose associated invariant measures converge toward $\mu$.
Then, these orbits are homoclinically related, and their maximal Lyapunov exponent along
$E'$ (i.e. their exponent along $E^c$)
converges toward $\lambda(E^c,\mu)$, hence belong to $(-\delta,0)$.
Their index $\dim(E')$ is not greater than $\dim(E)$ by definition, is not smaller either since we are not in the first case of the theorem, hence $E=E'$ on $K$.
We have shown that we are in the second case of the theorem.
\medskip

In the case $L$ is negative,
we introduce the splitting $E'\oplus F':=(E^{s}\oplus E^c)\oplus F'$ over $K$.
Using the domination $E'=E^{s}\oplus E^c$,
one sees that there exists some integer $\nu>1$ and a positive constant $\varepsilon\ll L$ such that
$\|Df^\nu_{|E'}\|\leq e^{\varepsilon_nu}.\|Df^\nu_{|E^c}\|$.
By property (*), any invariant closed subset $K'\subset K$
supports an ergodic measure $\mu$ whose exponent $\lambda(E^c,\mu)$ is negative,
hence smaller than $L$.
With~(\ref{e.2}) above, the following properties are satisfied.
\begin{enumerate}
\item[(a)] There exists $y\in K$ such that for any $k\geq 0$,
$$\prod_{i=0}^{k-1}\|Df^\nu_{|E'}(f^{i\nu}(y))\|\geq \|Df^{k\nu}_{|E^c}(y)\|\geq 1.$$
\item[(b)] For any invariant compact set $K'\subset K$,
there exists a point $x\in K'$ such that for any $k\geq 0$,
$$\prod_{i=0}^{k-1}\|Df^\nu_{|E'}(f^{i\nu}(x))\|\leq e^{k\varepsilon}.\|Df^{k\nu}_{|E^c}(x)\|
\leq e^{-k\nu (L-\varepsilon)}.$$
\end{enumerate}
These properties allow to apply Liao's selecting lemma and its consequences
(see~\cite[proposition 3.7 and lemma 3.8]{wen-conjecture} and~\cite{wen-selecting}): for any
$\delta>0$, there exists a sequence of
periodic orbits $(\cO_n)$ which converge to a subset of $K$ for the Hausdorff topology,
that are homoclinically related and whose maximal exponent along $E'$
(i.e. their exponent along $E^c$) belongs to $(-\delta,0)$.
One concludes as in the case $L=0$: we are in the second case of the theorem.
This ends the proof of theorem~\ref{t.trichotomie}.
\end{proof}

\subsection{Minimally non-hyperbolic sets}
Theorem~\ref{t.trichotomie} implies in particular one of the main results of~\cite{wen-conjecture}.
\begin{theorem*}[Wen]
Let $f$ be a $C^1$-generic diffeomorphism in $\diff^1(M)\setminus \overline{\Tang\cup\Cycl}$.
Then, any minimally non-hyperbolic chain-transitive set $K$ is partially hyperbolic;
its central bundle is either one-dimensional or the dominated sum of two one-dimensional bundles.
\end{theorem*}
\begin{remark}
With the same proof one obtains also a result due to S. Gan, D. Yang and Wen~\cite{GYW}:
any minimally non-hyperbolic chain-transitive set $K$
of a $C^1$-generic diffeomorphism in $\diff^1(M)\setminus \overline{\Tang}$ has a dominated decomposition
$$T_KM=E^s\oplus E^c\oplus E^u$$
where $E^s,E^u$ are uniform.
Moreover in the case $\dim(E^c)>1$, the set $K$ is contained in a homoclinic
class whose set of indices contains $i_1=\dim(E^s)+1$ and $i_2=dim(E^s\oplus E^c)-1$.
From~\cite{abcdw}, all the $i\in \{i_1,\dots,i_2\}$ are also indices of the class
and consequently, $F$ decomposes as the dominated sum of one-dimensional central invariant bundles.
\end{remark}
\begin{proof}
Since $f$ is $C^1$-generic, one can assume by~\cite{BC}
that all the homoclinic classes of $f$ are either disjoint or equal
and that $K$ is either contained in a homoclinic class or disjoint from all of them.
Since $f$ is $C^1$-generic and belongs to $\diff^1(M)\setminus \overline{\Cycl}$,
\cite[section 2.4]{abcdw} implies that each homoclinic class has a unique index.
Let us assume that $K$ does not contain a non-hyperbolic set $K'$ with
a partially hyperbolic splitting and a one-dimensional central bundle:
in this case one would have $K=K'$ (by minimality property of $K$) and the conclusion of the theorem would hold.

Let us consider a dominated splitting $E\oplus F$ over $K$.
If $E$ is not uniformly contracted, then theorem~\ref{t.trichotomie}
applies. The third case of the theorem can not occurs
since $K$ does not contains a non-hyperbolic set which has a partially hyperbolic structure
and a one-dimensional central bundle.
If one of the two first cases occurs, $K$ meets a homoclinic class
with index smaller than or equal to $\dim(E)$.

Note that one can consider on $K$
the trivial dominated splitting $T_KM\oplus \{0\}$
and that the first bundle is not uniformly contracted by assumption, so the previous argument
applies, proving that $K$ is contained in a homoclinic class $H(P)$.
Since $f$ belongs to $\diff^1(M)\setminus \overline{\Tang}$,
Wen's theorem in~\cite{wen-tangence} implies that $H(P)$ (and $K$) has another dominated splitting
$E\oplus F$ such that $\dim(E)$ coincides with the index of $P$.
Since $K$ is not hyperbolic either $E$ is not uniformly contracted or $F$ is not uniformly expanded.

If $E$ is not uniformly contracted,
since the index of $H(P)$ is unique and equal to $\dim(E)$,
we are in the second case of theorem~\ref{t.trichotomie}.
In particular, the bundle $E$ splits as a sum $E'\oplus E^c_1$
where $E^c_1$ is one-dimensional.
The bundle $E'$ is uniformly contracted since
otherwise one can apply theorem~\ref{t.trichotomie} again to the
splitting $E'\oplus (E^c_1\oplus F)$, contradicting the fact that the
index of $H(P)$ is unique.

If $F$ is not uniformly expanded, the same arguments hold
and prove that $F$ splits as a dominated decomposition $E^c_2\oplus F'$
where $E^c_2$ is one-dimensional and $F'$ is uniformly expanded.
\end{proof}

\subsection{Spreading the dominated splittings}\label{ss.spread}
In the previous sections we obtained some partially hyperbolic sets $K$
with a one-dimensional central bundle.
In the case $K$ is minimal, we explain here how the dominated decomposition on $K$
extends to a larger set.

\begin{proposition}\label{p.extension}
Let $f$ be a diffeomorphism in a dense G$_\delta$ set
$\cR_{Ext}\subset\diff^1(M)\setminus \overline{\Tang}$
and let $K$ be a minimal set which has a partially hyperbolic splitting
$E^s\oplus E^c \oplus E^u$ such that $E^c$ is one-dimensional
and any invariant measure supported on $K$ has a Lyapunov exponent along $E^c$ equal to $0$.
Then, for any chain-transitive set $A$ that strictly contains $K$,
there exists a chain-transitive set $A'\subset A$
that strictly contains $K$ and there exists a dominated decomposition $E_1\oplus E^c \oplus E_3$
on $A'$ that extends the partially hyperbolic structure of $K$.
\end{proposition}

At several times in the paper we will use the following result.

\begin{lemma}\label{l.isolation}
Let $A$ be a chain-transitive set, $K\subsetneq A$ be a minimal set
and assume that there exists a neighbourhood $U$ of $K$ having the following property:
\begin{description}
\item[(I)] Any chain-transitive set
$\Lambda$ satisfying $K\subset \Lambda\subset U\cap A$ coincides with $K$.
\end{description}
Then there exists an invariant  compact set $\Delta^+\subset U\cap A$
and a point $x^+\in U\setminus \Delta^ +\cap A$ such that
\begin{itemize}
\item[--] $\Delta^+$ contains $K$ and $\omega(x^+)$;
\item[--] for any $\varepsilon>0$ and $z\in \Delta^+$, there exists a $\varepsilon$-pseudo-orbit
of $f$ in $\Delta^ +$ that joints $z$ to $K$.
\end{itemize}
There also exists a compact set $\Delta^-\subset U\cap A$
and a point $x^-\in U\setminus \Delta^ -\cap A$
such that the same properties hold if one replaces $f$ by $f^ {-1}$.

Moreover for any such pairs $(\Delta^ +,x^ +)$ and $(\Delta^ -,x^ -)$
one has $\Delta^ -\cap \Delta^ +=K$.
\end{lemma}
\begin{proof}
Let us fix a point $x_0\in A\setminus U$
and a much smaller neighbourhood $U'\subset U$ of $K$
For any $\varepsilon>0$, there exists a $\varepsilon$-pseudo-orbit
$(z_0,\dots,z_n)$ in $A$ with $z_0=x_0$ and $z_n\in K$.
Let $z_i$ be the last point of the pseudo-orbit outside $U'$.
We will only use the suborbit $(z_i,\dots,z_n)$.

We claim that there exists $\eta>0$ and for each $\varepsilon>0$
and each $\varepsilon$-pseudo-orbit $(z_i,\dots,z_n)$ there exists a point $z_j$
such that the ball $B(z_j,\eta)$ does not contain any other
point $z_k$ for $i\leq k \leq n$ and $k\neq j$.
Indeed if this fails, then any limit set $X$ for the pseudo-orbits
$(z_i,\dots,z_n)$ as $\varepsilon$ goes to $0$ would be
an invariant compact set whose induced dynamics is chain-recurrent;
as $X$ is contained in $U$ and strictly contains $K$, this
contradicts the assumption (I).

One may extract a sequence of pseudo-orbits $\{z_j,z_{j+1},\dots,z_n\}$ which
converges as $\varepsilon$ goes to zero toward a compact set $Z^+$ and such that the points $z_j$
converge toward a point $x^+\in Z^+$.
By construction $Z^+$ is compact, forward invariant, included in $U\cap A$ and contains $K$.
For every $\varepsilon>0$ and every point $z\in Z^+$,
there exists a $\varepsilon$-pseudo-orbit in $Z^+$ that joints $z$ to $K$.
By our hypothesis on $z_j$,
the point $x^+$ is non-periodic and isolated in $Z^+$, hence it is disjoint from $K$.
One deduces that the set $\Delta^+=Z^+\setminus \{f^n(x^+),n\geq 0\}$ is compact,
contains $\omega(x^+)$ and is invariant by $f$, as required.

The same construction holds if one replaces $f$ by $f^{-1}$. This gives a point $x^-$ and a set
$\Delta^-$. The set $\Delta^-\cap \Delta^+$ contains $K$, is invariant and chain-transitive;
by (I) it coincides with $K$.
\end{proof}

We recall a consequence of the connecting lemma for pseudo-orbits
(see~\cite[theorem 7]{crovisier-approximation}).
\begin{lemma}\label{l.connecting}
There exists a dense G$_\delta$ subset $\cR_{Chain}\subset\diff^1(M)$
of diffeomorphisms $f$ having the following property.
Let $X\subset M$ be a compact set and let $x,y\in X$ be two points
such that for any $\varepsilon>0$, there exists a $\varepsilon$-pseudo-orbit
in $X$ that joints $x$ to $y$; then, for any
$\varepsilon>0$, there exists a segment of orbit
$(z,\dots, f^n(z))$ contained in the $\varepsilon$-neighbourhood of $X$
such that $z$ and $f^n(z)$ are $\varepsilon$-close to $x$ and $y$ respectively.
\end{lemma}

We also recall Hayashi's connecting lemma~\cite{hayashi} (see also~\cite[theorem 5]{crovisier-approximation}).
\begin{theorem*}[Hayashi]
To any $f\in \diff^1(M)$ and to any neighbourhood $\cU$ of $f$
is associated an integer $N\geq 1$ such that any non-periodic point $x\in M$
admits two neighbourhoods $W\subset \widehat W$
with the following property: for any $p,q\in M\setminus (f(\widehat W)\cup\dots\cup f^{N-1}(\widehat W))$
such that $p$ has a forward iterate $f^{n_p}(p)$ and $q$ a backward iterate $f^{-n_q}(q)$
which both belong to $W$, there exists a perturbation
$g\in \cU$ of $f$ with support in the sets $\widehat W,f(\widehat W),\dots, f^{N-1}(\widehat W)$
such that $q$ coincides with a forward iterate $g^m(p)$ of $p$ by $g$.

Moreover $\{p,g(p)\dots, g^m(p)\}$ is contained in the union of the orbits
$\{p,f(p)\dots, f^{n_p}(p)\}$,
$\{f^{-n_q}(q),\dots,q\}$ and of the neighbourhoods $\widehat W,\dots, f^N(\widehat W)$.
Also the neighbourhoods $\widehat W, W$ can be chosen arbitrarily small.
\end{theorem*}

One can now come to the proof of the proposition.
\begin{proof}[Proof of proposition~\ref{p.extension}]
Let $\cR_{Ext}$ be the set of diffeomorphisms that belong to  $\cR_{Chain}$
and whose periodic orbits are hyperbolic.
Let us assume that $f$ belongs to $\cR_{Ext}$ and prove the proposition.
One fixes $\delta_0>0$ as in corollary~\ref{c.wen} and a small neighbourhood $\cU$ of $f$ in $\diff^1(M)$.

One can assume that for some arbitrarily small neighbourhood $U$ of $K$
the isolation assumption (I) is satisfied since otherwise
the conclusion of the proposition holds obviously.
If $U$ has been chosen small enough,
then for any invariant measure in $U$ the exponent along
$E^c$ belongs to $(-\delta_0,\delta_0)$.

Let $x^+,x^-$ be two points in $A$ and $\Delta^+,\Delta^-$ be two subsets of $A$ as given
by lemma~\ref{l.isolation}.
One applies the connecting lemma and finds an integer $N\geq 1$
and some neighbourhoods $W^+\subset \widehat W^+$,
$W^-\subset \widehat W^-$ at $x^+$ and $x^-$ respectively.
If the neighbourhoods are chosen small enough, the following is satisfied.
\begin{itemize}
\item[--] All the iterates $f^i(\widehat W^+)$, $f^{-j}(\widehat W^-)$
for $0\leq i,j\leq N$ are pairwise disjoint and have closures disjoint from the set $\Delta^+\cup \Delta^-$.
\item[--] The iterates $f^{i}(\widehat W^+)$ for $0\leq i\leq N$ are disjoint from the
backward orbit $\{f^{-k}(x^-),k\geq 0\}$.
\end{itemize}
By lemma~\ref{l.connecting} there exists a segment of orbit
$(x_0,\dots,f^\ell(x_0))$ contained in a small neighbourhood of $A$
such that $x_0\in W^-$ and $f^\ell(x_0)\in W^+$.

By the connecting lemma there exists a perturbation
$h_0 \in \cU$ of $f$ supported in the $f^k(\widehat W^+)$
for $0\leq k \leq N$ such that the orbit of $x_0$ contains the forward orbit $\{f^{i}(x^+),i\geq m\}$
for some $m$ large.
Note that the sets $\Delta^+\cup \Delta^-$ and the backward orbit of $x^-$
have not been modified by the perturbation.
One can apply the connecting lemma a second time (using the fact that the 
$f^i(\widehat W^-)$, $f^j(\widehat W^-)$ for $0\leq i,j\leq N$ are pairwise disjoint)
and find a perturbation $h_1$ of $h_0$ supported in the
$f^{-k}(\widehat W^-)$ for $0\leq k \leq N$ such that the backward orbit $\{f^{-i}(x^-), i\geq m^-\}$ and the
forward orbit $\{f^{i}(x^+), i\geq m^+\}$, for some $m^-,m^+$ large, are connected.
Moreover $h_1$ still belongs to $\cU$.
(This way to compose perturbations with disjoint supports
is detailed in~\cite[section 2.1]{crovisier-approximation}
and is related to the lift axiom of~\cite[section 2]{pugh-robinson}.)
We denote by $x$ a point of this ``heteroclinic" orbit for $h_1$ from $\Delta^-$ to $\Delta^+$.

Since the set $\Delta=\Delta^+\cup \Delta^-$ has a dominated splitting $E^s\oplus E^c\oplus E^u$,
and since the orbit of $x$ accumulates on this set in the future and in the past,
there exist at $x$ some (unique) tangent subspaces $E^{u}_x\subset E^{cu}_x$ such that
$D(h_1)^{-n}.E^u_x$ and $D(h_1)^{-n}.E^{cu}_x$ accumulate as $n$ goes to $\infty$
towards the bundles $E^u$ and $E^c\oplus E^u$ respectively.
Similarly there exists at $h_1(x)$ two subspaces $E^{s}_{h_1(x)}\subset E^{cs}_{h_1(x)}$ such that
$D(h_1)^{n}.E^s_{h_1(x)}$ and $D(h_1)^{n}.E^{cs}_{h_1(x)}$ accumulate as $n$ goes to $\infty$
towards the bundles $E^s$ and $E^s\oplus E^c$ respectively.
By a $C^1$-small perturbation $h$ of $h_1$ at $x$, one can ensure that $Dh.E^{cu}_x$ is transverse to
$E^s_{h(x)}$ and that $Dh.E^{u}_x$ is transverse to $E^{cs}_{h(x)}$.
Hence the dominated splitting on $\Delta$ extends to the orbit of $x$.

We now realise a $C^1$-small perturbation $g$ of $h$ which possesses a periodic orbit $\cO$
contained in an arbitrarily small neighbourhood of $\Delta\cup \{h^n(x),n\in\ZZ\}$ and
has a point close to $x$.
By construction, there exists a large forward iterate $h^L(x)$
and a large backward iterate $h^{-L}(x)$ which coincide with some large
iterates $f^{n^+}(x^+)$ and $f^{-n^-}(x^-)$.
By lemma~\ref{l.connecting}, there exists
an orbit $(\zeta,f(\zeta),\dots, f^r(\zeta))$ contained in
an arbitrarily small neighbourhood of
$$\{f^n(x^+), n\geq n^+\}\cup \Delta\cup \{f^{-n}(x^-), n\geq n^-\}$$
such that $\zeta$ is close to $f^{n^+}(x^+)$
and $f^r(\zeta)$ is close to $f^{-n^-}(x^-)$.
One deduces that $h^{-L}(\zeta)$
and $h^{r+L}(\zeta)$ are close to $x$.
By the connecting lemma applied at $x$, one can close as required the orbit of $\zeta$
by a perturbation $g$.
Note that the period $r+2L$ of the obtained orbit $\cO$ is arbitrarily large
if $\zeta$ is chosen close enough to $f^{n^+}(x^+)$,
and that $r$ iterates are contained in a small neighbourhood
of $\Delta$. This implies that the $(\dim(E^s)+1)$-th exponent of $\cO$
belongs to $(-\delta_0,\delta_0)$.

We have shown that there exists a sequence of diffeomorphisms $(g_n)$ which converge to $f$,
and a sequence of periodic orbits $(\cO_n)$ which converge toward an invariant compact subset $A'$ of $A$
which strictly contains $K$ (it contains $x^-$ and $x^+$).
Moreover each orbit $\cO_n$ has a $(\dim(E^s)+1)$-th Lyapunov exponent
in $(-\delta_0,\delta_0)$,
hence by corollary~\ref{c.wen}, $\cO_n$
holds a dominated splitting $E_3\oplus E_2\oplus E_3$ with $\dim(E_2)=1$ and $\dim(E_1)=\dim(E^s)$.
Since the splitting is uniform, it can be passed to the limit set $A'$.
By uniqueness of the dominated splittings, $E^s\oplus E^c\oplus E^u$ and $E_1\oplus E_2\oplus E_3$
coincide on $K$.
\end{proof}

\begin{remark} Under the setting of proposition~\ref{p.extension},
an additional property is satisfied (that we will not use later):

\noindent
\emph{ The set $A'$ is the Hausdorff limit of periodic orbits of index $\dim(E^s)$
whose Lyapunov exponent along $E^c$ is arbitrarily close to zero.}

The proof uses the results proved in sections~\ref{s.model} and~\ref{s.HPR}.
Let us give a sketch.
in the case the isolation property (I) is satisfied, this is a consequence of
the proof of proposition~\ref{p.extension}.
If the isolation property (I) fails, the set $A'$ may be chosen in an arbitrarily small neighbourhood
of $A$, hence it is partially hyperbolic; moreover
by~\cite{crovisier-approximation}, $A'$ is the Hausdorff limit of a sequence of periodic orbits $(\cO'_n)$.
If all the measures supported on $A'$ have a Lyapunov exponent
along $E^c$ equal to zero, then the central Lyapunov exponents of the orbits $\cO'_n$ are close to zero
and the additional property holds again.
In the remaining case, the discussion of section~\ref{s.model} will imply that the dynamics
in the central direction is either chain-transitive or chain-hyperbolic: in both cases
by the results of section~\ref{s.HPR}, $A'$ is
limit of periodic orbits $(\cO'_n)$ that are homoclinically related to each other
in a small neighbourhood of $A'$.
One may consider smaller sets $A''\subset A'$ and the same argument implies
that $A''$ is the limit of periodic orbits $(\cO''_n)$ that are homoclinically related to the $\cO'_n$
in a small neighbourhood of $A'$. Note that the central exponents of the periodic orbits
$\cO''_n$ are arbitrarily close to zero if $A''$ is close enough to $K$.
Since $f$ is $C^1$-generic, the periodic $\cO_n'$ and $\cO_n''$ are homoclinically related.
By the transition property~\cite{BDP}, there exists periodic orbits close to $A'$ for the Hausdorff topology
and whose central exponent is close to zero. This is the required property.
\end{remark}


\section{Central models}\label{s.model}
Let $f$ be any diffeomorphism in $\diff^1(M)$ and
$K$ be a chain-transitive compact set which
has a dominated splitting $T_K M=E_1\oplus E^c \oplus E_3$ where $E^c$ is a one-dimensional subbundle.
In this section we analyse in a neighbourhood of $K$ the dynamics that is ``tangent to $E^c$"
and study the expansivity property of $K$.
The description is topological and gives some information even when the Lyapunov exponents of $K$
along $E^c$ are all zero. In the two first sections
we recall definitions and results from~\cite{crovisier-palis-faible}.

\subsection{Construction and definition}
The plaque family theorem~\cite[theorem~5.5]{hirsch-pugh-shub} associates to any invariant compact set $K$, which has a
dominated decomposition $E_{1}\oplus E^c\oplus E_{3}$, a \emph{plaque family tangent to $E^c$}
(not unique in general).
This is a continuous map $\cD^c$ from the linear bundle $E^c$ (over $K$) into $M$ satisfying:
\begin{itemize}
\item[--] for each $x\in K$, the induced map $\cD^c_x\colon E^c_{x}\to M$ is a $C^1$-embedding
that is tangent to $E^c_{x}$ at the point $\cD^c_x(0)=x$;
\item[--] $(\cD^c_x)_{x\in K}$ is a continuous family of $C^1$-embeddings of $E^c_x$ into $M$;
\item[--] the plaque family is \emph{locally invariant}, i.e. there exists a neighbourhood $U$ of the section $0$
in $E^c$ such that for each $x\in K$, the image of $\cD^c_x(E^c_x\cap U)$ by $f$ is contained in $\cD^c_{f(x)}(E_{f(x)})$.
\end{itemize}
The last property allows to lift the dynamics of $f$ as a fibred dynamics $\hat f$ on the bundle $E^c$
that is locally defined in a neighbourhood of the (invariant) section $0$.
\smallskip

In our setting $K$ is chain-transitive and $E^c$ is one-dimensional. So, two cases occur.
\begin{description}
\item[-- \emph{The orientable case}.] There exists a continuous orientation of the bundle $E^c$
that is preserved by $Df$. Taking the opposite orientation, one gets exactly two preserved orientations:
equivalently, the unitary bundle $UE^c$ decomposes as the union of two disjoint chain-transitive sets.
The bundle $E^c$ may be trivialised as $K\times \RR$ in such a way that the lifted dynamics $\hat f$ is fibrewise orientation preserving.
One can thus study separately the dynamics of $\hat f$ on $K\times [0,+\infty)$ and $K\times (-\infty,0]$.
In the following, we consider these dynamics as two fibred systems $\hat f^+$ and $\hat f^-$ on
$\hat K\times [0,+\infty)$ where $\hat K=K$.

\item[-- \emph{The non-orientable case}.] There does not exist any continuous orientation of the bundle $E^c$
that is preserved by $Df$. Equivalently, the dynamics of $Df$ on $UE^c$ is chain-transitive.
Writing any element $v\in E^c$ as a product $t.\bar v$ with $(\bar v,t)\in UE^c\times [0,+\infty)$, the dynamics of
$\hat f$ may be considered on the space $\hat K\times [0,+\infty)$ where $\hat K=UE^c$.
\end{description}
The fibred systems $\hat f^-,\hat f^+$ or $\hat f$ on $\hat K\times [0,+\infty)$ are \emph{central models}
associated to the dynamics of $f$ along the central plaques $\cD^c_{x}$ of $K$.
\bigskip

More generally, in \cite{crovisier-palis-faible} we defined (abstract) central models $(\hat K,\hat f)$
where the basis $\hat K$ is a compact metric space
and $\hat f$ is a continuous map $\hat K\times [0,1]\to \hat K\times [0,+\infty)$,
such that
\begin{itemize}
\item[--] $\hat f(\hat K\times \{0\})=\hat K\times \{0\}$,
\item[--] $\hat f$ is a local homeomorphism in a neighbourhood of $\hat K\times \{0\}$,
\item[--] $\hat f$ is a skew product and takes the form
$ \hat f (x,t)= (\hat f_1(x),\hat f_2(x,t))$.
\end{itemize}

\subsection{Classification}

\begin{definition}
Let $(\hat K,\hat f)$ be a central model.
\smallskip

\noindent
A \emph{chain-recurrent central segment} is a (non-trivial) segment $I=\{x\}\times [0,a]$
contained in a chain-transitive $\hat f$-invariant compact subset of $\hat K\times [0,1]$.
\smallskip

\noindent
An open neighbourhood $S$ of $\hat K\times \{0\}$ is a \emph{trapping strip}
if, for any point $x\in \hat K$, the intersection $S\cap \left(\{x\}\times [0,+\infty)\right)$ is an interval
and if $\hat f(\overline{S})\subset S$.
\smallskip

\noindent
The \emph{chain-stable set} of $\hat K\times \{0\}$ is
the set of points $z\in \hat K\times [0,1]$ such that for each $\varepsilon>0$,
there is a $\varepsilon$-pseudo-orbit
$(z_0,\dots,z_n)$ in $\hat K\times [0,1]$ with $z_0=z$ and $z_n\in \hat K\times \{0\}$.
(One defines symmetrically the chain-unstable set.)
\end{definition}

The following result restates~\cite[section 2]{crovisier-palis-faible}.
\begin{proposition}\label{p.conley}
Let $(\hat K,\hat f)$ be a central model whose basis is chain-transitive.
Then, four disjoint cases are possible:
\begin{itemize}
\item[--] The chain-stable and the chain unstable sets are both non-trivial.
Equivalently, there exists a chain-recurrent central segment.
\item[--] The chain-stable and the chain-unstable sets of $\hat K\times \{0\}$ are trivial:
they are reduced to $\hat K\times \{0\}$.
Equivalently, there exist arbitrarily small neighbourhoods of $\hat K\times \{0\}$ that are trapping strips for $\hat f$ and
arbitrarily small neighbourhoods that are trapping strips for $\hat f^{-1}$.
\item[--] The chain-stable set contains a neighbourhood of $\hat K\times \{0\}$ and the chain-unstable set
is trivial. In particular, there exist arbitrarily small neighbourhoods of $\hat K\times \{0\}$
that are trapping strips for $\hat f$.
\item[--] The chain-stable set is trivial and the chain-unstable set contains a neighbourhood of $\hat K\times \{0\}$:
this case is symmetrical to the former one.
\end{itemize}
\end{proposition}
\begin{proof}
Let us first recall the results in~\cite[section 2.5]{crovisier-palis-faible}.
The chain-unstable set is trivial if and only if there exist arbitrarily small neighbourhoods of
$\hat K\times \{0\}$ that are trapping strips for $\hat f$.
Moreover if there is no chain-recurrent central segment then the chain-stable set
is either trivial or contains a neighbourhood of $\hat K\times \{0\}$, and the same holds
for the chain-unstable set.

From this one easily gets the equivalences in the two first cases of the proposition.
If we are not in one of these two cases, no chain-recurrent central segment exists,
either the chain-stable or the chain-unstable set of $\hat K\times \{0\}$ is trivial,
and the other set contains a neighbourhood of $\hat K\times \{0\}$.
As a consequence we are in one of the last two cases of the proposition.
\end{proof}
\bigskip

Let us come back to the description of the dynamics along
the central plaques $\cD^c_{x}$ for an invariant compact set $K$ as considered in the introduction of section~\ref{s.model}.
Using proposition~\ref{p.conley}, one gets the four following (non-disjoint) types
(see also figure~\ref{f.type}).
\begin{description}
\item[\bf{-- Type (R)}, chain-recurrent.]
For any cone $\cC$ around $E^c$ and
for any $\varepsilon>0$, there exists a point $x\in K$ and
a non-trivial $C^1$-curve $\gamma$ containing $x$ and  contained in the chain-recurrence class of $K$,
such that for any $n\in \ZZ$, the iterate $f^n(\gamma)$ remains tangent to $\cC$
and has a length bounded by $\varepsilon$.
The curve $\gamma$ is called a \emph{central segment} of $K$.

\item[\bf{-- Type (N)}, chain-neutral.] There exist arbitrarily small neighbourhoods of the section $0$ in $E^c$
that are attracting for $\hat f$, and arbitrarily small neighborhoods that are attracting for $\hat f^{-1}$.

\item[\bf{-- Type (H)}, chain-hyperbolic.]
There exists arbitrarily an small closed neighbourhood of $K\times \{0\}$ in $E^c$ that is
mapped into its interior by $\hat f$ (in the \emph{attracting case})
or by $\hat f$ (in the \emph{repelling case}) and whose image by $\cD^c$ is contained in the chain-stable or
in the chain-unstable set of $K$ respectively.

\item[\bf{-- Type (P)}, chain-parabolic.]
We are in the orientable case. Different situations occur.
\begin{itemize}
\item[--] {\bf Type (P$_{\mathbf{SU}}$)}.
For one central model, $K\times [0,+\infty)$, there exists arbitrarily small neighbourhoods of $K\times \{0\}$
that are trapping strips for $\hat f^+$ and which projects by $\cD^c$ in the chain-stable set of $K$;
for the other one, $K\times (-\infty,0]$, there exists arbitrarily small neighbourhoods of $K\times \{0\}$
that are trapping strips for $(\hat f^-)^{-1}$ and which projects by $\cD^c$ in the chain-unstable set of $K$.
\item[--] {\bf Types (P$_\mathbf{SN}$)} and {\bf (P$_\mathbf{UN}$)} respectively.
For one central model, $K\times [0,+\infty)$, there exists arbitrarily small neighbourhoods of $K\times \{0\}$
that are trapping strips for $\hat f^+$ (resp. for $(\hat f^+)^{-1}$)
and which projects by $\cD^c$ in the chain-stable set of $K$ (resp. in the chain-unstable set of $K$);
for the other one, $K\times (-\infty,0]$, there exists arbitrarily small neighbourhoods of $K\times \{0\}$
that are trapping strips for $\hat f^-$
and there exists arbitrarily small neighbourhoods of $\hat K\times \{0\}$
that are trapping strips for $(\hat f^-)^{-1}$.
\end{itemize}
\end{description}
\medskip

\begin{remarks}\label{r.classification}
\begin{itemize}
\item[a)] For types (R), (N), (H), one considers the chain-stable and chain-unstable sets
and the chain-recurrence classes for the dynamics of $f$ in $M$ and not only in the central models.
Since the chain-recurrence class of $K$ can be much larger than $K$,
the global dynamics of $f$ is used in this definition.
In particular, the four types are not disjoint;
proposition~\ref{p.conley} however says that they cover all the cases.
One could have been more restrictive in the definition, considering only pseudo-orbits in the central models:
one would have obtained four disjoint types.
But the definition we have chosen here will be more adapted to our needs.
 
\item[b)] If $K$ has type (N), this also holds on subsets $K'\subset K$.
If $K$ has type (R), this also holds on larger sets $K'\supset K$.

\item[c)] If $K$ has type (N) or (P), all the Lyapunov exponents along the bundle $E^c$
for the invariant measures supported on $K$ have to be equal to zero.
In particular there is no hyperbolic periodic orbit in $K$, the bundle $E_1$ is uniformly contracted
and the bundle $E_3$ is uniformly expanded.
Any periodic orbit contained in a small neighbourhood of $K$ has also a Lyapunov exponent along $E^c$ close to $0$.

If $K$ has type (H) attracting (resp. repelling), all the Lyapunov exponents along the bundle $E^c$
have to be non-positive (resp. non-negative).
In particular the bundle $E_1$ (resp. $E_3$) is uniformly contracted (resp. expanded).

\item[d)] ``Having type (N), (H)-attracting or (P$_{SN}$)" is equivalent to
the existence of arbitrarily small neighbourhoods of the section $0$
in $E^c$ that are trapping strips for $\hat f$.
Since $E_1$ is uniformly contracted in these cases,
one can choose a plaque family $\cD^{cs}$ tangent to $E_1\oplus E^c$
over $K$ which satisfies the following \emph{trapping property}:
$$\forall x\in K,\quad f(\overline{\cD^{cs}_x})\subset \cD^{cs}_{f(x)}.$$

\end{itemize}
\end{remarks}
\medskip

\subsection{The type is well defined}
We now prove that the type of a set is well-defined and does not depend on the choice of a plaque family $\cD^c$.
This is clear for type (R) since the definition does not involve the choice of a central model.
For the other types, this is based on the following classical lemma.
\begin{lemma}\label{l.uniq-cs}
Let $f$ be a diffeomorphism and $K$ be an invariant compact set
which has a dominated splitting $E_1\oplus E_2$.
Then there exists $r>0$ such that any plaque family $\cD$
tangent to $E_1$ whose plaques have a diameter bounded by $r$
and satisfy the \emph{trapping property}
$$\forall x\in K,\quad f(\overline{\cD_x})\subset \cD_{f(x)},$$
also satisfies the following:
\begin{description}
\item[-- \emph{Uniqueness.}]
If $\cD'$ is another plaque family tangent to $E_1$ then
there exists a neighbourhood $U$ of the section $0$ in $E_1$
such that $\cD_x(U_x)\subset \cD'_x$ for each $x\in K$.
\item[-- \emph{Coherence.}]
There exists $\varepsilon>0$ such that for any points $x,x'\in K$
that are $\varepsilon$-close with $x'\in \cD_x$, one also has
$f(\overline{\cD_{x}})\subset \cD_{f(x')}$.
\end{description}
\end{lemma}
\begin{proof}
Since $K$ has a dominated splitting $E_1\oplus E_2$, 
there exists a neighbourhood $V$ of $K$ where a locally forward invariant cone field $\cC$ around the direction
$E_2$ is defined.
One chooses $r>0$ small so that the plaques of $\cD$ over $K$ are contained in $V$.

We prove the uniqueness.
There exists a neighbourhood $U$ of the section $0$ in $E_{1}$ such that
for each $x\in K$ and for each $z\in \cD_{x}(U_x)$ there exists a curve $\gamma$
tangent to $\cC$ which connects $z$ to a point of
$\cD'_x$. By forward iterations $f^n(\gamma)$ is still contained in $V$,
tangent to the cone $\cC$ and connects the plaques $\cD_{f^n(x)}$ and $\cD'_{f^n(x)}$.
Let $\sigma\subset \cD_x$ and $\sigma'\subset \cD'_x$ be two curves
that connect $x$ to the endpoints of $\gamma$.
If th elength of $\gamma$ is non-zero,
due to the domination, for $n$ large the lengths of $f^n(\sigma)$ and $f^n(\sigma')$
decrease exponentially faster than the length of $f^n(\gamma)$, which contradicts
the triangular inequality.
One thus deduces that $\cD_x(U_x)$ is contained in $\cD'_x$ as announced.

The argument for proving the coherence is similar:
$f(\overline{\cD_{f^{-1}(x)}})$ is a compact neighbourhood of $x$ in $\cD_x$.
If $x,x'$ are close enough, then for any point $z\in f(\overline{\cD_{f^{-1}(x)}})$
there exists a curve $\gamma$ tangent to $\cC$
which connects $z$ to a point of $\cD_{x'}$.
As before one deduces $f(\overline{\cD_{f^{-1}(x)}})\subset \cD_{x'}$.
By the trapping property one has $f(x')\in \cD_{f(x)}$, so the same proof gives also
$f(\overline{\cD_{x}})\subset \cD_{f(x')}$.
\end{proof}

\begin{lemma}\label{l.uniq-NH}
Let $f$ be a diffeomorphism, $K$ be a chain-transitive set
which has a dominated splitting $E_1\oplus E^c\oplus E_3$ such that $E^c$ is one-dimensional
and $\cD^c$ be a plaque family tangent to $E^c$ having type (N) or (H)-attracting.
Let $\cD^{cs}$ be any plaque family tangent to $E_1\oplus E^c$.
Then,
\begin{itemize}
\item[--] there exists a neighbourhood $U$ of the section $0$ in $E^{c}$ such that
for each $x\in K$ one has $\cD^c_x(U_x)\subset \cD^{cs}_x$;
\item[--] any other plaque family ${\cD^c}'$ tangent to $E^c$ has also type (N) or (H)-attracting.
\end{itemize}
\end{lemma}
\begin{proof}
Reducing the plaques $\cD^{c}$, one may assume that they satisfy the trapping property
stated in lemma~\ref{l.uniq-cs} (see remark~\ref{r.classification}.d).
Arguing as in the previous lemma one gets the first item.

Since $E^c$ has type (N) or (H)-attracting, the bundle $E_1$ is uniformly
attracting and the plaques $\cD^{cs}$ are foliated by strong stable manifolds
tangent to $ E_1$. Since $\cD^c$ satisfies the trapping property and is contained in $\cD^{cs}$,
one deduces that $\cD^{cs}$ can be chosen to satisfy also the trapping property.

If $\cD^c,{\cD^c}'$ are two plaques families tangent to $E^c$,
then by the first item one can assume that they are contained in the plaques of $\cD^{cs}$.
The central dynamics along $\cD^c$ and ${\cD^c}'$ are conjugated by the strong stable holonomies:
this proves that if a central model associated to $\cD^c$ has a small neighbourhood 
which is a trapping strip for $\hat f$ or $\hat f^{-1}$,
then the same holds for a central model associated to ${\cD^c}'$.
Hence if $\cD^c$ has type (N), ${\cD^c}'$ has type (N) also.
If $\cD^c$ has type (H)-attracting, then the central models associated to ${\cD^c}'$ have
arbitrarily small trapping strips and the plaques $\cD^{cs}$ are contained in the chain-stable set of $K$;
consequently ${\cD^c}'$ has type (H)-attracting also.
\end{proof}

The same argument holds for the type (P). Hence we obtain:

\begin{corollary}
Let $f$ be a diffeomorphism, $K$ a chain-transitive set
which has a dominated splitting $E_1\oplus E^c\oplus E_3$ such that $E^c$ is one-dimensional, and
$\cD^c,{\cD^c}'$ two plaque families tangent to $E^c$.
Then, the dynamics along the plaques of ${\cD^c}$ and along those of ${\cD^c}'$
have the same types.
\end{corollary}
\begin{proof}
We first note that $E_1=E^s$ is uniformly contracted.
Let $\hat f^+$ be one of the central models associated to the plaque family $\cD^c$
and assume for instance that there exists trapping strips
that are arbitrarily small neighborhoods of $\hat K\times \{0\}$.
Let us consider a plaque family $\cD^{cs}$ tangent to the bundle $E^s\oplus E^c$.
Each set $\cD^{cs}_x\setminus W^{ss}(x)$,
$\cD^{c}_x\setminus W^{ss}(x)$, $\cD^{c'}_x\setminus W^{ss}(x)$ has two components:
we denote by $\cD^{cs,+}_x$, $\cD^{c,+}_x$ and
$\cD^{c'+}_x$ the component tangent to the half plaques of the central model $\hat f^+$.
Arguing as before, one proves that $\cD^{c,+}_x,\cD^{c',+}_x$ are contained in $\cD^{cs,+}$. 
The strong stable lamination on the plaques $\cD^{cs}$
conjugates the dynamics along $\cD^{c,+}$ and $\cD^{c',+}$, so that
the central model along the plaque family $\cD^{c',+}$ also admits trapping strips
that are arbitrarily small neighborhoods of $\hat K\times \{0\}$.
The end of the proof follows.
\end{proof}

\subsection{Extremal bundles}
We now discuss the possibility for one of the bundles $E_1$ or $E_3$ to be degenerated.

\begin{proposition}\label{p.extremal}
Let $f$ be a diffeomorphism and $K$ be a chain-transitive set
which has a dominated splitting $E_1\oplus E^c$ such that $E^c$ is one-dimensional.
If $K$ has type (N), (P) or (H)-attracting, then $K$ contains a periodic orbit.
\end{proposition}
\begin{proof}
One can assume that $K$ is minimal. We know that the bundle
$E^s=E_1$ is uniformly contracted.
We will deal with the orientable case (the non-orientable case is similar and simpler).
By assumption there exists a central model for $K$ which has arbitrarily small trapping strips.
By projecting this central model,
one finds a family of half open plaques $\cD^{cs,+}_x$, $x\in K$ tangent to $E^s\oplus E^c$
bounded by $W^{ss}_{loc}(x)$ and such that
$$f\left(\overline{\cD^{cs,+}_x}\right)\subset \cD^{cs,+}_{f(x)}\cup W^{ss}_{loc}(f(x)).$$
Each plaque $\cD^{cs,+}_x$ is a component of $\cD^{cs}_x\setminus W^{ss}_{loc}(x)$
and hence is a uniform half neighbourhood of $x$.
Consequently if $f^n(x)$ is a return close to $x$, either
$f^n(x)\in W^{ss}_{loc}(x)$,
or $x\in \cD^{cs,+}_{f^n(x)}$,
or $f^n(x)\in \cD^{cs,+}_{x}$.
In the first case $(f^n(x))_{n\geq 0}$ converges to a periodic orbit.
In the second case, since $K$ is minimal, $x'=f^n(x)$ has a return $f^{n'}(x')$ arbitrarily close to $x$,
hence which satisfies $f^{n'}(x')\in \cD^{cs,+}_{x'}$; we are thus reduced to the third case.
In the third case, the set $\cD^{cs,+}_x\cup W^{ss}_{loc}(x)$
is forward invariant by $f^n$ and foliated by strong stable one-codimensional manifolds.
One concludes that there exists a closed segment $\sigma\subset \cD^{cs,+}_x\cup W^{ss}_{loc}(x)$
tangent to the central bundle, which is $n$-periodic and such that
the orbit of any point in $\overline{\cD^{cs,+}_x}$ converge toward a periodic point in $\sigma$.
In all the cases, the orbit of $x$ accumulates on a periodic point of $K$.
\end{proof}

\begin{remark}
\label{r.extremal}
The same result in case (R) does not hold: there exist a diffeomorphism
which possesses a fixed circle $C$
which is normally attracting and whose induced dynamics is conjugated to an irrational
rotation; the minimal set in $C$ has type (R) but does not contain any periodic orbit.
The dynamics in case (R) is covered by~\cite[section 3]{pujals-sambarino3}:
in particular if $f$ is $C^2$, then $K$ contains either a periodic orbit or 
a periodic circle with the dynamics of an irrational rotation.
\end{remark}

As a consequence of proposition~\ref{p.extremal},
if $f$ is a Kupka-Smale diffeomorphism, then a chain-transitive set which has a splitting
$E_1\oplus E^c$ cannot have type (N) or (P).
It will follows from results the next sections (see proposition~\ref{p.R-extremal})
that if $f$ belongs to a dense G$_\delta$ subset of $\diff^1(M)$,
the set can not have type (R) either. Hence, it has type (H);
if moreover it has type (H)-attracting, this is a sink.

\subsection{A weak shadowing lemma}

In this section, $f$ is a diffeomorphism and $K$ is a chain-transitive set
which has a partially hyperbolic splitting $E^s\oplus E^c\oplus E^u$
such that $E^c$ is one-dimensional. One fixes a small neighbourhood $U_0$ of $K$
such that the maximal invariant set
$\widehat K$ in $\overline{U_0}$ still has the partially hyperbolic structure of $K$.
One also fixes a plaque family $\cD^{cs}$ tangent to $E^s\oplus E^c$ over $\widehat K$.

Let $0<\delta<r$.
An orbit $(f^n(x))_{n\geq 0}$ in $\widehat K$ is said to
\emph{$r$-shadows} a sequence $(z_n)_{n\geq 0}$ in $U_0$ if
for each $n\geq 0$, the point $z_n$ is at distance $<r$ from $f^n(x)$.
It is said to \emph{$(r,\delta)$-cs-shadows} a sequence $(z_n)_{n\geq 0}$ in $U_0$ if
it $r$-shadows $(z_n)_{n\in \ZZ}$ and for each $n\geq 0$, the point $z_n$ is at distance
$<\delta$ from $\cD^{cs}_{f^n(x)}$. The following result generalises the classical shadowing lemma.
(See also~\cite{gan} for another generalisation.)

\begin{lemma}\label{l.shadows}
If $E^c$ has type (N), (H)-attracting or (P$_{SN}$),
then for any $r>0$ there exists a neighbourhood $U_1\subset U_0$ of $K$
such that for any $\delta>0$ there exists $\varepsilon>0$ with the following properties.
\begin{itemize}
\item[--]
Any $\varepsilon$-pseudo-orbit $(z_n)_{n\geq 0}$ in $U_1$ is $(r,\delta)$-cs-shadowed
by an orbit $(f^n(x))_{n\geq 0}$ in $\widehat K$.
Moreover if $(z_n)_{n\geq 0}$ is $\tau$-periodic, one can choose $(f^n(x))_{n\geq 0}$ to be $\tau$-periodic.
\item[--]
If $(f^n(x))_{n\geq 0}$ and $(f^n(x'))_{n\geq 0}$ both $r$-shadow $(z_n)_{n\geq 0}$,
then $x'$ belongs to $\cD^{cs}_x$.
\end{itemize}
\end{lemma}

\begin{remarks}
If $K$ has type (N), the central plaques $\cD^c$ are uniquely defined and
coincides with the intersections $\cD^{cs}\cap \cD^{cu}$ (by lemma~\ref{l.uniq-cs}).
If $(z_n)_{n\in \ZZ}$ is a $\varepsilon$-pseudo-orbit in $U_1$,
one deduces that there exists an orbit $(f^n(x))_{n\in \ZZ}$
which $(r,\delta)$-c-shadows $(z_n)_{n\in \ZZ}$:
for each $n$, the point $z_n$ is at distance $<\delta$ from $\cD^{c}_{f^n(x)}$.
\end{remarks}
\begin{proof}
The proof is similar to the geometrical proof of the classical shadowing lemma,
but the uniform contraction along the stable direction has to be replaced by the topological trapping property
stated in lemma~\ref{l.uniq-cs}.

By remark~\ref{r.classification}.d), one may consider a plaque families $\widetilde \cD^{cs}$
whose plaques have a diameter much smaller than $r$ and whose restriction to the set $K$
has the trapping property of lemma~\ref{l.uniq-cs}.
This trapping property extends to the maximal invariant set in a neighbourhood $U_1$ of $K$.
By lemma~\ref{l.uniq-NH} the plaques $\widetilde \cD^{cs}$ are contained in the plaques
$\cD^{cs}$.

Let us fix $\delta_0>0$ small.
For each point $x\in \widehat K$, one can define a tubular neighbourhood $T_x$ of $\widetilde \cD^{cs}_x$
with width $\delta_0$ in the direction $E^u$.
The trapping property and the uniform expansion along $E^u$
implies that $f(T_x)\cap T_{f(x)}$ is a tubular neighbourhood
$f(\widetilde \cD^{cs}_x)$ having the same width. Hence it is contained in $T_{f(x)}$.
By continuity there exists $\varepsilon_0>0$ such that
if $f(x)$ and $x'$ are $\varepsilon_0$-close then
$f(T_x)\cap T_{x'}$ is a tubular neighbourhood of a subset of $\widetilde \cD^{cs}_{x'}$
with width $\delta_0$.
By induction if $(z_n)_{n\geq 0}$ is a $\varepsilon_0$-pseudo-orbit
then for any $n\geq 0$ the intersection of the $f^{n-k}(T_{z_k})$
for $0\leq k \leq n$ is a (non-empty) tubular neighbourhood of a subset of $\widetilde \cD^{cs}_{z_n}$.

One may reduce the set $U_1$ so that any $\varepsilon_0$-pseudo-orbit
$(z_n)_{n\geq 0}$ in $U_1$ can be extended in the past.
With this property, one can define the intersection $A_n$ of the $f^{-k}(T_{z_k})$
for $|k| \leq n$: the previous argument shows that this is non-empty.
One deduces that the intersection of the $A_n$ is non-empty: it contains points $x$
in $\widehat K$ such that each iterate $f^n(x)$ is $\delta_0$-close to a point
in $\widetilde \cD^{cs}_{z_n}$. In particular the points $f^n(x)$
and $z_n$ are at distance less than $\diam(\widetilde \cD^{cs}_{z_n})+\delta_{0}<r$.

Let us assume that the sequence $(z_n)_{n\geq 0}$ is a $\varepsilon$-pseudo-orbit
for $\varepsilon$ smaller. We have to show that it is $(r,\delta)$-cs-shadowed
by the half orbits $(f^n(x))_{n\geq 0}$ that we have built, where the constant $\delta$ can be taken
arbitrarily small as $\varepsilon$ goes to zero.
We thus fix some $\delta'>0$ and consider an integer $N\geq 1$
such that a tubular neighbourhood (of width $\delta_0$) is contracted by $f^{-N}$
as a tubular neighbourhood of width less than $\delta'$.
For any point $y\in \widehat K$, the closure of the set $f(T_{f^{-1}(y)})\cap T_y\cap f^{-1}(T_{f(y)})$
is contained in the interior of $T_y$, so that 
by choosing $\varepsilon$ small enough for any $k\in \{0,\dots,N\}$,
the set $f(T_{z_{k-1}})\cap T_{z_k}\cap f^{-1}(T_{z_{k+1}})$
is contained in $T_{f^{k}(z_0)}$.
One deduces that the intersection of the $f^{-k}(T_{z_k})$ for $0\leq k\leq N$
is contained in the intersection of the $f^{-k}(T_{f^k(z_0)})$ for $0\leq k< N$,
hence is contained in the $\delta'$-tubular neighbourhood of $\widetilde \cD^{cs}_{z_0}$.
With this argument one deduces that each forward iterate $f^n(x)$
of a point $x$ built above is both $\delta'$-close to a point in $\widetilde \cD^{cs}_{z_n}$
and $r$-close to $z_n$. The point $z_n$ is thus $\delta$-close to $\cD^{cs}_{f^n(x)}$
for some constant $\delta$ that goes to zero as $\delta'$ decreases.
This proves the required property and half of the first item of the lemma.

Let us assume that two orbits $(f^n(x))_{n\geq 0}$ and $(f^n(x'))_{n\geq 0}$
$r$-shadow the same pseudo-orbit $(z_n)_{n\geq 0}$.
One can repeat the argument of lemma~\ref{l.uniq-NH}: for each $n\geq 0$, the points
$f^n(x')$ and $f^n(x)$ are $2r$-close. Let $\gamma$ be a curve tangent to a cone
around the direction $E^u$ that joints $x'$ to $\cD^{cs}_x$.
If the length of $\gamma$ is non-zero,
by forward iterates the length of $f^n(\gamma)$ increases exponentially,
contradicting the fact that $f^n(x')$ and $f^n(x)$ remain $2r$-close.
This proves that $x'\in\cD^{cs}_x$ and proves the second item of the lemma.

Let us assume that the sequence $(z_n)_{n\geq 0}$ is $\tau$-periodic
and let $(f^n(x))_{n\geq 0}$ be a sequence that $(r,\delta)$-cs-shadows $(z_n)_{n\geq 0}$.
Then the sequences $(f^{n+k\tau}(x))_{n\geq n}$, for any $k\geq 1$, also $(r,\delta)$-cs-shadows $(z_n)_{n\geq 0}$,
so that $x'$ and $f^\tau(x)$ are $2r$-close and
$f^\tau(x)$ belongs to the plaque $\cD^{cs}_{x}$.
Note that one can have reduced the plaques $\cD^{cs}$ from the very beginning so that
they satisfy the trapping property of lemma~\ref{l.uniq-cs}.
One then deduces that $f^\tau(\overline{\cD^{cs}_{x}})$ is mapped inside
$\cD^{cs}_{x}$ and that $x$ belongs to the stable manifold of a $\tau$-periodic point $p$.
The sequences $(f^{n+k\tau}(x))_{n\geq n}$ converge pointwise towards the periodic sequence $(f^n(p))_{n\geq 0}$.
Hence the orbit of $p$ also $(r,\delta)$-cs-shadows $(z_n)_{n\geq 0}$.
This finishes the proof of the first item.
\end{proof}

As a consequence one can control the strong stable manifolds of periodic orbits
close to a partially hyperbolic set of type (N).

\begin{corollary}\label{c.instable-fort}
Let $K$ be a chain-transitive set which has a partially hyperbolic splitting
$E^s\oplus E^c\oplus E^u$ such that $E^c$ is one-dimensional and has type (N).
Let $x\in K$.
Then, for every $\eta>0$ there exists a neighbourhood $U$ of $K$ with the following property.
For any periodic orbit $\cO$ contained in $U$, there exists a point $y\in W^{ss}(\cO)$
that is $\eta$-close to $x$ and whose forward orbit is contained in the $\eta$-neighbourhood of $K$.
\end{corollary}
\begin{proof}
We apply the weak shadowing lemma to the set $K$ and to a small constant $r\in (0,\eta)$ to be precised later.
Let $\delta,\varepsilon>0$ be two small constants as in lemma~\ref{l.shadows}.
Let $\cO$ be a periodic orbit contained in a $\varepsilon$-neighbourhood of $K$.

Since $K$ is chain-transitive there exists a $\varepsilon$-pseudo-orbit $(z_{n})_{n\geq 0}$ in $K$
whose tail coincides with $\cO$ and such that $z_0=x$.
One deduces that there exists an orbit $(f^n(x'))_{n\geq 0}$
that $(r,\delta)$-cs-shadow the pseudo-orbit. Moreover, by uniqueness in lemma~\ref{l.shadows},
for $n$ large $f^n(x')$ belongs to $\cD^{cs}_{f^n(p)}$ where $p$ is a point in $\cO$.
Since the points $f^n(x')$ and $f^n(p)$ are at distance $r$,
one deduces from lemmas~\ref{l.uniq-cs} and~\ref{l.uniq-NH} that a uniform
neighbourhood of $f^n(x')$ in $\cD^{c}_{f^n(x')}$ is contained in the plaque $\cD^{cs}_{f^n(p)}$.
In particular $W^{ss}_{loc}(f^n(p))$ and $\cD^{c}_{f^n(x')}$
intersect at some point $y^n$.

Having chosen $r>0$ small, the distance between $f^n(x')$ and $y^n\in \cD^c_{f^n(x')}$ is small.
Since the central dynamics along the plaques $\cD^c$ is trapped under backward iterations,
the distance between $f^{-n}(y^n)$ and $x'$ is small also.

We thus have shown that provided $r$ is small,
the point $x'$ is arbitrarily close to a point $y\in W^{ss}(\cO)$.
Since $r<\eta$,
by construction the forward orbit of $y$ belongs to the $\eta$-neighbourhood of $K$.
\end{proof}

The existence of non-minimal sets of type (N) allows to create heterodimensional cycles.

\begin{corollary}\label{c.NN}
Let $f$ be a diffeomorphism whose periodic orbits are hyperbolic
and let $A$ be a chain-transitive partially hyperbolic set
with a one-dimensional central bundle of type (N).
If $A$ is not minimal, then for each minimal set $K\subset A$ and
each neighbourhoods $U$ of $K$ and $V$ of $A$,
there exists a diffeomorphism $g$ arbitrarily $C^1$-close to $f$
having a heterodimensional cycle contained in $V$
and associated to periodic orbits contained in $U$.
\end{corollary}
\begin{proof}
Consider a neighbourhood $\cU$ of $f$ in $\diff^1(M)$ and a point $x\in A\setminus K$.
One may always assume that it is non-periodic:
if $x$ is periodic (and hyperbolic by assumption),
the chain-transitive set $A$ should contain a point $x'\in W^s(x)\setminus \{x\}$
close to $x$ which is non-periodic.
One can apply the connecting lemma to $(f,\cU)$ and introduce an integer $N\geq 1$
and two small neighbourhoods $W\subset \widehat W$ of $x$ such that the $N$ first iterates of $\widehat W$
are disjoint from $K$ and contained in $V$.
By lemma~\ref{l.shadows}, there exists a periodic orbit $\cO$ contained in an arbitrarily small
neighbourhood of $K$ and by corollary~\ref{c.instable-fort},
the strong stable and unstable manifolds of $\cO$ intersect $W$.
The connecting lemma produces a $C^1$-perturbation of $f$ in $\cU$ with support in
the $N$ first iterates of $\widehat W$ (hence disjoint from $K$)
which connects the strong manifolds of $\cO$.
Since the central exponent of $\cO$ is arbitrarily close to $0$ (by remark~\ref{r.classification}.d),
another $C^1$ perturbation gives a heterodimensional cycle (proposition~\ref{p.strong-cycle}).
By construction the heteroclinic orbits are contained in $V$.
\end{proof}

\section{Twisted partially hyperbolic sets}

We here study partially hyperbolic sets which have a twisted geometry that already appeared
in~\cite{bonatti-gan-wen} and~\cite{crovisier-palis-faible}.
It generates strong homoclinic intersections by $C^1$-perturbation.

\subsection{Definition and main result}
Let $K$ be an invariant compact set with a partially hyperbolic structure $E^s\oplus E^c\oplus E^u$,
such that $E^s$, $E^u$ are both non-degenerated and $E^c$ is one-dimensional.
Recall that any point $x\in K$ has local strong stable and local strong unstable manifolds
$W^{ss}_{loc}(x)$ and $W^{uu}_{loc}(x)$ tangent to $E^s_x$ and $E^u_x$; this is also the case
for any diffeomorphism $C^1$-close and any point whose orbit is contained in a small neighbourhood of $K$.

Some pairs of close points $p,q\in K$ are in a \emph{twisted position} if
one can connects $W^{ss}_{loc}(p)$ to $W^{uu}_{loc}(q)$
and $W^{ss}_{loc}(q)$ to $W^{uu}_{loc}(p)$ by two (small) curves tangent to
a central cone field and having a same orientation (see figure~\ref{f.twist}):
this is possible since the bundle $E^c$ can be locally oriented.
In particular if $W^{ss}_{loc}(p)$ and $W^{uu}_{loc}(q)$ intersect, the pair $(p,q)$
is in a twisted position.

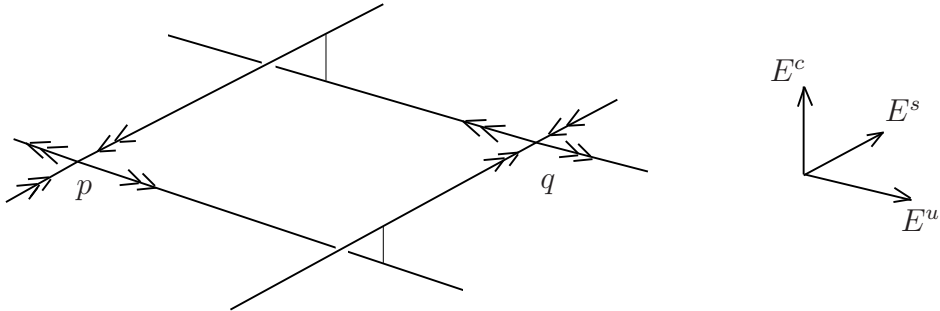
\begin{figure}[ht]
\begin{center}
\input{twist.pstex_t}
\end{center}
\caption{Two points $p$, $q$ in a twisted position.\label{f.twist}}
\end{figure}

In principle this definition depends on the choice of local orientations and cone fields
but it is noticed in~\cite{crovisier-palis-faible} that this notion is well-defined when the distance
$d(p,q)$ goes to zero. Hence, this allows the following definition.

\begin{definition}
The set $K$ is said to be \emph{twisted} if there is $\varepsilon>0$ such that
for any $p,q\in K$ satisfying $d(p,q)<\varepsilon$, the pair $(p,q)$ is twisted.
\end{definition}

\begin{proposition}\label{p.twisted}
Let $f$ be a diffeomorphism and let $K$ be a twisted partially hyperbolic set
which holds a minimal dynamics and is non-periodic.
Then for any neighbourhood $\cU$ of $f$ in $\diff^1(M)$ and $U$ of $K$ in $M$,
there exist $g\in \cU$ having a periodic orbit $\cO$ which
exhibits a strong homoclinic connection associated to a homoclinic
orbit contained in $U$.
\end{proposition}

\begin{corollary}\label{c.twist}
Let $f$ be a diffeomorphism and consider a non-periodic minimal set $K$
which has a partially hyperbolic structure $E^s\oplus E^c\oplus E^u$
such that $E^c$ is one-dimensional.
If $K$ has type (P) or (N) and is twisted, then, for any neighbourhood $\cU$ of $f$ and
$U$ of $K$, there exists a diffeomorphism $g\in \cU$ which exhibit a heterodimensional cycle
contained in $U$.
\end{corollary}
\begin{proof}
One applies proposition~\ref{p.twisted}:
there exists a perturbation $g_0$ of $f$ having a periodic orbit $\cO$
which exhibits a strong homoclinic connection associated to a homoclinic orbit contained in
a small neighbourhood of $K$.
Since $K$ has type (P) or (N), the exponent of $\cO$ along $E^c$ is close to zero.
By proposition~\ref{p.strong-cycle}, one gets with another $C^1$-perturbation a heterodimensional cycle in $U$.
\end{proof}

The end of the section is devoted to the proof of proposition~\ref{p.twisted}.
One will first create
a perturbation $g$ of $f$ and a periodic orbit $\cO=(x,g(x),\dots,g^\tau(x)=x)$ by the closing lemma:
the period of $\cO$ can be chosen arbitrarily long but there exists an integer $N\geq 1$
(which is the time needed to perform the perturbation) such that the
segment of orbit $(x,g(x),\dots, g^{\tau-N}(x))$ coincides with
$(x,f(x),\dots, f^{\tau-N}(x))$ and belongs to $K$.
Since $K$ is twisted and most of the iterates of $\cO$ belong to $K$, we will deduce that
the orbit $\cO$ is ``almost twisted" for $g$.
Since its period is arbitrarily long, one can repeat the argument in~\cite{crovisier-palis-faible}
for twisted orbits and by perturbation one will obtain a strong homoclinic intersection for $\cO$.

\subsection{The closing lemma argument}
We first recall how one can perturb the dynamics of $f$ in order to create a periodic
orbit close to $K$. We denote by $d$ the dimension of $M$.
A \emph{standard cube} in $\RR^d$ is a set of the form
$C=[-a,a]^d+v$ where $a>0$ and $v\in \RR^d$.
For $t>0$ we denote by $t.C$ the standard cube $[-t.a,t.a]^d+v$.
A main step towards Pugh's closing lemma is the following
(see~\cite[th\'eor\`eme 22]{arnaud} and~\cite[th\'eor\`eme A.1]{BC}).

\begin{theorem*}[Pugh]
Let $f$ be a diffeomorphism, $\cU_0$ be a neighbourhood
of $f$ in $\diff^1(M)$, $z\in M$ be a non-periodic point and $\lambda>1$.
Then there exists
\begin{itemize}
\item[--] a chart $\varphi\colon \Omega\subset \RR^d \to M$ around $z$,
\item[--] an integer $N\geq 1$,
\end{itemize}
with the following property.

For any standard cube $\lambda.C\subset \Omega$ and any $x,y\in \varphi(C)$,
there exists $g\in \cU_0$ such that:
\begin{itemize}
\item[--] $g^N(y)=f^N(x)$;
\item[--] $g$ is a perturbation of $f$ with support in the union of the $f^k(\varphi(\lambda.C))$
with $0\leq k<N$.
\end{itemize}
\end{theorem*}

In order to create a periodic orbit, one looks for a (non-trivial) segment of orbit
$(x,\dots,f^n(x))$ with $x,f^n(x)\in \varphi(C)$ such that all the intermediary iterates
$f^k(x)$ for $0<k<n$ avoid $\varphi(\lambda.C)$.
One possible approach is to assume that $z$ belongs to the support of an invariant probability
measure $\mu$ and to argue in a similar way as in the proof of Ma\~n\'e's ergodic closing lemma~\cite{mane-ergodic}.
The restriction of the measure $\mu$ to the image of $\varphi$
is non-zero and lifts as a measure $\tilde \mu$ on $\Omega$.

\begin{lemma}\label{l.measure}
If $\lambda^{d+1}<\frac 3 2$,
there exists a standard cube $\lambda.C\subset \Omega$
which satisfies $\tilde \mu(C)>\frac 1 2 \tilde \mu(\lambda.C)$.
\end{lemma}
\begin{proof}
The proof is by contradiction:
one assumes that for each standard cube $\lambda.C\subset \Omega$,
one has $\tilde \mu(C)\leq \frac 1 2 \tilde \mu(\lambda.C)$.
Let $C_0$ be a standard cube with positive measure $m_0=\tilde \mu(C_0)$
and such that $2.C_0$ is contained in $\Omega$.
One builds inductively a sequence of standard cubes $C_n$ of positive measure
$m_n=\tilde \mu(C_n)$ in the following way.
The cube $C_n$ can be divided into $2^d$ standard cubes with disjoint interior
and same radii. One of them, denoted by $C'_n$, has measure larger than $2^{-d}m_n$.
We define $C_{n+1}=\lambda^{d+1}.C'_n$.
By our assumption it has measure larger than $2^{d+1}\tilde \mu(C'_n)$, hence larger than
$2 \tilde \mu(C_n)$. However by our choice of $\lambda$,
the cube $2.C_{n+1}$ is contained in $2.C_n$ and has a radius less
than $3/4$ times the radius of $2.C_n$.
The cubes $C_n$ are thus all contained in $\Omega$ and
their measures increase exponentially contradicting the fact that $\tilde \mu$ is finite.
\end{proof}

As a consequence there exists a set $A\subset \varphi(C)$
with positive measure for $\mu$
and an integer $n\geq 1$ such that $f^k(A)\cap \varphi(\lambda.C)=\emptyset$
for any $0<k<n$ and $f^n(A)\subset \varphi(C)$.
Any segment of orbit $(x,\dots,f^n(x))$ with $x\in A$ can be closed by Pugh's theorem.

\subsection{A codimensional version of the closing lemma.}
Pugh's theorem allows to modify the image of a point with
a good control on the support of the perturbation.
We state a variation of this result which allows to modify the image of a point
by moving its first $N$ iterates in a given direction.
We fix some integer $1\leq s\leq d$ and denote by $p_2$ the orthogonal projection
of $\RR^n$ on $\{0\}^s\times \RR^{n-s}$.

\begin{proposition}\label{p.closing}
Let $f$ be a diffeomorphism, $\cU_0$ be a neighbourhood
of $f$ in $\diff^1(M)$,
$K\subset M$ be an invariant compact set with
a continuous splitting $T_KM=E\oplus F$ with $s=\dim(E)$,
$z\in K$ be a non-periodic point and $\lambda>1$.
Then, there exists
\begin{itemize}
\item[--] a chart $\varphi\colon \Omega\subset \RR^d \to M$
around $z$ satisfying $D_0\varphi.(\RR^s\times \{0\}^{n-s})=E_z$,
\item[--] an integer $N\geq 1$ and a constant $\alpha_0>0$,
\item[--] two functions $\beta,\gamma\colon (0,\alpha_0)\to (0,1)$ satisfying
$\beta(\alpha),\gamma(\alpha)\underset{\alpha\to 0}\longrightarrow 0$,
\end{itemize}
with the following property.

For any $\alpha\in (0,\alpha_0)$, any standard cube $\lambda.C\subset B_{\RR^d}(0,\gamma(\alpha))$,
and any $x,y\in \varphi(C)$ satisfying
\begin{equation}\label{e.codim}
d(p_2\circ \varphi^{-1}(x),p_2\circ \varphi^{-1}(y))<\alpha. \operatorname{diam}(C),
\end{equation}
there exists $g\in \cU_0$ such that
\begin{itemize}
\item[--] $g^N(y)=f^N(x)$;
\item[--] $g$ is a perturbation of $f$ with support in the union of the $f^k(\varphi(\lambda.C))$
with $0\leq k<N$;
\item[--] for each $1\leq k \leq N$ one has
\begin{equation}\label{e.pince}
d(p_2\circ \varphi^{-1}(f^{-k}g^k(x)),p_2\circ \varphi^{-1}(y))<\beta(\alpha). \operatorname{diam}(C).
\end{equation}
\end{itemize}
\end{proposition}

As for Pugh's theorem above, this can be deduced from
a result about linear maps (see~\cite[proposition A.1]{BC}).

\begin{lemma}[Pugh]\label{l.pugh}
Let $(T_n)_{n\geq 0}$ be a sequence in $GL(s,\RR)$, $\eta$ be a constant in $(0,1)$ and $\lambda>1$.
Then there exist $A\in GL(s,\RR)$ and an integer $N\geq 1$ with the following property.

For any $x,y$ in the standard cube $C_0=[0,1]^s$, there exists a sequence $(z_0,\dots,z_N)$ in the cube
$\sqrt\lambda.C_0$ such that
\begin{itemize}
\item[--] $z_0=x$, $z_N=y$;
\item[--] for each $1\leq k \leq N$ one has
$$d\left(T_{k}A(z_{k-1}),\;T_kA(z_{k})\right)\;<\;\eta.
d\left(T_kA(\sqrt\lambda.C_0),\;\RR^s\setminus T_k(\lambda.C_0)\right).$$
\end{itemize}
\end{lemma}
\medskip

\begin{proof}[Proof of proposition~\ref{p.closing}]
We fix a Riemannian metric on $M$.
Reducing the neighbourhood $\cU_0$ of $f$ there exist $\delta>0$ and $\rho\in (0,1)$ such that
for any disjoint balls $B(x_1,r_1),\dots,B(x_\ell,r_\ell)$ in $M$ with radius $r_i<\delta$
and any points $a_i,b_i\in f^{-1}(B(x_i,\rho.r_i))$, $i=1,\dots,\ell$, there exists an
(elementary) perturbation $g\in \cU_0$ supported in the union of the balls $f^{-1}(B(x_i,r_i))$ such that
$g(a_i)=f(b_i)$ (this exists by the lift axiom of~\cite[sections 2 and 6]{pugh-robinson}).

Let us fix orthonormal bases at the forward iterates of $z$ by $f$.
Each tangent space is thus identified to $\RR^d$ and one can assume furthermore that
$Df^n.E_z$ is identified to the space $\RR^s\times \{0\}^{d-s}$ for each $n\geq 0$.
The derivative $D_zf^n\colon T_zM\to T_{f^n(z)}M$
induces a linear map $\widehat T_n\in GL(d,\RR)$ and the restriction of
$D_zf^n$ to $E_z$ induces a linear map $T_n\in GL(s,\RR)$.

Since the splitting $E\oplus F$ is continuous, the angle between vectors
in $E$ and $F$ is bounded from below. Consequently, one can
consider the splitting $E_z\oplus F_z$ at $z$:
there exists $\sigma\in (0,1)$ such that for any set $\Delta\subset \RR^s$
and any $n\geq 0$, if a ball $B_{\RR^s}(x,r)$ in $\RR^s$ is contained in $T_n(\Delta)$,
then the ball $B_{\RR^d}(x\times \{0\}^{d-s},\sigma.r)$ in $\RR^d$ is contained in
$\widehat T_n(\Delta\times F)$.

We apply lemma~\ref{l.pugh} to the sequence $(T_n)$ and to the constant $\eta=\rho\sigma$.
This gives us an integer $N\geq 1$ and a linear map $A\in GL(s,\RR)$.
Let us choose $a>0$ and define $\widehat A\in GL(d,\RR)$
which preserves and coincides with $A$ on $\RR^s\times \{0\}^{d-s}=E_z$
and which maps $\{0\}^{s}\times \RR^{d-s}$ on $F_z$ by a conformal linear map
with scaling factor $a$.
Let $C_0=[0,1]^s$ be the standard cube in $\RR^s$ and $\widehat C_0$
be the standard cube in $\RR^d$.
For any $x,y$ in $\widehat C_0$ such that $y-x\in E_z$,
there exists a sequence $(z_0,\dots,z_N)$ in $\sqrt \lambda.\widehat C_0$
such that:
\begin{itemize}
\item[--] $z_0=x$, $z_N=y$,
\item[--] for each $1\leq k\leq N$, the points
$a_k=\widehat T_k\widehat A(z_{k-1})$ and $b_k=\widehat T_k\widehat A(z_k)$
satisfy $a_k-b_k\in \RR^s\times \{0\}^{d-s}$
and their distance is less than $\eta$ times the distance
$$\delta_k=d\left(T_kA(\sqrt\lambda.C_0),\;\RR^s\setminus T_kA(\lambda.C_0)\right).$$
\end{itemize}
For each $1\leq k\leq N$, let us denote by $c_k\in \widehat T_k\widehat A(\sqrt\lambda.\widehat C_0)$
the middle between $a_k$ and $b_k$ and let us set $r_k=\sigma. \delta_k$.
By definition of $\delta_k$ and the choice of $\sigma$, the ball
centred at $c_k$ with radius $r_k$ is contained in
$\widehat T_k\widehat A_k(\lambda .C_0\times \RR^{n-s})=
\widehat T_k(A(\lambda .C_0)\times F)$,
hence in $\widehat T_k\widehat A(\lambda.\widehat C_0)$
If $a$ has been chosen large enough,
the ball centred at $c_k$ with radius $r_k$ is contained in
$\widehat T_k\widehat A(\lambda.\widehat C_0)$.
Moreover the ball centred at $c_k$
with radius $\rho.r_k=\eta\delta_k$ contains both $a_k$ and $b_k$.

By continuity, the same property holds if we replace the assumption $y-x\in E_z$
by $p_2(x-y)$ close to $0$: there exists $\alpha_0>0$ such that,
for any $x,y\in \widehat C_0$ satisfying $d(p_2(x),p_2(y))<\alpha$
with $\alpha\in (0,\alpha_0)$,
there exists a sequence $(z_0,\dots,z_N)$ in $\sqrt \lambda.\widehat C_0$
and for each $1\leq k\leq N$
there is a ball $B(c_k,r_k)$ contained in $\widehat T_k\widehat A(\lambda.\widehat C_0)$
such that $B(c_k,\rho.r_k)$ contains the points
$a_k=\widehat T_k\widehat A(z_{k-1})$ and $b_k=\widehat T_k\widehat A(z_k)$.
Note that when $\alpha$ goes to zero, the sequence $(z_0,\dots,z_N)$
is closer to an affine plane in the direction of $\RR^s\times \{0\}^{d-s}$.
One deduces that the distance between $p_2(a_k)$ and $p_2(b_k)$
is less than some function $\beta_0(\alpha)$ that goes to zero when $\alpha\to 0$.
Note that the same properties are valid if one rescales or translates the standard cube $C_0$.
This ends the proof in the linear case.
\medskip

Let us define the chart $\varphi=\exp_z\circ \widehat A$
where $\exp_z$ is the exponential map from a neighbourhood $\Omega$ of $0$ in $\RR^d$ (identified to $T_zM$) to a neighbourhood of $z$ in $M$.
By construction $\varphi$ maps $0$ to $z$ and $D_0\varphi$ maps
$\RR^s\times \{0\}^{d-s}$ to $E_z$.
Since at smale scales $f, f^2,\dots,f^N$ act on the distances as the linear maps
$T_1,\dots,T_N$, one gets the following construction, assuming that the neighbourhood
$\Omega$ is small enough.
Let $C\subset \RR^d$ be a standard cube such that $\lambda. C\subset \Omega$
and let $x,y\subset\varphi(C)$
satisfying~(\ref{e.codim}) for some $\alpha\in (0,\alpha_0)$.
\begin{itemize}
\item[--] There exists a sequence $(x=z_0,\dots,z_N=y)$ of points in the cube
$\varphi(\sqrt\lambda.C)$ such that $z_0=x$, $z_N=y$ and
a sequence of balls $B(x_1,r_1),\dots,B(x_N,r_N)$ such that
$B(x_k,r_k)$ is contained in $f^k(\varphi(\lambda.C))$
and $B(x_k,\rho.r_k)$ contains the points $a_k:=f^k(z_{k-1})$ and $b_k:=f^k(z_k)$.
\item[--] Since $z$ is non-periodic, by choosing $\Omega$ small,
the iterates $\Omega,f(\Omega),\dots,f^N(\Omega)$ are pairwise disjoint and in particular
the balls $B(x_k,r_k)$ for $k=1,\dots,N$ are disjoint.
One can apply an elementary perturbation to the points
$a_k$ and $b_k$ and obtain a perturbation
$g\in \cU$ supported on the iterates $f^k(\varphi(\lambda.C))$ with $0\leq k \leq N-1$
and such that $g^N(x)=y$.
\end{itemize}
When one decreases $\alpha$, the property obtained in the linear case
implies~(\ref{e.pince})
provided $f^k$ is close enough to $T_k$, ie. that
the cube $\lambda.C$ is contained in a small neighbourhood
$B(0,\gamma( \alpha))$ of $0$ in $\RR^d$.
This gives all the conclusions of the proposition.
\end{proof}

\subsection{Periodic orbits with almost twisted returns.}
The proof of proposition~\ref{p.twisted} will be implied by the two lemmas below.
We first note that
by applying the results of the previous sections at the twisted set $K$, we do not apriori get
a twisted periodic orbit $\cO$. We keep however some information (provided by lemma~\ref{l.t1}):
any close points $p,q\in \cO$ will almost
lie in a plane tangent to $E^s\oplus E^c$.
This allows to get by a new perturbation a strong homoclinic intersection
(lemma~\ref{l.t2}).
\medskip

Fix a central cone field around the bundle $E^c$ and fix
a small neighbourhood $U$ of $K$ and $\cU$ of $f$.
For $\rho>0$ and $g\in \cU$,
two points $p,q$ whose $g$-orbit remain in $U$
are said to be \emph{$\rho$-twisted} for $g$ if
one can connect $W^{ss}_{loc}(p)$ to $W^{uu}_{loc}(q)$
and $W^{ss}_{loc}(q)$ to $W^{uu}_{loc}(p)$ (as local manifolds for $g$)
by two (small) curves tangent to the
central cone and whose length is smaller than $\rho.d(p,q)$.
Since at small scales the bundles $E^s$ and $E^u$ are almost constant,
this property holds for points $p,q$ in the twisted set $K$: \\
\emph{For any $\rho>0$ there is $\varepsilon>0$ such that
any $p,q\in K$ at distance less than $\varepsilon$ are $\rho$-twisted for $f$.}

This property can be also seen in coordinates.
Consider an exponential chart $\exp_z\colon U_z\subset \RR^d\to M$ at some point $z\in K$
and the induced (non-dominated) splitting $(E^s_z\oplus E^u_z)\oplus E^c_z$ of $\RR^d$.
One defines in $\exp_z(U_z)$ a central distance $d^c$ by considering the projection on $E^c_z$
in the chart.
Any points $p,q\in \exp_z(U_z)\cap K$ are in twisted position,
hence $d^c(p,q)\leq \eta. d(p,q)$ where $\eta$
can be chosen arbitrarily small if $U_z$ is small enough.
For $g\in \cU$ and any point $p\in \exp_z(U_z)$ whose orbit by $g$ remains in $U$,
the manifold $\exp_z^{-1}(W^{ss}_{loc}(p))$ is tangent to a cone arbitrarily thin
around $E_z^s$ provided $g$ is close enough to $f$ and $U_z$ is small enough;
a similar property holds for $W^{uu}_{loc}(p)$.
One deduces the following:
\begin{claim}\label{c.twist1}
For any $\rho>0$ and $z\in K$, one can choose $\eta>0$,
a $C^1$-neighbourhood $\cU'$ of $f$ and a neighbourhood $U_z$ of $0\in \RR^d$
with the following property.
For any $g\in \cU'$ and any $p,q\in \exp_z(U_z)$ whose orbits by $g$ remain in $U$,
if $d^c(p,q)<\eta. d(p,q)$ then $p$ and $q$ are $\rho$-twisted for $g$.
\end{claim}
\medskip

Let $(f_n)$ be a sequence of diffeomorphisms which converges toward $f$ in $\diff^1(M)$
and let $(\cO_n)$ be a sequence of $f_n$-periodic orbits which converges toward $K$
in Hausdorff topology. We say that $(\cO_n)$ has \emph{almost twisted returns}
if for any $\rho\in (0,1)$, there exists $\varepsilon>0$ such that, for any $n$,
any points $p,q$
contained in $\cO_n$ and at distance smaller than $\varepsilon$
are $\rho$-twisted for $f_n$.

\begin{lemma}\label{l.t1}
There exists a sequence of diffeomorphisms $(f_n)$ converging towards $f$
in $\diff^1(M)$ and a sequence of periodic orbits $(\cO_n)$ which converges toward $K$
for the Hausdorff topology and which has almost twisted returns.
\end{lemma}
\begin{proof}
Since $K$ is twisted, there exist two decreasing sequences of positive numbers
$(\varepsilon_n)_{n\geq 0}$ and $(\rho_n)_{n\geq 0}$ which goes to zero as $n$ goes to $+\infty$
and such that any $p,q\in K$ at distance smaller than $2\varepsilon_n$
are $\frac{\rho_n}{2}$-twisted.

Let us fix $n\geq 0$.
The local manifolds at a point $x$ whose orbit is contained in $U$ are $C^1$ immersed discs
which depend continuously on the diffeomorphism and on the point $x$.
This implies that there exists a neighbourhood $\cU_n$ of $f$ satisfying:
\begin{itemize}
\item[1)] For any $0\leq i \leq n$ and any points $p_0,q_0\in K$ at distance less than $2\varepsilon_i$,
there exist some neighbourhoods $U(p_0)$
and $U(q_0)$ such that, for any $g\in \cU_n$, any points $p\in U(p_0)$, $q\in U(q_0)$
whose orbit under $g$ remain in $U$ are $\rho_i$-twisted for $g$.
\end{itemize}
By claim~\ref{c.twist1}, one can cover $K$ by finitely many exponential charts $\exp_{z}(U_{z})$
with $z\in K$, reduce the neighbourhood $\cU_n$ of $f$ and choose $\eta>0$ with the following property.
\begin{itemize}
\item[(*)]
For any $g\in \cU_n$ and any points $p,q$
in a same exponential chart and whose orbit under $g$ remains in $U$,
if $d^c(p,q)<\eta.d(p,q)$, then $p$ and $q$ are $\rho_n$-twisted for $g$.
\end{itemize}
\medskip

Let us now create a periodic orbit $\cO_n$ by using the method described in the previous sections.
One first considers a point $z_0\in K$ which is the centre of one of the exponential charts which cover $K$.
One applies proposition~\ref{p.closing} to $(f,\cU_n)$, to the set $K$,
to the continuous splitting $(E^s\oplus E^u)\oplus E^c$ and to some constant $\lambda>1$
such that $\lambda^{2(d+1)}<\frac 3 2$.
One gets another chart $\varphi\colon \Omega\subset \RR^d\to M$ at $z_0$.
One also gets an integer $N\geq 1$ and two functions $\beta,\gamma\colon (0,\alpha_0)\to (0,1)$.
One may reduce the open set $\Omega$ so that
the iterates $f^k(\varphi(\Omega))$, $0\leq k\leq N$, are disjoint
and each of them is contained in an exponential chart $\exp_{z_k}(U_{z_k})$
from the covering of $K$; in particular $\varphi(\Omega)\subset \exp_{z_0}(U_{z_0})$.
One may reduce $\Omega$ again so that:
\begin{itemize}
\item[2)] There exists $\eta'\in (0,1)$ such that for any $p,q\in \varphi(\Omega)$
satisfying $d^c(p,q)\leq\eta'.d(p,q)$ for the coordinates of the chart $\exp_{z_0}$,
we also have $d^c(f^k(p),f^k(q))\leq\eta.d(f^k(p),f^k(q))$ for the coordinates of the chart $\exp_{z_k}$
and for any $1\leq k\leq N$.
\end{itemize}
Using 1), one can also require:
\begin{itemize}
\item[3)] For any $g\in \cU_n$, any $0\leq i \leq n$, any $k\neq \ell$ in $\{0,\dots, N\}$,
any points $p\in f^k(\varphi(\Omega))$ and  $q\in f^\ell(\varphi(\Omega))$ that are $\varepsilon_i$-close
whose orbit by $g$ remain in $U$, are $\rho_i$-twisted for $g$.
\end{itemize}
There exists a constant $\Delta>0$ such that for any
standard cube $\lambda^2.C\subset \Omega$, and any points $q\not\in \varphi(\lambda^2.C)$
and $x,y\in \varphi(\lambda.C)$ one has
\begin{equation}\label{e.cube}
\operatorname{diam}(\varphi(C))\leq \Delta\;.\; d(q,y) \text{ and }
d(q,x)\leq \Delta\;.\; d(q,y).
\end{equation}
For any points $p,q\in K\cap \varphi(\Omega)$ one has
\begin{equation}\label{e.twist}
d^{c}(p,q)\leq\eta''.d(p,q),
\end{equation}
for the coordinates of $\exp_{z_0}$, where $\eta''>0$
can be chosen arbitrarily small if $\Omega$ is small enough.
As a consequence,
if one applies proposition~\ref{p.closing}
to some standard cube $\lambda.C\subset \Omega$,
and some points $x,y\in \varphi(C)\cap K$,
condition~(\ref{e.pince}) holds with $\alpha$ arbitrarily small.
One deduces from~(\ref{e.pince}) that
there exists a perturbation $g\in \cU_n$ such that $g^N(y)=g^N(x)$ and moreover
\begin{equation}\label{e.it}
d^c(x,f^{-k}(g^k(y)))\leq\beta.\operatorname{diam}(\varphi(C)),
\end{equation}
where $\beta>0$ can be chosen arbitrarily small if $\Omega$ is small.
Consequently, one can choose $\eta''$ and $\beta$ small so that
\begin{equation}\label{e.sum}
(\eta''+\beta).\Delta\leq\eta'.
\end{equation}

Let $\mu$ be an invariant probability measure supported on $K$.
By minimality $\supp(\mu)=K$.
By lemma~\ref{l.measure} and our choice of $\lambda$,
there exists a standard cube $\lambda^2.C\subset \Omega$
such that
\begin{equation}\label{e.measure}
\mu(\varphi(\lambda^2.C))<2\mu(\varphi(C)).
\end{equation}
One may assume that $\varphi(\lambda^2.C)$ is contained in an arbitrarily small neighbourhood of $p$.
From 1) and a compactness argument, this implies the following:
\begin{itemize}
\item[4)] For any $g\in \cU_n$, any $0\leq i \leq n$, any $0\leq k \leq N$,
any points $p\in f^k(\varphi(\lambda.C))$ and $q\in K\setminus f^k( \varphi(\Omega))$
that are $\varepsilon_i$-close and
whose orbit by $g$ remain in $U$ are $\rho_i$-twisted for $g$.
\end{itemize}

One deduces from~(\ref{e.measure}) that there exists a point $x\in \varphi(C)$ having
a return $f^m(x)\in \varphi(C)$ such that all the iterates $f^k(x)$ for $0< k<m$
are disjoint from $\varphi(\lambda^2.C)$.
Consequently, there exists a perturbation $f_n\in \cU_n$ of $f$
supported on the $f^k(\varphi(\lambda.C))$, $0\leq k<N$,
such that $f_n^N(f^m(x))=f^N(x)$.
In particular, the orbit $\cO_n$ of $f^m(x)$ for $f_n$ is $m$-periodic.
\medskip

The lemma will be implied by the following.
\begin{claim} If $p,q\in \cO$ are at distance less than $\varepsilon_i$ for some $i\leq n$,
then they are $\rho_i$-twisted.
\end{claim}
This is proved by assuming $p\neq q$ and by considering several cases.
\begin{itemize}
\item[--] When $p,q$ do not belong to the support of the perturbation,
they belong to $K$ by construction. Property 1) implies the claim in this case.

\item[--] When $p$ belongs to some $f^k(\varphi(\lambda.C))$
and $q$ to some $f^\ell( \varphi(\lambda.C))$, with $0\leq k,\ell\leq N$, one has
$p=f_n^{k}(f^m(x))$ and $q=f_n^{\ell}(f^m(x))$. Hence $k\neq \ell$.
By 3) above, one concludes the claim in this case.

\item[--] When $p$ belongs to some $f^k( \varphi(\lambda.C))$, $0\leq k\leq N$, and $q$ does not belong to
$f^k( \varphi(\Omega))$ nor to any $f^\ell(\varphi(\lambda.C))$, $0\leq \ell\leq N$,
one has $q\in K$. By 4) above, one concludes the claim again.

\item[--] When $p$ belongs to some $f^k(\varphi(\lambda.C))$ and $q$ to $f^k( \varphi(\Omega))$,
for some $0\leq k\leq N$, one considers the points $f^{-k}(p),f^{-k}(q)\in \varphi(\Omega)$.
Note that $p=f^k_n(f^m(x))$.
One has $q\not\in f^k(\varphi(\lambda.C))$, hence
the points $x,f^{-k}(q)\in \varphi(\Omega)$ belong to $K$ and by~(\ref{e.twist}) satisfy
$d^c(x,f^{-k}(q))\leq\eta''.d(x,f^{-k}(q))$ for the coordinates of the chart $\exp_{z_0}$.
One has also $d^c(x,f^{-k}(p))=d^c(x,f^{-k}(f_n^k(f^m(x))))\leq \beta.\operatorname{diam}(\varphi(C))$ by~(\ref{e.it}).
Note also that since $f^{-k}(q)$ is disjoint from $\varphi(\lambda^2.C)$ and
since $f^{-k}(p)$ and $x$ belong to $\varphi(\lambda.C)$, one has by~(\ref{e.cube})
$$\sup(\operatorname{diam}(\varphi(C)),\; d(x,f^{-k}(q)))\; \leq\; \Delta.d(f^{-k}(p),f^{-k}(q)).$$
One thus deduces from~(\ref{e.measure}).
$$d^c(f^{-k}(p),f^{-k}(q))\leq (\eta''+\beta)\Delta.d(f^{-k}(p),f^{-k}(q))\leq
\eta'.d(f^{-k}(p),f^{-k}(q)).$$
Now using 2), one gets $d^c(p,q)\leq\eta.d(p,q)$.
By our choice (*) of the exponential chart $\exp_{z_k}(U_{z_k})$, this implies
that $p$ and $q$ are $\rho_n$-twisted, implying the claim in this case.
\end{itemize}

The claim and the lemma are now proved.
\end{proof}

One can now conclude the proof of proposition~\ref{p.twisted}.
\begin{lemma}\label{l.t2}
Let $(f_n)$ be a sequence of diffeomorphisms which converges toward $f$ in $\diff^1(M)$
and let $(\cO_n)$ be a sequence of $f_n$-periodic orbits which converges toward $K$
in Hausdorff topology and which has almost twisted returns.

Then, for any neighbourhood $U$ of $K$ and $\cU$ of $f$ in $\diff^1$, there exists
$n\geq 1$ and a diffeomorphism $g\in \cU$ which coincides with $f_n$ outside $U$
and on a small neighbourhood of $\cO_n$, such that $\cO_n$ has a strong homoclinic
intersection associated to a homoclinic orbit contained in $U$.
\end{lemma}
The proof is the same as for~\cite[theorem 3.18]{crovisier-palis-faible}
or~\cite[theorem 9]{bonatti-gan-wen}.

\section{Partially hyperbolic sets of types (R), (H) or (P)}\label{s.HPR}
In this section $f$ is a $C^1$-generic diffeomorphism. Unless mentioned,
$K$ is a chain-transitive set having a partially  hyperbolic structure $E^s\oplus E^c\oplus E^u$ with a one-dimensional central bundle.
We discuss the existence of a periodic orbit $\cO$ contained in an arbitrarily small neighbourhood of $K$
such that either $\cO$ is contained in the chain-recurrence class of $K$, or by a $C^1$-perturbation of the diffeomorphism $f$,
there exists a strong homoclinic intersection associated to the continuation of $\cO$.

\subsection{The chain-transitive type (R)}
The results in~\cite[section 3.3]{crovisier-palis-faible} give the following.
\begin{proposition}\label{p.R}
Let $f$ be any diffeomorphism in a dense G$_\delta$ subset $\cR_{R}\subset \diff^1(M)$ and
$K$ be a chain-transitive set which has a partially hyperbolic structure $E^s\oplus E^c\oplus E^u$
such that $E^c$ is one-dimensional. If $K$ has type (R), then, for any small neighbourhood $U$ of $K$
there exists a periodic orbit $\cO\subset U$ contained in the chain-recurrence class of $K$.

More precisely:
\begin{itemize}
\item[--] Any periodic orbit $\cO$ contained in $U$ and having a point close to the
middle of a central segment $\gamma$ of $K$ belongs to the chain-recurrence class of $K$.
\item[--] For any such periodic orbit $\cO$, there exist some diffeomorphisms $g$
that are arbitrarily $C^1$-close to $f$
and such that the continuation $\cO_{g}$ of $\cO$ has a strong homoclinic connection.
\end{itemize}
\end{proposition}

Using the same kind of argument as the one used in the proof of proposition~\ref{p.R} one gets:
\begin{proposition}\label{p.R'}
Let $f$ be any diffeomorphism in a dense G$_\delta$ subset $\cR'_{R}\subset \diff^1(M)$ and
$H$ be a homoclinic class which has a dominated splitting $E^s\oplus E^c\oplus E_3$
such that the bundle $E^c$ is one-dimensional and the bundle $E^s$ is uniformly contracted.
Let us assume furthermore that $H$ contains periodic orbits whose Lyapunov exponent
along $E^c$ is arbitrarily close to $0$.
If $H$ has type (R), then it contains periodic points with index $\dim(E^s)$.
\end{proposition}
\begin{proof}
By~\cite{BC} and the Franks lemma (which allows to change the derivative along a periodic orbit
by a $C^1$-perturbation), there exists a dense G$_{\delta}$ subset $\cR'_{R}$ of $\diff^1(M)$ such that
any diffeomorphism $f\in \cR_{R}$ has the following properties:
\begin{enumerate}
\item[a)] Each homoclinic class of $f$ is a chain-recurrence class.
\item[b)] Any orbits $\cO,\cO'$ in a same chain-recurrence class and with the same index
are homoclinically related.
\item[c)] Any set $\Lambda$ which is the Hausdorff limit of hyperbolic periodic orbits $\cO_{n}$ of index $i$ and whose $i$-th Lyapunov
exponent is arbitrarily close to zero is also the limit of hyperbolic periodic orbits $\cO_{n}'$ of index is $i-1$.
Moreover if there exists a hyperbolic periodic orbit $\cO$ such that
$W^u(\cO_{n})$ intersects transversely $W^s(\cO)$ for each $n\geq 0$,
then the same property can be required for the periodic orbits $\cO'_{n}$.
\end{enumerate}

Let us fix a hyperbolic periodic orbit $\cO$ contained in $H$.
By property a),
the homoclinic class of $\cO$ coincides with $H$.
By assumption, one can choose $\cO$ to have a Lyapunov exponent along $E^c$ close to $0$,
hence the index of $\cO$ is $\dim(E^s)$ or $\dim(E^s)+1$.
In the first case we are done, so we will assume that the index of $\cO$
is $\dim(E^s)+1$.

By properties a) and b), there exists a dense set of periodic points $Q$ in $H$
that are homoclinically related to $\cO$ and have an arbitrarily weak central Lyapunov exponents.
By property c), one can thus conclude that
for any point $z\in H$, any neighbourhoods $U$ of $z$ and $V$ of $H$,
there exists a periodic point $Q\in U$ of index $\dim(E^s)$, whose orbit is contained in $V$,
and such that $W^u(Q)\cap W^s(\cO)\neq \emptyset$.

Let $\gamma$ be a central segment and consider $z$ in its interior.
Since the backward iterates of $\gamma$ remain tangent to a neighbourhood of $E^c$
and have a length bounded, one deduces that the exponential growth of $\|Df^{-n}_{|E^c}(z)\|$
when $n\to +\infty$ is weak. By domination, $\|Df^{-n}_{E_3}(z)\|$
decreases exponentially fast and any point of $\gamma$ has a uniform strong unstable leaf tangent to the bundle $E_3$.

For a periodic point $Q$ as above and close to $z$,
the local stable manifold  (tangent to $E^s$)
has a uniform size.
The proof of~\cite[proposition 3.7]{crovisier-palis-faible} then shows
that $W^{s}(Q)$ intersects the strong unstable leaf of some point in $\gamma\subset H$.
Since we also have $W^u(Q)\cap W^s(\cO)\neq \emptyset$ this implies that $Q$ belongs to the chain-recurrence class of $H$.
By property a), one deduces that $Q$ belongs to $H$, as wanted.
\end{proof}

We also complete remark~\ref{r.extremal}.
\begin{proposition}\label{p.R-extremal}
Let $f$ be any diffeomorphism in a dense G$_\delta$ subset $\cR''_{R}\subset \diff^1(M)$ and
$K$ be a chain-transitive set which has a partially hyperbolic structure $E^s\oplus E^c\oplus E^u$
such that $E^c$ is one-dimensional. If $K$ has type (R), then $\dim(E^s)$ and
$\dim(E^u)$ are non-zero.
\end{proposition}
\begin{proof}
By~\cite{BC} there exists a dense G$_{\delta}$ subset $\cR''_{R}\subset \diff^1(M)$
of diffeomorphisms whose periodic orbits are all hyperbolic
and dense in the chain-recurrent set.

Let us assume by contradiction that $\dim(E^u)=0$.
Consider $\varepsilon>0$ small and
a segment $\gamma$ tangent to a central cone,
contained in the chain-recurrence class of $K$ and
whose iterates are contained in an arbitrarily small neighbourhood $U$ of $K$,
and have a size smaller than $\varepsilon$.

Since $f$ belongs to $\cR''_{R}$, there exists a periodic point $z$ close to the middle of $\gamma$.
One deduces that the strong stable manifold of some point of $\gamma$
intersects $z$. Consequently the orbit $\cO$ is contained in $\overline{U}$
and has the dominated splitting $E^s\oplus E^c$.
By assumption $\cO$ is hyperbolic.
If $\cO$ is a sink, then the chain-recurrence class of $K$ contains a sink.
If $\cO$ is a saddle, the inclination lemma implies that the accumulation set of the
iterates $f^n(\gamma)$, $n\geq 0$, contains the unstable set of $\cO$.
Since the length of $f^n(\gamma)$ is bounded by $\varepsilon$, one deduces that
the length of the unstable manifolds of $\cO$ are bounded by $\varepsilon$.
The unstable manifolds of $\cO$ are thus bounded by a hyperbolic periodic orbit, which is a sink.
This proves that $\gamma$ meets the basin of a sink.
Consequently the chain-recurrence class of $K$ contains a sink.

In any cases, one deduces that the chain-recurrence class of $K$ is a hyperbolic sink.
This contradicts the fact that $K$ has type (R).
\end{proof}

\subsection{The chain-hyperbolic type (H) and the parabolic type (P)}

\begin{proposition}\label{p.CH}
Let $f$ be any diffeomorphism in a dense G$_\delta$ subset $\cR_{H}\subset \diff^1(M)$ and consider any
chain-transitive set $K$ which has a partially hyperbolic structure $E^s\oplus E^c\oplus E^u$
such that $E^c$ is one-dimensional.
If $K$ has type (H) or if $K$ is not twisted and has type (P), then, for any small neighbourhood $U$ of $K$
there exists a periodic orbit $\cO\subset U$ contained in the chain-recurrence class of $K$.

More precisely, if one considers the maximal invariant set $K'$ in a closed neighbourhood of $K$
and a plaque family $\cD^{cs}$ tangent to $E^c$ over $K'$, then
for any periodic orbit $\cO_{0}$ contained in a small neighbourhood of $K$,
the orbit $\cO$ can be chosen in such a way that
\begin{itemize}
\item[--] $\cO$ is contained in the union of the plaques $\cD^{cs}_x$ with $x\in\cO_{0}$,
\item[--] if $K$ has type (H)-attracting, (P$_{SU}$) or (P$_{SN}$),
the index of $\cO$ is equal to $dim(E^{s})+1$.
\end{itemize}
\end{proposition}

\begin{proof}
The diffeomorphisms in the set $\cR_H$ have their periodic orbits hyperbolic
and dense in the chain-recurrent set.

Let us first assume that $K$ has type (H)-attracting.
By proposition~\ref{p.extremal}, one can assume that $E^u$ is not degenerated (otherwise $K$ is reduced to a sink).
Let $\cD^{cs}$ be a plaque family tangent to $E^s\oplus E^c$ over a $K'$.
By remark~\ref{r.classification}.d), one can reduce the plaques $\cD^{cs}$ and the set $K'$,
so that the plaques are trapped: for each $x\in K'$ one has,
$$f(\overline{\cD^{cs}_x})\subset \cD^{cs}_{f(x)}.$$
In particular, if one considers two points $x,y\in K$ that are close enough, then
for any point $p$ such that $f^{-n}(p)\in \cD^{cs}_{f^{-n}(y)}$ for each $n\geq 0$,
the strong unstable manifold $W^{uu}_{loc}(p)$ meets $\cD^{cs}_x$.
By continuity this property is also satisfied if $y$ is replaced by any point $p_0\in K'$
close enough to $y$.

Let $\cO_0\subset K'$ be a periodic orbit contained in a small neighbourhood of $K$:
there exists $p_0\in \cO$ arbitrarily close to $y$, so that for each periodic point $p\in \cD^{cs}_{p_0}$,
the strong unstable manifold $W^{uu}_{loc}(p)$ meets $\cD^{cs}_x$.
Conversely, since the plaques $\cD^{cs}$ are trapped,
the strong unstable manifold $W^{uu}_{loc}(x)$ meets $\cD^{cs}_{p_0}$
at some point which belongs to the stable manifold of a periodic point $p\in \cD^{cs}_{p_{0}}$.
This implies that $p$ belongs to the chain-recurrence class of $K$.

If the index of $p$ is equal to $\dim(E^s)$, there exists another periodic point $p'\in\cD^{cs}_{p_0}$
of index $\dim(E^s)+1$ and whose stable manifold intersects the unstable manifold of $p$.
Consequently, $p'$ also belongs to the chain-recurrence class of $K$.
This proves that the chain-recurrence class contains periodic orbits $\cO$ of index $\dim(E^s)+1$
that are included in arbitrarily small neighbourhoods of $K$. Moreover $\cO$ can be obtained
inside the centre-stable plaques of any periodic orbit close enough to $K$.

This proves the proposition in the case $K$ has type (H).
In the case $K$ has type (P), the dynamics on $E^c$ is orientable.
Let us fix a central model $(K\times [0,+\infty), \hat f)$
with attracting type associated to $K$ and an invariant orientation on $E^c_K$.
By proposition~\ref{p.extremal}, one can assume again that $E^u$ is not degenerated.
Let $\cD^{cs}$ be a plaque family tangent to $E^s\oplus E^c$ over $K$
and $\cD^c\subset \cD^{cs}$ a plaque family tangent to $E^c$.
For each $x\in K$, the plaque $\cD^{cs}\setminus W^{ss}_{loc}(x)$ has two components:
one of them, $\cD^{cs,+}_x$, is tangent to the central model.
Since the central model has attracting type, one can assume (reducing the plaques)
that these half plaques are trapped: for each $x\in K$, one has
$$f(\overline{\cD^{cs,+}_x})\subset \cD^{cs,+}_{f(x)}\cup W^{ss}_{loc}(f(x)).$$
Moreover, since $K$ is not twisted, there exist two points $x,y$ arbitrarily close such that
one can connect $W^{ss}_{loc}(x)$ to $W^{uu}_{loc}(y)$ and
$W^{uu}_{loc}(x)$ to $W^{ss}_{loc}(y)$ by two curves with non-zero length,
tangent to a central cone field and having the same orientation.
Up to exchange $x$ and $y$, this implies that $W^{uu}_{loc}(y)$ intersects $ \cD^{cs,+}_x$.
By the trapping property, assuming that $x,y$ have been chosen close enough, one deduces
that for each point $y\in \cD^c\cap \cD^{cs,+}_y$
such that $f^{-n}(y)\in \cD^{cs,+}_{f^{-n}(y)}$ for each $n\geq 0$,
the strong unstable manifold $W^{uu}_{loc}(y)$ meets $\cD^{cs,+}_x$.
By continuity, this property is still satisfied if $y$ is replaced by any point $p_0\in K'$
close enough to $y$.
The end of the proof is similar to the case $K$ has type (H)-attracting by considering
the half plaques $\cD^{cs,+}$ instead of the plaques $\cD^{cs}$: one just has to note that if
$p_0\in K'$ is a periodic point, then any periodic point $p\in \cD^{cs}_{p_0}$ also belongs to $\cD^c_{p_0}$.
\end{proof}

\begin{remark}
In the proposition~\ref{p.CH}, by choosing carefully the orbit $\cO_0$,
one may replace the partial hyperbolicity on $K$
by a dominated splitting $E^s\oplus E^c\oplus E_3$ such that $E^s$ is uniformly contracted,
$E^c$ has dimension $1$  and there exists an invariant measure supported on $K$
whose Lyapunov exponents along $E_3$ are all positive (we will not use this generalisation).

Let us sketch the proof. In the case the type is (P) or (H)-repelling,
the bundle $E_3$ is uniformly expanded, the proposition~\ref{p.CH} applies.
Let us consider the type (H)-attracting. One may assume that $K$ is the support of an ergodic measure
whose Lyapunov exponents along $E_3$ are all positive.
From Ma\~n\'e's ergodic closing lemma,
$K$ is the Hausdorff limit of periodic orbits $\cO_n$ whose associated periodic measures converges towards $\mu$.
In particular, there exists $N\geq 1$ such that
for each orbit $\cO_n$ a uniform proportion of iterates $z\in \cO_n$
satisfies for any $k\geq 0$,
\begin{equation}\label{e.unif}
\prod_{i=0}^{k-1} \|Df^{-N}_{|E_3}(f^{-i.N}(z))\|\leq 2^{-k}.
\end{equation}
By passing to the limit, one can thus find
two close points $x,y\in K$ satisfying~(\ref{e.unif})
and a sequence of points $p_n\in \cO_n$ satisfying~(\ref{e.unif})
and converging to $y$.
The points $x,y,p_n$ hence have uniform unstable manifolds tangent to $E_3$
Now the proof is the same as before.
\end{remark}


\section{Partially hyperbolic sets of type (N)}
In this section $f$  is a $C^1$ generic diffeomorphism and
$K$ is a chain-transitive set with a partially hyperbolic structure
$E^s\oplus E^c\oplus E^u$ such that the bundle $E^c$ is one-dimensional and has type (N).
Our aim is to prove that if $K$ is strictly contained in its chain-recurrence class, then
either $f$ is $C^1$-approximated by diffeomorphisms having a heterodimensional cycle,
or $K$ is included in a homoclinic class having periodic
orbits whose $(dim(E^s)+1)$-th Lyapunov exponent is arbitrarily close to $0$.

\subsection{Statement and first reductions}

\begin{proposition}\label{p.N}
Let $f$ be any diffeomorphism in a dense $G_\delta$ subset $\cR_N\subset \diff^1(M)\setminus
\overline{\Tang}$ and consider any minimal set $K$ which has a partially
hyperbolic structure $E^s\oplus E^c\oplus E^u$ such that $E^c$ is one-dimensional and has type (N).
If $K$ is strictly contained in a chain-transitive set $A$, then
one of the two following situations occurs.
\begin{enumerate}
\item[1)] $K$ is contained in a homoclinic class $H$ whose index is equal to $\dim(E^s)$
or $\dim(E^s)+1$ and which contains weak periodic orbits:
for any $\delta>0$, there exists a sequence of periodic orbits
$(\cO_n)$ that are homoclinically related together, that converge for the Hausdorff topology
toward a compact subset of $A$, whose indices are equal to $\dim(E^s)$
or $\dim(E^s)+1$ and whose $(dim(E^s)+1)$-th Lyapunov exponents belong to $(-\delta,\delta)$.
\item[2)] For any neighbourhoods $U$ of $K$ and $V$ of $A$, there exists a $C^1$-perturbation
of $f$ which exhibits a heterodimensional cycle contained in $V$
and associated to periodic orbits contained in $U$ (whose indices
are equal to $\dim(E^s)$ and $\dim(E^s)+1$ respectively).
\end{enumerate}
\end{proposition}
\bigskip

The G$_\delta$ dense subset $\cR_N\subset \diff^1(M)$ will be the intersection of:
\begin{itemize}
\item[--] the set of the diffeomorphisms whose periodic orbits are all hyperbolic
(which is G$_\delta$ and dense by Kupka-Smale theorem),
\item[--] a  G$_\delta$ and dense set of diffeomorphisms whose minimal sets are limit of
hyperbolic periodic orbits for the Hausdorff topology
(provided by Pugh's closing lemma),
\item[--] a G$_\delta$ and dense set of diffeomorphisms having the following property
(see~\cite{BC}):
any chain-recurrence class that contains
a hyperbolic periodic orbit $\cO$ coincides with the homoclinic class $H(\cO)$;
any hyperbolic periodic orbits contained in a same chain-recurrence class and
having the same index are
homoclinically related,
\item[--] the sets $\cR_{Chain},\cR_{Ext}$, $\cR_{R}$, $\cR_{H}$ provided by lemma~\ref{l.connecting},
propositions~\ref{p.extension}, \ref{p.R} and~\ref{p.CH},
\item[--] two other G$_\delta$ dense sets specified during the proof
of lemmas~\ref{l.dev1} and~\ref{l.segment} stated below.
\end{itemize}
\bigskip

Before coming into the details of the proof, let us begin with some preliminary
reductions.
\begin{itemize}
\item From the hypothesis $f\in \cR_N$, the set $K$ is not a periodic orbit. Hence,
by proposition~\ref{p.extremal} the extremal bundles $E^s$, $E^u$ are both non-degenerate.

\item We consider small open neighbourhoods $U$ of $K$.
The maximal invariant set in $\overline U$ still has the partially hyperbolic splitting
$E^s\oplus E^c\oplus E^u$.
If for such arbitrarily small $U$ there exist some chain-transitive sets $K'\subset U\cap A$ having type (R), (P) or (H),
or some chain-transitive non-minimal sets $K'$ having type (N),
then the conclusion of the proposition holds immediately:
if $K'$ has type (R), (H) or if $K'$ has type (P) and is untwisted,
then by propositions~\ref{p.R} and~\ref{p.CH}
the chain-recurrence class of $K$ contains periodic orbits that are included in $U$
(hence having a central exponent along $E^c$ close to $0$ if $U$ has been chosen small enough);
if $K'$ has type (N) and is non-minimal or if $K'$ has type (P) and is twisted,
by corollaries~\ref{c.NN} and~\ref{c.twist}
one can create in $U$ a heterodimensional cycle by a $C^1$-small perturbation of $f$.

As a consequence, one can fix a small neighbourhood $U$ of $K$ and make the following isolation hypothesis:
\end{itemize}

\begin{enumerate}
\item[(I')] Any chain-transitive set $K'\subset U\cap A$ is minimal, non twisted and
has type (N) only.
\end{enumerate}

\begin{itemize}
\item By proposition~\ref{p.extension}
one can assume (up to reducing $A$ if necessary) that there exists on $A$ a dominated
splitting $E_1\oplus E^c\oplus E_3$ that extends the partially hyperbolic
structure on $K$.
One can thus fix some plaque families $\cD^{cs}, \cD^c,\cD^{cu}$ tangent to $E_1\oplus E^c$, $E^c$ and $E^c\oplus E_{3}$ over $A$.
One may take the plaques $\cD^c_x$ contained in the intersections $\cD^{cs}_x\cap \cD^{cu}_x$.

Let us consider a small neighbourhood $V_0$ of $A$.
Note that one can choose the plaque families over a set slightly larger than $A$ where the domination extends;
the plaques hence exist at any point whose orbit stays in a neighbourhood $V_0$ of $A$.
All the neighbourhoods $V$ of $A$ we will consider will be contained in $V_0$.
\end{itemize}

\subsection{Structure of stable sets}
In this section one studies the stable sets of invariant sets $K'\subset U\cap A$.
This will allow us to define four cases to be discussed in the proof.
\medskip

The first result approximates the stable sets by stable manifolds of periodic orbits.

\begin{lemma}\label{l.psl1}
Let $x\in A$ be a point whose $\omega$-limit set $\omega(x)$ is contained in $U$.
Then, for any $\varepsilon>0$, there exist a periodic orbit $\cO\subset U$
and a point $y\in W^s(\cO)$ whose forward orbit has a closure $\overline{\{f^n(y),n\geq 0\}}$
that is $\varepsilon$-close to $\overline{\{f^n(x),n\geq 0\}}$ for the Hausdorff distance.
\end{lemma}
\begin{proof}
By our assumptions (I'), the chain-transitive set $K'=\omega(x)$ has type (N).
One considers the maximal invariant set $\widehat K'$ in a small neighbourhood $\overline{U_0}$
of $K'$ and one fixes a plaque family $\widetilde{\cD}^{cs}$ that is tangent to $E^s\oplus E^c$ over $\widehat K'$
and whose plaques have a diameter $r$ smaller than $\varepsilon/2$.
Since $K'$ has type (N), by remark~\ref{r.classification}.d)
one can assume the following trapping property:
for each $x\in \widehat K'$, the closure of $f(\widetilde\cD^{cs}_x)$ is contained in $\widetilde\cD^{cs}_{f(x)}$.
One will apply the weak shadowing lemma~\ref{l.shadows}
to $r$ and to the neighbourhood $U_0$ of $K'$ and $r>0$: it allows to cs-shadows pseudo-orbits in a smaller
neighbourhood $U_1\subset U_0$ of $K'$.
One also considers $n_0\geq 0$ such that $f^n(x)$ belongs to $U_1$ for any $n\geq n_0$.

Let $\delta_0>0$ be smaller than $\varepsilon/ 2\|Df^{-k}\|$ for each $0<k\leq n_0$
and let $\varepsilon_0>0$ be the constant associated to $\delta_0$ by the shadowing lemma~\ref{l.shadows}.
One can build a $\varepsilon_0$-pseudo-orbit $\{z_n\}_{n\geq n_0}$ in $U_1$
by first following the forward orbit $\{f^n(x)\}_{n\geq n_0}$ of $x$ and then a periodic
$\varepsilon_0$-pseudo-orbit in $\omega(x)$.
One thus gets an orbit $\{f^n(y')\}_{n\geq n_0}$ in $\widehat K'$
that $(r,\delta_0)$-cs-shadows $\{z_n\}_{n\geq n_0}$:
in particular the point $f^{n_0}(x)$ is $\delta_0$-close to a point $f^{n_0}(y)\in \widetilde\cD^{cs}_{f^{n_0}(y')}$.
By the trapping property the point $f^n(y)$ still belongs to $\widetilde\cD^{cs}_{f^{n}(y')}$
for $n\geq n_0$.

By our choice of $\delta_0$, the points $f^n(x)$ and $f^n(y)$ are $\varepsilon/2$-close
also for any $0\leq n \leq n_0$.
By having chosen $\delta_0$ small enough, one can assume that the forward iterates of $y$
and $x$ remain arbitrarily close during an arbitrarily long time $n_1$.
After time $n_1$, the plaque $\widetilde \cD^{cs}_{f^n(y')}$ contains both $f^n(y)$
and a point which is $\delta_0$-close to $K'$.
One deduces that $f^n(y)$ belongs to the $(\delta_0+\varepsilon/2)$-neighbourhood of $K'=\omega(x)$.
Hence the sets $\overline{\{f^n(x),n\geq 0\}}$
and $\overline{\{f^n(y),n\geq 0\}}$ are $\varepsilon$-close for the Hausdorff distance.

For $n_2$ large the sequence $(z_n)_{n\geq n_2}$ is periodic
and $(r,\delta_0)$-cs-shadowed by $(f^n(y))_{n\geq n_2}$.
Hence by the weak shadowing lemma
there exists a periodic point $p$ such that $f^n(y)$ belongs to $\widetilde \cD^{cs}_{f^n(p)}$
for each $n\geq n_2$. By the trapping property the point $f^n(y)$ belongs to the stable set of
a periodic point $q\in \widetilde \cD^{cs}_{f^n(p)}$.
This concludes the proof.
\end{proof}
\bigskip

\noindent
This lemma justifies the following definitions.
\begin{definition}\label{d.approx}
Let $x\in A$ such that $\omega(x)\subset U$.
\smallskip

\noindent -- We say that $x$ is \emph{approximated by strong stable manifolds}
if for any $\varepsilon>0$, there exist a periodic orbit $\cO\subset U$
and a point $y\in W^{ss}(\cO)$ such that $\overline{\{f^n(y),n\geq 0\}}$
is $\varepsilon$-close to $\overline{\{f^n(x),n\geq 0\}}$ for the Hausdorff distance.
\smallskip

\noindent -- We say that $x$ is \emph{approximated by large stable manifolds}
if for any $\eta>0$,
there exists $\rho>0$ and for any $\varepsilon>0$ there exists a periodic orbit
$\cO\subset U$ and a point $y\in W^s(\cO)$ such that $\overline{\{f^n(y),n\geq 0\}}$
is $\varepsilon$-close to $\overline{\{f^n(x),n\geq 0\}}$ for the Hausdorff distance,
and there exists a $(\dim(E^s)+1)$-dimensional disc of radius $\rho$ centred at $y$
which belongs to the stable set of $\cO$ and whose forward iterates have a diameter smaller than $\eta$.
\end{definition}
\smallskip

\begin{lemma}\label{l.dicho}
Any point $x\in A$ such that $\omega(x)\subset U$
is either approximated by strong stable manifolds or by large stable manifolds.
\end{lemma}
\begin{proof}
By lemma~\ref{l.psl1}, there exists some points $y\in M$ such that $\overline{\{f^n(y),n\geq 0\}}$
is arbitrarily close to $\overline{\{f^n(x),n\geq 0\}}$ for the Hausdorff distance
and $y$ belongs to the stable set of a periodic orbit $\cO$. We fix some $\eta>0$.

The set $\omega(x)$ has type (N), hence by the trapping property there exists $\rho_{0}>0$
such that all the forward iterates of the $\rho_{0}$-neighbourhood of $x$ in $\cD^{cs}_{x}$
have a diameter smaller than $\eta/2$.

By continuity and trapping, the same property holds for any point $y\in W^s(\cO)$ such that $\overline{\{f^n(y),n\geq 0\}}$
is close to $\overline{\{f^n(x),n\geq 0\}}$ for the Hausdorff distance:
the plaque family $\cD^{cs}$ is defined above the periodic orbit $\cO$ and for each $n\geq 0$
there exists a $(\dim(E^s)+1)$-dimensional disc $D_{n}$ centred at $f^n(y)$ whose large iterates $f^m(D_{n})$ all belong to the centre-stable plaque of $\cO$
and have a diameter bounded by $\eta$.
For $n$ larger than some $n_{0}$, the disc $D_{n}$ can be chosen with a uniform radius $\delta$.
One then choose $\rho$ such that the image of any disc of radius $\rho$ by $f^{n_{0}}$ has diameter bounded by $\delta$:
the preimage $f^{-n_{0}}(D_{n_{0}})$ contains a disc $D_{y}$ centred at $y$ of uniform radius $\rho$ and whose forward iterates
have a diameter smaller than $\max\{\varepsilon/2,\eta\}$.
Note that any point in a disc $D_{y}$ accumulates by forward iterations inside the union of the plaques $\cD^{c}$ over $\cO$;
hence any point in $D_{y}$ belong to the stable manifold of a periodic orbit.

Two cases are now possible.
\begin{itemize}
\item[--] For each point $y$ the disc $D_{y}$ at $y$ is contained in the stable set of $\cO$.
By definition $x$ is then approximated by large stable manifolds.
\item[--] For some point $y$ there exist in $D_{y}$ two points that belong to the stable set of distinct
periodic orbits. The complement of the strong stable sets of the periodic orbits is open in $D_{y}$
and not connected, hence $D_{y}$ contains a point $y'$ which belongs to the strong stable
set of a periodic orbit $\cO'$. By construction $\overline{\{f^n(y'),n\geq 0\}}$
is $\varepsilon/2$-close to $\overline{\{f^n(y),n\geq 0\}}$ hence $\varepsilon$-close to $\overline{\{f^n(x),n\geq 0\}}$.
\end{itemize}
The result follows.
\end{proof}
\bigskip

We now describe the geometry of the stable set of points $x\in A$ such that $\overline {\{f^n(x),n\geq 0\}}\subset U$
inside the plaque $\cD^{cs}_x$.
Since $K$ has type (N), remark~\ref{r.classification}.d) applies.
Hence, if the plaques of $\cD^{cs}$ and the neighbourhood $U$
have been chosen small enough, the forward iterates of the plaque $\cD^{cs}_x$
have a small diameter; consequently the plaques $\cD^{cs}_x$ are foliated by strong
stable leaves $W^{ss}_{loc}(z)$ (tangent to the bundle $E^{ss}$).
By reducing the plaques of $\cD^{cs}$ if necessary one also deduces that
$\cD^c_x$ parametrises the strong stable leaves of $\cD^{cs}_x$.

\begin{definition}
The \emph{local stable set} of a point $x$ such that $\overline {\{f^n(x),n\geq 0\}}\subset U$ is the set:
$$W^s_{loc}(x)=\left\{\zeta,\; f^n(\zeta)\in \cD^{cs}_{f^n(x)} \text{ for every $n\geq 0$ and }
d(f^n(x),f^n(z))\underset{n\to +\infty}{\longrightarrow} 0\right\}.$$
\end{definition}
\noindent Its trace on the central plaque will be denoted by
$$I^s_{x}=W^s_{loc}(x)\cap \cD^c_{x}.$$
By the previous discussion, $I_x^s$ is a subinterval in $\cD^c_x$ containing $x$ such that
$$ W^s_{loc}(x)=\bigcup_{\zeta\in I^s_x}W^{ss}_{loc}(\zeta).$$
Hence $I^s_x$ is reduced to $\{x\}$ if the stable set of $x$
is trivial in the central direction.

Note that by reducing $U$, the trapping properties along the plaques $\cD^c$ implies that
the length of $I^s_{x}$ can be assumed to be bounded by a small constant. In particular
$I^s_x$ and $\cD^{c}_x$ do not coincide and one has
\begin{equation}\label{e.stable}
f(I^s_x)=I^s_{f(x)}.
\end{equation}
One may consider the two \emph{endpoints}
of $I^s_x$ (i.e. the points in $\overline{I^s_x}\cap \overline{\cD^s_x\setminus I^s_x}$).
The following result shows that if $x$ is an endpoint of $I^s_{x}$ then it is approximated by strong stable manifolds.

\begin{lemma}\label{l.stable-controle}
Let $x\in A$ be a point such that $\overline {\{f^n(x),n\geq 0\}}\subset U$
and assume that $x$ is an endpoint of the curve $I^s_{x}$.
Then for any $\varepsilon>0$, there exists a neighbourhood $U'$ of $\omega(x)$ with the following property.
For any periodic orbit $\cO$ contained in $U'$, there exists a point $y\in W^{ss}(\cO)$
whose forward orbit has a closure $\overline{\{f^n(y),n\geq 0\}}$
that is $\varepsilon$-close to $\overline{\{f^n(x'),n\geq 0\}}$ for the Hausdorff distance.
\end{lemma}
\begin{proof}
One first parametrises the plaque $\cD^c_x$ centred at $x$ and identifies it to the interval
$[-1,1]_x$. Since $x$ is an endpoint of $I^s_x$
one can assume that $I^s_x$ is disjoint from $(0,1]_x$.

Let us choose $\eta>0$ small.
We claim that \emph{there exists $\delta>0$ such that
all the forward iterates of $[0,\eta]_x$ have a length larger than $\delta>0$.}
Indeed, for a central model of $\omega(x)$,
consider some neighbourhood $B$ of the section $0$ that is attracting for $\hat f^{-1}$:
by continuity, one may extend $B$ so that for any point $z$ whose forward orbit remains close to $\omega(x)$
one still has $f(B\cap \cD^c_{z})\supset B\cap \cD_{f(z)}^c$.
For $n_0$ large enough, the forward orbit of $f^n(x)$, $n\geq n_0$,
stays in a small neighbourhood of $\omega(x)$. Hence two cases appear.
\begin{itemize}
\item[--] There exists $n\geq n_0$
such that $f^n([0,\eta]_x)\subset \cD^c_{f^n(x)}$ meets the complement of $B$.
In this case, all the larger iterates are also trapped by $B$ and
the lengths $f^n([0,\eta]_x)$ are bounded from below by the width of $B$, as claimed.
\item[--] The length of $f^n([0,\eta]_x)$ for $n$ large is smaller that the width of $B$.
This case can not occur for any $B$ inside a decreasing sequence of trapping neighbourhoods, since
this would imply that $|f^n([0,\eta]_x)|\to 0$ as $n\to +\infty$ and contradict the fact that
$(0,\eta]_x$ is disjoint from $I^s_x$.
\end{itemize}
We thus obtain the claim.

From the claim, one deduces that
for each $z\in \omega(x)$, the limit set of the curves $f^n([0,\eta]_x)$ contains
an interval $J_z\subset \cD^c_z$ of length larger than $\delta$ and that contains $z$.
Since $\omega(x)$ is a non-twisted minimal set (by assumption (I')),
there exists $z\neq z'$ in $\omega(x)$ such that $W^{ss}_{loc}(z')$ intersects
$\bigcup_{\zeta\in J_z}W^{uu}_{loc}(\zeta)$.
Let $\cO$ be a periodic orbit close to $\omega(x)$: it contains a point $p$ arbitrarily close to $z'$.
There exists an iterate $f^n(x)$ close to $z$, so that $W^{ss}_{loc}(p)$ meets
$\bigcup_{\zeta\in f^n([0,\eta]_x)}W^{uu}_{loc}(\zeta)$ at some point $f^n(y)$.

By construction the iterates $(y,\dots,f^n(y))$ belong to the plaques $\cD^{cu}$ along the
iterates $(x,\dots,f^n(x))$ and the iterates $f^k(y)$ for $k\geq n$
belong to the local strong stable manifold of $\cO$.
Since the plaques and the local invariant manifolds may be assumed to have an arbitrarily
small diameter, this implies that the set $\overline{\{f^n(y),n\geq 0\}}$
can be chosen arbitrarily close to $\overline{\{f^n(x),n\geq 0\}}$ for the Hausdorff distance.
\end{proof}
\bigskip

One considers similarly the $y\in A$ such that $\overline {\{f^{-n}(y),n\geq 0\}}\subset U$:
they are approximated by strong unstable manifolds or by large unstable manifolds.
One defines the set $I^u_{y}\subset \cD^c_{x}$.
\bigskip

\paragraph{Discussion.}
In the following we continue the proof of proposition~\ref{p.N}.
By lemmas~\ref{l.dicho} and~\ref{l.stable-controle},
the following cases should be considered.
\begin{enumerate}
\item[1.] There is no $x,y\in A$ satisfying $x\not\in \omega(x)$, $y\not\in \alpha(y)$
and $\omega(x)=\alpha(y)\subset U$.
\item[2.] There exist $x,y\in A$ satisfying $x\not\in \omega(x)$, $y\not\in \alpha(y)$
and $\omega(x)=\alpha(y)\subset U$. Taking iterates, one can assume that
$\overline{\{f^n(x),n\geq 0\}}$ and $\overline{\{f^{-n}(y),n\geq 0\}}$ are contained in $U$.
We introduce three subcases.
\begin{enumerate}
\item[a)] $x$ is approximated by strong stable manifolds and $y$ is an endpoint of $I^u_y$.\\
(Or $x$ is an endpoint of $I^s_{x}$ and $y$ is approximated by strong stable manifolds.)
\item[b)] $x$ is approximated by large stable manifolds and $y$ is an endpoint of $I^u_{y}$.\\
(Or $x$ is an endpoint of $I^s_{x}$ and $y$ is approximated by large stable manifolds.)
\item[c)] $x$ is not an endpoint of $I^s_{x}$ and $y$ is not an endpoint of $I^u_{y}$.
\end{enumerate}
\end{enumerate}

\subsection{Trapped central behaviour}

Let us conclude the proof of proposition~\ref{p.N} in the cases 1 and 2.a.
One will show that the dynamics along the centre plaques has bounded deviations,
so that one can hope to recover the argument of
corollary~\ref{c.NN} and create a heterodimensional cycle.

We first deal with case 1.
\begin{lemma}\label{l.1}
Let us suppose that there is no points $x,y\in A$
satisfying $x\not\in \omega(x)$, $y\notin\alpha(y)$
and $\omega(x)=\alpha(y)\subset U$.
Then, one can create in each neighbourhood $V$ of $A$ a heterodimensional cycle associated to
a periodic orbit contained in $U$ by a $C^1$-small perturbation of $f$.
\end{lemma}
\begin{proof}
We first prove that $U\cap A$ contain non-recurrent dynamics.
\begin{claim}\label{c.1}
Let $K'\subset A\cap U$ be a minimal compact set and $U'\subset U$ a neighbourhood of $K'$.
There exists a point $z$ such that
\begin{itemize}
\item[--] $\overline{\{f^n(z),n\in \ZZ\}}\subset U'\cap A$ and $z\not\in \alpha(z)\cup \omega(z)$;
\item[--] for any neighbourhood $U_\alpha$ of $\alpha(z)$, there exists a neighbourhood $U_\omega$
of $\omega(z)$ such that for any periodic orbit $\cO\subset U_\omega$, there exists a point
$y\in U_{\alpha}\cap W^{ss}(\cO)$ whose forward orbit is contained in $U'$.
\end{itemize}
\end{claim}
\begin{proof}
Let us apply lemma~\ref{l.isolation} to the set $K'$:
there exists a point $x^+\not\in K'$ and an invariant compact set $\Delta^+\subset U'\cap A$
containing $K'\cup \omega(x^+)$.
If $\Delta^+$ contains a point $z$ such that $z\not\in \alpha(z)\cup \omega(z)$
we are done, otherwise the set $\Delta^+$ is chain-transitive.
In this second case by the assumption (I') one has $K'=\omega(x^+)=\Delta^+$.
Similarly there exists a point $x^-\not\in K'$ and an invariant compact set $\Delta^-\subset U'\cap A$
containing $K'\cup \alpha(x^-)$.
Either it contains a point $z$ such that $z\not\in \alpha(z)\cup \omega(z)$
or $K'=\alpha(x^-)=\Delta^-$.
But by assumption $\omega(x^+)$ and $\alpha(x^-)$ cannot both coincide with $K'$.
This gives the first item of the claim.

We then prove that the length of $f^{-n}(I^s_z)$ goes to zero as $n\to +\infty$.
By~(\ref{e.stable}), one has
$$f^{-n}(I^s_z)=I^s_{f^{-n}(z)}.$$
If one assumes that the length of $f^{-n}(I^s_z)$ does not decrease to zero, one finds
an accumulation point $x\in \alpha(z)$ and an interval $\gamma\subset \cD^c_x$
such that $f^n(\gamma)\subset \cD^c_{f^n(x)}$ for each $n\geq 0$ and $\gamma$ is contained in
the chain-stable set of the chain-recurrence class containing $K$.
Since $\alpha(z)$ has type (N) by assumption (I'),
$\alpha(z)$ is a minimal set whose central dynamics is trapped, so
at any point $x\in \alpha(z)$ such a segment $\gamma$ exists and has a uniform size.
This implies that $\alpha(z)$ has type (P$_{SN}$), contradicting our assumption (I').

We end with the proof of the second item.
Let us fix a neighbourhood $U_\alpha$ of $\alpha(z)$.
Since the length of $f^{-n}(I^s_z)$ goes to zero,
there exists an iterate $f^{-n}(z)$ such that
$I^s_{f^{-n}(z)}$ is contained in $U_\alpha$.
Lemma~\ref{l.stable-controle} then gives a neighbourhood $U_\omega$ of $\omega(z)$ such that
for any periodic orbit $\cO\subset U_\omega$ there is a point $y\in U_\alpha\cap W^{ss}(\cO)$
whose forward orbit is contained in a small neighbourhood of the segments $f^{k-n}(I^s_z)$, $k\geq 0$.

The second item now follows if one proves that all the segments $f^{k}(I^s_z)$, $k\geq 0$
are contained in $U'$: this will be ensured by choosing the point $z$ carefully.
Note first that the construction of $z$ can be performed in any small neighbourhood of $K'$:
one gets a sequence $(z_k)$ such that
the sets $\overline{\{f^n(z_k),n\in \ZZ\}}$ converge to $K'$ for the Hausdorff topology.
Let us assume that the supremum of the lengths $f^n(I^s_{z_k})$ for $n\in \ZZ$
does not goes to zero when $k$ goes to infinity: a sequence of arcs $f^n(I^s_{z_k})$
converges toward a non-trivial segment $\gamma$ contained in a central plaque $\cD^c_x$
for some $x\in K'$ and contained
in the chain-recurrence class of $K$. Hence $K'$ has type (R),
contradicting the assumption (I').
One thus get the required property by choosing $z=z_k$ for $k$ large.
\end{proof}

One can now conclude the proof of the lemma.
The claim~\ref{c.1} applied to $(K,U)$ provides us with a point $z$.
One chooses $\zeta^+\alpha(z)$.
Let us consider any neighbourhood $\cU$ of $f$ in $\diff^1(M)$.
By the connecting lemma (recalled in section~\ref{ss.spread}), one obtains an integer $N\geq 1$ and
arbitrarily small neighbourhoods $W^+\subset \widehat W^+$ of $\zeta^+$.
One can require that the $N$ first iterates of $\widehat W^+$ are disjoint and have closures
disjoint from $\omega(z)$.
Since $\alpha(z)$ is minimal one can find a neighbourhood $U_\alpha$ such that
any orbit that has a point in $U_\alpha$ meets $\widehat W^+$.
One then applies the second item of claim~\ref{c.1} and introduces a neighbourhood $U_\omega$
or $\omega(z)$: for any periodic orbit $\cO$ contained in $U_\omega$
the strong stable set $W^{ss}(\cO)$ meets $\widehat W^+$.

Let us now apply claim~\ref{c.1} to $(\omega(z),U_\omega)$.
One obtains a point $z'$ such that $\overline{\{f^n(z'),n\in \ZZ\}}$
is contained in $U_\omega$ and $z'\not\in \alpha(z')\cup\omega(z')$.
One then chooses two small neighbourhoods $W^-\subset \widehat W^-$
of a point $\zeta^-\in \omega(z')$ and whose $N$ first backward iterates are disjoint and have closures
disjoint from $\alpha(z')$ and are contained in $U_\omega$.
In particular, the iterates of $f^i(\widehat W^+)$ and $f^{-j}(\widehat W^-)$ for $0\leq i,j\leq N$
can be chosen disjoint with closures disjoint from $\alpha(z')$.
By Pugh's closing lemma and our choice of $\cR_N$, there exists a periodic orbit $\cO$ close to the minimal set $\alpha(z')$
for the Hausdorff topology; its strong stable set meets $\widehat W^+$
at a point $x^+$ and its strong unstable set meets $\widehat W^-$ at a point $x^-$.
The forward orbit of $x^+$ is contained in $U$ and the backward orbit of $x^-$ is contained in $U_\omega$.
In particular the iterates $f^i(\widehat W^+)$, for $ 0\leq i\leq N$ are disjoint from the backward orbit
of $x^-$.

Since $f$ belongs to $\cR_N\subset \cR_{Chain}$,
lemma~\ref{l.connecting} is satisfied and there exists a segment of orbit $(x_0,\dots,f^\ell(x_0))$ contained in $V$
such that $x_0$ belongs to $W^-$ and $f^\ell(x_0)$ to $W^+$.
Arguing as in the proof of proposition~\ref{p.extension},
one can then connects the backward orbit $\{f^{-i}(x^-), i\geq m^-\}$
to the forward orbit $\{f^i(x^+),i\geq m^+\}$, for $m^�,m^+$ large,
using the finite orbit $(x_0,\dots,f^\ell(x_0))$
by a perturbation $g_0\in\cU$ supported in the iterates $f^k(\widehat W^+)$
and $f^{-k}(\widehat W^-)$ for $0\leq k\leq N$.
Hence the periodic orbit $\cO$ has a strong homoclinic intersection.
Since $\cO$ is contained in an arbitrarily small neighbourhood of the set $\alpha(z')$
of type (N), its centre Lyapunov exponent is arbitrarily close to zero
and by proposition~\ref{p.strong-cycle},
one can by $C^1$-small perturbation $g$ of $g_0$ create a heterodimensional cycle
associated to periodic orbits contained in $U$.
By construction the cycle is contained in the union of the support of the perturbation with
a neighbourhood of the backward orbit of $x^-$, the forward orbit of $x^+$ and
the segment of orbit $(x_0,\dots,f^\ell(x_0))$. In particular, it is contained in $V$.
\end{proof}
\medskip

The case 2.a is even simpler.

\begin{lemma}
Let $K'\subset U\cap A$ be an invariant compact set and
let $x,y\in A\setminus K'$ such that $\omega(x)=\alpha(y)=K'$
and $\overline{\{f^{-n}(y), n\geq 0\}} \subset U$.

If $x$ is approximated by strong stable manifolds and $y$ is an endpoint of $I^u_y$,
then 
by a $C^1$-small perturbation of $f$ one can create in each neighbourhood $V$ of $A$ a heterodimensional cycle associated to a periodic orbit contained in $U$.
\end{lemma}
\begin{proof}
Let $\cU$ be a $C^1$-neighbourhood of $f$.
The connecting lemma (recalled in section~\ref{ss.spread})
associates an integer $N\geq 1$ and some arbitrarily small neighbourhoods
$W^+\subset\widehat W^+$ of $x$ and $W^-\subset \widehat W^-$ of $y$.
One may assume that the iterates $f^i(\widehat W^+)$ and $f^{-j}(\widehat W^-)$
for $0\leq i,j\leq N$ are pairwise disjoint and have closures disjoint from $K'$.
Moreover the iterates $f^i(\widehat W^+)$ for $0\leq i\leq N$ are disjoint from the
backward orbit of $y$.
Since $f$ belongs to $\cR_N\subset \cR_{Chain}$,
lemma~\ref{l.connecting} is satisfied and there exists a segment of orbit $(x_0,\dots,f^\ell(x_0))$ contained in $V$
such that $x_0$ belongs to $W^-$ and $f^\ell(x_0)$ to $W^+$.

Since $x$ is approximated by strong stable manifolds, there exists a periodic orbit $\cO$
arbitrarily close to $K'$ for the Hausdorff topology whose strong stable manifold meets $\widehat W^+$
at some point $x^+$. Since $y$ is an endpoint of $I^u_y$
and by lemma~\ref{l.stable-controle}, the strong unstable manifold of $\cO$
meets $W^-$ at a point $x^-$,
provided $\cO$ has been chosen close enough to $K'$.
The sets $\overline{\{f^{-k}(x^-),k\geq 0\}}$
and $\overline{\{f^{-k}(y),k\geq 0\}}$ are close for the Hausdorff topology.
One thus deduces that the iterates $f^i(\widehat W^+)$ for $0\leq i\leq N$ are disjoint from the
backward orbit of $x^-$.

Arguing as in the proofs of proposition~\ref{p.extension} and lemma~\ref{l.1},
one concludes that there exists a perturbation $g\in \cU$ of $f$ with support 
in the union of the $f^i(\widehat W^+)\cup f^{-i}(\widehat W^-)$, for $0\leq i\leq N$
such that $\cO$ has a strong homoclinic intersection.
Since $\cO$ is close to $K$, its central exponent is close to zero
and from proposition~\ref{p.strong-cycle} one can obtain a heterodimensional cycle by another $C^1$-perturbation,
associated to periodic orbits close to $K'$ (hence contained in $U$).
By construction the cycle is contained in $V$ as required.
\end{proof}

\subsection{Large stable set and trapped unstable set}
In case 2.b, the dynamics can be approximated by periodic orbits
having an iterate whose local central manifold has a uniform size.
Hence, one can hope to argue as in the chain-hyperbolic case
and prove that the chain-recurrence class of $K$
contains hyperbolic periodic orbits whose central exponent is weak.

\begin{lemma}\label{l.dev1}
Let $K'\subset U\cap A$ be an invariant compact set and
let $x,y\in A\setminus K'$ such that $\omega(x)=\alpha(y)=K'$ and $\overline{\{f^{-n}(y), n\geq 0\}}
\subset U$.

If $x$ is approximated by large stable manifolds and $y$ is an endpoint of $I^u_{y}$,
then for any $\delta,\eta>0$, there exists $\rho>0$ and a sequence of hyperbolic periodic orbits $(\cO_n)$
with the following properties.
\begin{enumerate}
\item The sequence $(\cO_n)$ converges for the Hausdorff topology toward a
compact set $\Lambda\subset A$.
\item The $(\dim(E^s)+1)$-th Lyapunov exponent of each orbit $\cO_n$
belongs to $(-\delta,0)$. In particular the index of $\cO_n$ is $\dim(E^s)+1$.
\item In each orbit $\cO_n$ there exists a point $p$ whose stable manifold
contains an immersed ball $B_p$ centred at $p$ of radius $\rho$
and whose forward iterates have radius bounded by $\eta$.
\end{enumerate}
\end{lemma}
\begin{proof}
Let $\cS=\{1/n, n> 0\}$ and let $(V_n)_{n\geq 0}$ be a countable family of open sets such that
for any compact set $A'$ and any neighbourhood
$V'$ of $A'$, there exists $V_n$ in the family satisfying $A'\subset V_n\subset V'$.
For $\delta,\eta,\rho\in \cS$, let us denote by
$\cU(V_n,\delta,\eta,\rho)$ the set of diffeomorphisms which possess a hyperbolic periodic orbit contained in $V_n$
satisfying the items~2 and~3 above. This set is open.
Let $\cR_{2.b}$ be the intersection of the $\cU(V_n,\delta,\eta,\rho)\cup \left(\diff^ 1(M)
\setminus \overline{\cU(V_n,\delta,\eta,\rho)}\right)$ over
the $4$-uples $(V_n,\delta,\eta,\rho)$. This is a G$_\delta$ dense subset of $\diff^1(M)$.
Since $f$ is $C^1$ generic one can assume that it belongs to $\cR_{2.b}$.

In the following, one fixes $\delta,\eta\in \cS$ and one associates $\rho\in \cS$
which satisfies the definition~\ref{d.approx} of approximation by large stable manifold
for $x$ and $\eta$. The following claim gives for each $V_n$ a periodic orbit $\cO_n\subset V_n$ that
satisfies the items~2 and~3; in particular it implies the lemma.
\begin{claim}
For each $V_n$ containing $A$, the diffeomorphism $f$ belongs to $\cU(V_n,\delta,\eta,\rho)$.
\end{claim}

The definition~\ref{d.approx} provides us with a sequence $(\widetilde \cO_k)$
of periodic orbit that converges towards $K'$ and a sequence of points $x_k$
in the stables sets $W^{s}(\widetilde \cO_k)$ that converges toward $x$ such that
the ball $D_k$ centred at $x_k$ and of radius $\rho$ is contained in the stable manifold of $\widetilde \cO_k$ and its forward iterates have a radius bounded by $\eta$.

We note that the diameter of the disks
$(f^n(D_k))_{n\geq 0}$ decreases uniformly to zero as $n\to \infty$.
Indeed if this were not satisfied, a sub-sequence of disks $(f^n(D_k))$
would converge toward a non-trivial segment $\gamma$ contained in the central plaque of a point of $K'$;
since $\gamma$ is contained in the chain-stable set of $K'$, this would contradict the fact that $K'$
has type (I') only. This gives the property.

We continue with the proof of the claim.
In order to prove that $f$ belongs to $\cU(V_n,\delta,\eta,\rho)$
it is enough to show that $f\in \cR_{2.b}$ is limit of diffeomorphisms in $\cU(V_n,\delta,\eta,\rho)$.
We thus consider a neighbourhood $\cU$ of $f$ in $\diff^1(M)$
and we have to show that $\cU\cap \cU(V_n,\delta,\eta,\rho)$ is non-empty.

One associates to $(f,\cU)$ by the connecting lemma an integer $N\geq 1$ and
some small neighbourhoods $W^+\subset \widehat W^+$ and $W^-\subset \widehat W^-$
of $f^{-N}(x)$ and $y$ respectively such that the iterates
$f^i(\widehat W^+)$ and $f^{-j}(\widehat W^-)$ for $0\leq i,j\leq N$ are pairwise disjoint
and have closures that are disjoint from the set $K'$ and the
orbits $\{f^{k}(x),k>0\}$ and $\{f^{-k-N}(y),k>0\}$.
Recall that the iterates $f^n(D_k)$ considered above have a diameter that decrease uniformly to zero as
$n$ goes to $+\infty$ and accumulate on the set $K'$.
Hence, if $\widehat W^+$ has a small diameter, it is disjoint from all the $f^n(D_k)$, $n,k\geq 0$,
provided $k$ is large enough.

Since $f$ belongs to $\cR_N\subset \cR_{Chain}$,
lemma~\ref{l.connecting} is satisfied and there exists a segment of orbit $(x_0,\dots,f^\ell(x_0))$ contained in $V_n$
such that $x_0$ belongs to $W^-$ and $f^\ell(x_0)$ to $W^+$.
Since $x$ is approximated by large stable manifolds
and $y$ is an endpoint of $I^u_{y}$,
there exists (among the family $(\widetilde \cO_k)$) a periodic orbit $\widetilde \cO$ of index $\dim(E^1)+1$
close to $K'$ and two points $x^+\in W^s(\widetilde \cO)\cap W^+$ and
$x^-\in W^{uu}(\widetilde \cO)\cap W^-$
such that
\begin{itemize}
\item[--] $\overline{\{f^n(x^+),n\geq 0\}}$ is close to $\overline{\{f^n(x),n\geq 0\}}$
and $\overline{\{f^{-n}(x^-),n\geq 0\}}$ is close to $\overline{\{f^{-n}(y),n\geq 0\}}$
for the Hausdorff distance;
\item[--] the stable manifold of $\widetilde \cO$ contains a ball of radius $\rho$ centred at $x^+$
whose forward iterates have a diameter less than $\eta$.
\end{itemize}

Arguing as in the proofs of proposition~\ref{p.extension} and lemma~\ref{l.1},
one concludes that there exists a perturbation $g\in \cU$ of $f$ with support 
in the union of the $f^i(\widehat W^+)\cup f^{-i}(\widehat W^-)$, for $0\leq i\leq N$
such that $\widetilde \cO$ has a homoclinic intersection at $x$.
Moreover the neighbourhood $\widehat W^+$ has been chosen small enough to be disjoint
from the forward iterates of the disk centred at $x^+$ of radius $\rho$ in $W^s(\widetilde \cO)$.
In particular the stable manifold of $\widetilde \cO$ for $g$
still contains a disk $B_{x^+}$ of radius $\rho$ centred at $x^+$ and whose forward iterates 
have a diameter smaller than $\eta$.

By arbitrarily small perturbation, one can assume that the homoclinic intersection at $x^+$
is transversal. As a consequence,
there exists a transitive hyperbolic set containing $\widetilde \cO$
and the orbit of $x^+$. One deduces that there exists a periodic orbit $\cO_n$
whose central Lyapunov exponent is close to the exponent of $\widetilde \cO$,
and a point $p\in \cO_n$ whose local stable manifold is close to the stable set of $x^+$:
in particular, it contains a ball $B_p$ of radius $\rho$.
By hyperbolicity, there exists $n_0$ such that the radius of $f^n(B_p)$ is smaller than $\eta$ for any $n\geq n_0$.
On the other hand, if $p$ has been chosen close enough to $x^+$,
the diameter of $f^n(B_p)$ for $0\leq n\leq n_0$ is close to the diameter
of $B_{x^+}$, hence have a radius smaller than $\eta$.
Consequently all the forward iterates
of have a diameter smaller than $\eta$, as required.

Since $\widetilde \cO$ is close to $K$, its central exponent is close to zero
and belongs to $(-\delta,0)$.
This ends the proof of the claim and of the lemma.
\end{proof}
\medskip

The following result asserts that the periodic orbits provided by lemma~\ref{l.dev1}
are homoclinically related together. This concludes the proof of proposition~\ref{p.N}
in the case 2.b.

\begin{lemma}
Let us fix $\eta$ and $\delta$ small.
Let $(\cO_n)$ be a sequence of hyperbolic periodic orbits satisfying properties 1, 2 and 3 of lemma~\ref{l.dev1}.
Then the orbits $\cO_n$ from an extracted subsequence are homoclinically related together.
\end{lemma}
\begin{proof}
Let us define $\lambda\geq 1$ as the supremum of $\displaystyle \sup_{z,z'\in \cD^ c_x}\frac{\|Df^ {-1}_{|E^c}(z) \|}{\|Df^ {-1}_{|E^c}(z')\|}$
over the $x\in A$. By choosing the plaques $\cD^ c$ small enough one can
ensure that $\lambda \leq e^{\delta}$.
If the constant $\eta$ that appear in
lemma~\ref{l.dev1} is small enough, one can assume that
$\eta$ and $\eta \|Df \|$ are smaller than the infimum of the radii of the plaques $\cD^c$.

Consider any orbit $\cO_n$.
Since the Lyapunov exponent along $E^c$ belongs to $(-\delta,0)$,
by Pliss lemma there exists $q\in \cO_n$
such that for each $k\geq 0$ one has
\begin{equation}
\label{e.exp}
\|Df^{-k}_{|E^c}(q)\|\leq e^{k\delta}.
\end{equation}

For each point $z\in \cO_{n}$ let us define $\Gamma_{z}\subset\cD^c_{z}$ as the maximal interval
containing $z$ such that $f^k(\Gamma_{z})\subset \cD^c_{z}$ for each $k\geq 0$ and $\Gamma_{z}$ is contained in
the stable set of $\cO_{n}$.
By property~3 of lemma~\ref{l.dev1} and our choice of $\eta$, there exists $p\in \cO_n$ such that:
\begin{description}
\item[(*)] $\Gamma_p$ contains the $\rho$-neighbourhood of $p$ in $\cD^c_{z}$.
\end{description}
Let us write $p=f^{\ell}(q)$ where $\ell\geq 0$ is strictly smaller than the period of $\cO_n$.
One can choose $p$ so that it is the only iterate $f^{k}(q)$ with $0\leq k\leq \ell$
which satisfies property (*).
We claim that for each $k\geq 0$ one has
\begin{equation}
\label{e.exp2}
\|Df^{-k}_{|E^c}(p)\|\leq e^{k\delta}.
\end{equation}

By our choice of $p$, for any iterate $f^{-k}(p)$ with $0< k \leq \ell$ one of the branches $J$ of $\Gamma_{f^ {-k}(p)}\setminus \{f^{-k}(p)\}$
has size smaller than $\rho$. Since $\rho\|Df\|$
is smaller than the radius of the plaque $\cD^c$,
this implies that $f(J)$ is contained in
$\cD^c_{f^{-k+1}(p)}$, hence coincides with one branch of $\Gamma_{f^{-k+1}(p)}\setminus \{f^{-k+1}(p)\}$.
One can thus define a sequence $j_{0}=0<j_{1}<\dots<j_{s}=\ell$ having the following property:
\begin{description}
\item[\quad\quad]
For each $0\leq i < s$
there is a branch $J_{i}$ of $\Gamma_{f^{-j_{i}}(p)}\setminus \{f^{-j_{i}}(p)\}$ 
of size larger than $\rho$ and such that all the iterates $f^{-k}(J_{i})$ with $1\leq k\leq j_{i+1}-j_{i}$ have a length smaller than $\rho$.
\end{description}
In particular this implies that $\|Df^{-k}_{|E^c}(f^{-j_{i}}(p)) \|$ with $1\leq k\leq j_{i+1}-j_{i}$
is bounded by $\lambda^k$. By combining the estimates over the different $j_i$,
one thus gets for each $0\leq k \leq \ell$,
$$\|Df^{-k}_{|E^c}(p) \|\leq \lambda^k.$$
With our choice of $\lambda\leq e^\delta$ and~(\ref{e.exp}) one deduces~(\ref{e.exp2}) as claimed.

By the domination $E^c\oplus E_3$, this implies that there exists $\nu\geq 1$
(which does not depend on $\cO_n$) such that
$$\prod_{i=0}^{k-1}\|Df^{-\nu}_{|E_3}(f^{-\nu i}(q))\|\leq e^{-k\rho},$$
where $\rho$ is uniform and positive if $\delta$ has been chosen close enough to $0$.
In particular one deduces that the local unstable manifold at $p$ has a uniform size.

We claim that the local stable manifold at $p$ also has a uniform size:
by property (*) the set $\Gamma_p$ contains the $\rho$-neighbourhood of $p$ in $\cD^ c_p$,
so the stable manifold at $p$ has a uniform size along $\cD^c_p$.
Now since $f^ k(\Gamma_p)$ belongs to $\cD^ c_{f^ k(p)}$ for each $k\geq 0$,
one has at any point $z\in \Gamma_p$
$$\|Df^{k}_{|\Gamma_p}(z) \|\leq C\lambda^k,$$
where $C>0$ is uniform.
The domination $E_1\oplus E^c$ implies that for $k\geq 0$ one has
$$\|Df^{k}_{|E_1}(p) \|\leq Ce^ {-k\rho},$$
where $\rho$ is positive and uniform As a consequence a uniform tubular neighbourhood
of $\Gamma_p$ in $\cD^ {cs}_p$ is contracted by forward iterations and contained in
the stable set of $p$. This proves the claim.

We have shown that each periodic orbit $\cO_n$
contains one point $p_n$ whose local stable and unstable manifolds have uniform sizes.
By considering an extracted subsequence of $(\cO_n)$ one may assume that all the
points $p_n$ are close together. This implies that all the orbits $\cO_n$
are homoclinically related together.
\end{proof}

\subsection{Large stable and unstable sets}
In case 2.c, one can create by perturbation some central segments
that are homoclinic to a minimal set $K'\subset U\cap A$.
Taking their limit, one gets a central segment for the set $A$ and the diffeomorphism $f$.
One can hope to argue as for proposition~\ref{p.R}.
The difficulty is that the bundles $E_1$ and $E_3$ on $A$ are not uniform.

\begin{lemma}\label{l.segment}
Let $K'\subset U\cap A$ be an invariant compact set and
let $x,y\in A\setminus K'$ such that $\omega(x)=\alpha(y)=K'$ and
$\overline{\{f^{n}(x), n\geq 0\}}\cup\overline{\{f^{-n}(y), n\geq 0\}}\subset U$.
If $x$ is not an endpoint of $I^s_{x}$ and $y$ is not an endpoint of $I^u_{y}$,
then there exists $\rho>0$ such that the following property holds.

For each neighbourhoods $\cU$ of $f$ and $V$ of $A$, there exists $g\in \cU$,
$z\in M\setminus K'$ and an arc $\gamma$ such that the following properties hold.
\begin{itemize}
\item[--] $f$ and $g$ coincide in a neighbourhood of $K'$;
the backward and the forward orbits of $z$ by $g$ accumulate on $K'$
and are contained in $V$;
\item[--] $\gamma$ has length $\rho$; the length the iterates $f^k(\gamma)$ is bounded by $\rho$
and goes to zero as $k\to\pm \infty$.
\end{itemize}
\end{lemma}
\begin{proof}
By assumption there exists $\rho>0$ such that the ball $B^{cs}_x$ centred at $x$
of radius $\rho$ in $\cD^{cs}_x$ is contained in the stable set of $x$ and
the ball $B^{cu}_y$ centred at $y$ of radius $\rho$ in $\cD^{cu}_y$ is contained in the unstable set of $y$
Note also that since $x$ is not in the basin of a periodic orbit,
the backward orbit of $x$ is disjoint from the forward orbit of $B^{cs}_x$
and the backward orbit of $y$ is disjoint from the forward orbit of $B^{cu}_y$.

The connecting lemma associates at $(f,\cU)$ some neighbourhoods
$W^+\subset \widehat W^+$ and $W^-\subset \widehat W^-$
of $x$ and $y$ respectively and an integer $N$.
If the neighbourhoods are small enough the iterates $f^{-i}(\widehat W^+)$
and $f^{j}(\widehat W^-)$ for $0\leq i,j\leq N$ are contained in $V$,
are pairwise disjoint, disjoint from the iterates $f^{k}(B^{cs}_x)$ and $f^ {-k}(B^{cu}_y)$ for $k>0$
and disjoint from $K'$.

Since $f$ belongs to $\cR_N\subset \cR_{Chain},$
lemma~\ref{l.connecting} applies and there exists
a segment or orbit $(x_0,\dots,f^\ell(x_0))$ contained in $V$
such that $x_0$ belongs to $W^-$
and $f^\ell(x_0)$ to $W^+$.
As in the proof of proposition~\ref{p.extension}, there exists a perturbation
$g\in \cU$ of $f$ with support in the iterates $f^{-i}(\widehat W^+)$
and $f^{j}(\widehat W^-)$ for $0\leq i,j\leq N$ such that the points $x$ and $y$
belong to a same orbit which accumulates on $K'$ in the future and in the past.
Note that the forward orbit of $x$ and the backward orbit of $y$ have not been perturbed.

Let us recall that $V_0$ is a fixed neighbourhood of $A$.
Let us write $x=g^r(y)$ and define
$$B^{cu}_x=g^r(B^{cu}_y)\cap \bigcap_{i=0}^r g^r(V_0).$$
This is a submanifold at $x$ whose backward iterates are contained in $V_0$ and
tangent to a centre-unstable cone.
Let $\gamma_x$ be the largest curve bounded by $x$ in the intersection $B^{cs}_x\cap B^{cu}_y$
whose iterates are all of length smaller than $\rho$.
By construction all the iterates of $\gamma_x$ are tangent to a centre-cone,
and by maximality, one of them $\gamma=f^n(\gamma_x)$ has length $\rho$.
We denote $z=f^n(x)$.
\end{proof}

We now conclude the proof of proposition~\ref{p.N}.

\begin{corollary}
Let $V$ be a neighbourhood of $A$, $K'\subset U\cap A$ be an invariant compact set and
let $x,y\in A\setminus K'$ such that $\omega(x)=\alpha(y)=K'$ and
$\overline{\{f^{n}(x), n\geq 0\}}\cup\overline{\{f^{-n}(y), n\geq 0\}}\subset U$.

If $x$ is not an endpoint of $I^s_{x}$ and $y$ is not an endpoint of $I^u_{y}$,
then there exists a sequence of hyperbolic periodic orbits $(\cO_n)$ contained in $V$
that are homoclinically related together,
which converges toward a compact subset of $A$, whose indices are $\dim(E^s)$,
and whose $\dim(E^s)+1$-th Lyapunov exponent goes to zero as
$n\to \infty$.
\end{corollary}
\begin{proof}
Let $\cS=\{1/n, n\geq 0\}$ and let $\cV$ be a countable basis of open sets of $M$.
For $\delta\in \cS$ and $V,W\in \cV$, let us denote by
$\cU(V,W,\delta)$ the set of diffeomorphisms which possess a hyperbolic periodic orbit $\cO$
contained in $V$ of index $\dim(E^s)$, whose $\dim(E^s)+1$-th exponent belongs to $(0,\delta)$ and
such that for some point $p\in \cO\cap W$ one has
\begin{equation}
\label{e.cont}
\forall k> 0,\quad \|D_{|E^c} f^k(p)\|< e^{k\delta},
\text{ and } \|D_{|E^c} f^{-k}(p)\|< e^{k\delta}.
\end{equation}
This set is open.
Let $\cR_{2.c}$ be the intersection of the $\cU(V,W,\delta)\cup \left(\diff^ 1(M)
\setminus \overline{\cU(V,W,\delta)}\right)$ over
the $3$-uples $(V,W,\delta)$. This is a G$_\delta$ and dense subset of $\diff^1(M)$.
Since $f$ is $C^1$ generic one can assume that it belongs to $\cR_{2.c}$.

Let us fix $\delta\in \cS$ small and let $V$ be a small neighbourhood of $A$.
By applying lemma~\ref{l.segment}, one gets a sequence of diffeomorphisms $g_n$
that converges to $f$ and a sequence of arcs $\gamma_n$.
Note that the constant $\rho$ provided by the lemma can be reduced so that the $\rho$-neighbourhood of $K'$ is contained
in $V$ and for any segment $\gamma'$ tangent to a central cone
and of length less than $\rho$, one has for any $x,y\in \gamma'$, and any $g$ in a neighbourhood of $f$,
$$\|Dg_{|E^c}(x)\|\leq e^{\delta/2} \|Dg_{|E^c}(y)\|.$$
Consequently, for each diffeomorphism $g_n$ and each arc $\gamma_n$
provided by the lemma~\ref{l.segment}, one has at any $z\in \gamma_n$,
\begin{equation}\label{e.cont2}
\forall k\geq 0, \|D_{|E^c} g_n^k(z)\|\leq e^{k\delta/2},
\text{ and } \|D_{|E^c} g_n^{-k}(z)\|\leq e^{k\delta/2}.
\end{equation}
A subsequence of the arcs $\gamma_n$ converges toward an arc $\gamma$
which is a central segment of $A$ for $g$:
in particular it is contained in the chain-recurrence
class of $K'$ and property~(\ref{e.cont2}) is still satisfied.
One can choose $W$ as a small neighbourhood of
the centre point of $\gamma$.

\begin{claim}
The diffeomorphism $f$ belongs to $\cU(V,W,\delta)$.
\end{claim}
\begin{proof}
Since $f$ belongs to $\cR_{2.c}$, it is enough to approximate $f$ by diffeomorphisms
in $\cU(V,W,\delta)$.

Let us consider $g_n$ arbitrarily close to $f$ and let $\gamma_n$
be a segment given by lemma~\ref{l.segment}.
Fix a point $p$ in $W\cap \gamma_n$.
The orbit of $p$ under $f_n$ accumulates in the future and in the past on the minimal set $K'$.
Hence, it is possible by the connecting lemma to close its orbit:
for a small perturbation $h_0$ of $g_n$, the orbit $\cO$ of $p$ is periodic,
its shadows the orbit of $p$ for $g_n$ during a large time interval $[-\ell,\ell]$
and spends the remaining iterates in a small neighbourhood of $K'$.
As a consequence, $\cO$ is contained in $V$.
If the orbit of $\cO$ is long enough, the Lyapunov exponent of $\cO$ along $E^c$ is arbitrarily close
to zero. Hence for a small perturbation $h$ of $h_0$ in a neighbourhood of $K'$,
the Lyapunov exponent of $\cO$ along $E^c$
belongs to $(0,\delta)$.
From~(\ref{e.cont2}) and the fact that the Lyapunov exponents along $E^c$ on $K'$
are all zero, the point $p$ satisfies
$$\forall k> 0, \quad \|D_{|E^c} h_0^k(p)\|< e^{k\delta},
\text{ and } \|D_{|E^c} h_0^{-k}(p)\|< e^{k\delta}.$$

This proves that the diffeomorphism $h$ belongs to $\cU(V,W,\delta)$ and is close to $f$,
implying the claim.
\end{proof}

Let $\cO$ be the periodic orbit provided by the claim.
If $\delta$ has been chosen small enough, the dominated splitting $E_1\oplus E^c\oplus E_3$
that holds on a neighbourhood of $A$ and~(\ref{e.cont}) ensure that for some uniform constants $C>0$
and $\lambda\in (0,1)$ one has
$$\forall k\geq 0,\quad \|D_{|E_1}f^k(p)\|\leq C\lambda^k
\text{ and } \|D_{|E_3}f^{-k}(p)\|\leq C\lambda^k.$$
In particular,
the point $p$ has large strong stable and strong unstable manifolds.
Similarly any point of $\gamma$ has large strong stable and strong unstable manifolds.
Since $p$ is close to the centre of $\gamma$, we are in the same situation as in proposition~\ref{p.R}
(we have replaced the uniformity on the bundles $E^s$ and $E^u$ by a control at $p$
and at the points of $\gamma$). By the same proof (see~\cite{crovisier-palis-faible}),
one deduces that the strong unstable manifold of some point of $\gamma$ intersects the strong stable
manifold of $p$ and he strong stable manifold of some point of $\gamma$ intersects the strong unstable
manifold of $p$. Hence the chain-recurrence class of $\cO$, $\gamma$ and $K'$ coincide.

We have found in $V$ a periodic orbit $\cO$ whose $\dim(E^s)+1$-th
Lyapunov exponent belongs to $(0,\delta)$ and whose chain-recurrence class intersects
$K'$: the same argument, for each $\delta_n\in \cS$ in a sequence which decreases to $0$
and each $V_n\in \cV$ in a sequence of neighbourhoods which decreases to $A$, will produce
a sequence of periodic orbits of index $\dim(E^s)$ whose chain-recurrence classes coincide
and by definition of $\cR_N$ these orbits are homoclinically related together.
\end{proof}

\section{Dynamics far from homoclinic bifurcations}\label{s.conclusion}
We complete in this section the proofs of the results announced in the introduction.

One chooses the residual set $\cR\subset \diff^1(M)\setminus \overline{\Tang\cup\Cycl}$
so that any $f\in \cR$ satisfies the following properties.
\begin{itemize}
\item[--] Each periodic point $p$ is hyperbolic and its chain-recurrence class
coincides with the homoclinic class $H(p)$ (see~\cite{BC}).
\item[--] For each periodic orbits $\cO,\cO'$ in a same chain-recurrence class,
and such that the index of $\cO$ is smaller or equal to the index of $\cO'$, then
it holds that $W^u(\cO)$ and $W^s(\cO')$ have a non-empty transverse intersection (see~\cite[lemma 2.9]{abcdw}).
\item[--] $\cR$ is contained in $\cR_R,\cR'_R,\cR_H,\cR_N$ and hence satisfies the conclusion of propositions~\ref{p.R}, \ref{p.R'}, \ref{p.CH} and~\ref{p.N}.
\end{itemize}

In the following $f$ is a diffeomorphism in $\cR$.
Let us recall that a homoclinic class of $f$ cannot contains periodic points of different indices:
indeed if $\cO,\cO'$ are contained in the same homoclinic class and if the index of $\cO$
is smaller than the index of $\cO'$, then $W^u(\cO)$ and $W^s(\cO')$ have a non-empty transverse intersection
which persits under small $C^1$-perturbations; with the connecting lemma it is possible to perturb the dynamics
so that $W^s(\cO)$ and $W^u(\cO')$ intersect (see~\cite[lemma 2.9]{abcdw}); one thus has built a heterodimensional
cycle and contradicted the fact that $f\notin \overline{\Cycl}$.

\paragraph{Minimal sets.}
Let us consider a minimal invariant compact set $K$ which has a partially hyperbolic
structure $E^s\oplus E^c\oplus E^u$ with $\dim(E^c)=1$ and such that
the Lyapunov exponent along $E^c$ of any invariant measure supported on $K$ is zero.

\begin{claim}\label{c.minimal}
One of the two following cases occur:
\begin{itemize}
\item[--] $K$ is a chain-recurrence class and has type (N);
\item[--] $K$ is contained in a homoclinic class having periodic pointy of index
$\dim(E^s)$ or $\dim(E^s)+1$ and whose central Lyapunov exponent is arbitrarily close to zero.
\end{itemize}
\end{claim}
\begin{proof}
First corollary~\ref{c.twist} asserts that $K$ can not have type (P) and be twisted.

If $K$ has type (P) and is not twisted, or if it has type (R) or (H),
then propositions~\ref{p.R} and \ref{p.CH} imply that $K$ is contained
in a homoclinic class associated to periodic orbits contained in an arbitrarily small neighbourhood
of $K$: in particular, the class contains periodic points of index $\dim(E^s)$ or $\dim(E^s)+1$
and whose $(\dim(E^s)+1)$-th Lyapunov exponent is arbitrarily close to zero.
The second item holds.

If $K$ has type (N), proposition~\ref{p.N} directly implies that
the first or the second item holds.
\end{proof}

\paragraph{Aperiodic classes.}
Let $\cC$ be an aperiodic class. The tangent bundle $T_\cC M$ is not uniformly contracted.
Hence by theorem~\ref{t.trichotomie} applied to the trivial decomposition
$E\oplus F=T_\cC M\oplus \{0\}$, the class contains a minimal set $K$ with a partially
hyperbolic structure as discussed in the previous paragraph.
Since $\cC$ is aperiodic, the set $K$ has type (N) and coincides with $\cC$.
This proves the addendum B.

\paragraph{Homoclinic classes.}
Let $H(p)$ be a non-trivial homoclinic class. It contains a dense set of periodic points
having the same index as $p$. As a consequence of Wen's corollary~\ref{c.wen},
there exists a dominated splitting $T_{H(p)}M=E\oplus F$ such that $\dim(E)$
coincides with the index of $p$.

Let us consider the case
$E$ is not uniformly contracted.
One can apply Theorem~\ref{t.trichotomie}:
\begin{itemize}
\item[--] Since the class $H(p)$ does not contain periodic points
with different indices, the first case of the theorem does not hold.

\item[--] If we are in the second case, the bundle $E$ splits again as $E'\oplus E^c_1$ with $\dim(E^c_1)=1$
and the homoclinic class contains periodic points homoclinically related to $p$
and whose Lyapunov exponent along $E^c_1$ is arbitrarily close to zero.

\item[--] If we are in the third case, the class contains a minimal set $K$ with a partially
hyperbolic structure such that $\dim(E^s_{|K})<\dim(E)$
as discussed in the first paragraph.
By the claim~\ref{c.minimal} $H(p)$ contains periodic points
of index less or equal to $\dim(E^s_{|E})+1$.
Since the index of the class is unique, one has $\dim(E^s_{|K})=\dim(E)-1$.
We have thus shown that
the class contains periodic points whose $\dim(E)$-th Lyapunov exponent is arbitrarily close to zero.
These points are dense in $H(p)$, hence by corollary~\ref{c.wen}, the bundle $E$ splits again as $E'\oplus E^c_1$ with $\dim(E^c_1)=1$.
\end{itemize}

From these two last cases, we claim that
the bundle $E'$ is uniformly contracted:
otherwise one would apply Theorem~\ref{t.trichotomie}, but
each case given by the theorem would imply that the class $H(p)$ contains periodic orbits of index smaller or equal to $\dim(E')$ which is a contradiction.
We then study the type of $E^c_{1}$:
the dynamics along $E^c_1$ above the periodic points is contracting, hence
$E^c_1$ cannot have type (N), (P) or (H)-repelling; by proposition~\ref{p.R'},
it cannot have type (R) either since this would imply that the class also contains periodic points of index
$\dim(E^s)$. As a consequence $E^c_1$ has type (H)-attracting.

The same argument applies to the other bundle $F$.
This concludes the proof of the addendum A and of the main theorem.

\paragraph{Finiteness.}
There could be apriori countably many homoclinic classes. However we have:
\smallskip

\emph{Any diffeomorphism $f\in \cR$ has only finitely many homoclinic classes
whose central bundle is two-dimensional.}
\smallskip

\begin{proof}
Let $H(\cO)$ be a homoclinic class of $f\in \cR$
whose non-hyperbolic central bundle $E^c_1\oplus E^c_2$ has dimension $2$.
From corollary~\ref{c.wen}, there exists $C>0$ and $\sigma>1$ such that for any $x\in H(\cO)$
and any unitary vectors $u\in E^c_{1,x}$ and $v\in E^c_{2,x}$, one has for any $n\geq 0$,
$$\|Df^n.u\|\leq C\sigma^{-n}\|D f^n.v\|.$$
Note that $\sigma$ does not depend on the class $H(\cO)$
but only on the uniform domination associated to periodic points of index $\dim(E^s\oplus E^c_1)$.
One fixes $a\in (0,\log\sigma)$.

The class contains periodic orbits whose Lyapunov exponent along $E^c_1$ is arbitrarily close to $0$
and by domination, one deduces that the Lyapunov exponent along $E^c_2$ for these orbits is larger than $a>0$.
The same property holds for $E^c_2$: the class contains some periodic orbits whose Lyapunov exponent along $E^c_1$
is smaller than $-a<0$ and whose Lyapunov exponent along $E^c_2$ is close to $0$.
The transition property of~\cite{BDP}
now implies that the class contains periodic orbits $\cO'$ whose Lyapunov exponent
along $E^c_1$ is close to $-\frac a 2$ (and, by domination, whose Lyapunov exponent along $E^c_2$ is larger than $\frac a 2$).

Let $\lambda=e^{-b}$ where $b$ (close to $-\frac a 2$) is the Lyapunov exponent of $\cO'$ along $E^c_1$.
There exists a point $p\in \cO'$ such that for each $n\geq 0$ one has:
$$\|Df^n_{|E^c_1}(p)\|\leq \lambda^n.$$
If $\tau$ is the period of $p$,
then $\|Df^\tau_{|E^c_1}(p)\|=\lambda^\tau$. Hence
for each $n\geq 0$ one has:
$$\|Df^{-n}_{|E^c_1}(p)\|^{-1}=\|Df^{n}_{|E^c_1}(f^{\tau-n}(p))\|\geq \lambda^{n}.$$
By the domination, one deduces that there exists some uniform constant $C>0$ such that
$$\|Df^{-n}_{|E^c_2}(p)\|\leq C.e^{-\frac {n.a} 2}.$$
Arguing as in proposition~\ref{p.anosov}, one deduces that $H(\cO)$ contains a periodic point
whose stable and unstable manifold have uniform sizes.

If one assumes that there exists an infinite number of such distinct homoclinic classes $H(\cO_n)$,
one considers a sequence of such periodic points $p_n\in H(\cO_n)$. Taking an extracting sequence,
the uniform size of the invariant manifolds implies that all the periodic points $p_n$ are homoclinically related,
contradicting the fact that the homoclinic classes are distinct.
\end{proof}

\vskip 1cm

\flushleft{\bf Sylvain Crovisier} \\
CNRS - Laboratoire Analyse, G\'eom\'etrie et Applications, UMR 7539,\\
Institut Galil\'ee, Universit\'e Paris 13, Avenue J.-B. Cl\'ement, 93430 Villetaneuse, France\\
\textit{E-mail:} \texttt{crovisie@math.univ-paris13.fr}\\

\end{document}

%% file: type.pstex_t
\begin{picture}(0,0)%
\includegraphics{type.pstex}%
\end{picture}%
\setlength{\unitlength}{1579sp}%
\begingroup\makeatletter\ifx\SetFigFont\undefined%
\gdef\SetFigFont#1#2#3#4#5{%
  \reset@font\fontsize{#1}{#2pt}%
  \fontfamily{#3}\fontseries{#4}\fontshape{#5}%
  \selectfont}%
\fi\endgroup%
\begin{picture}(13265,9210)(132,-10516)
\put(5401,-5761){\makebox(0,0)[b]{\smash{{\SetFigFont{8}{9.6}{\rmdefault}{\mddefault}{\updefault}{\color[rgb]{0,0,0}$p$}%
}}}}
\put(3001,-5761){\makebox(0,0)[b]{\smash{{\SetFigFont{8}{9.6}{\rmdefault}{\mddefault}{\updefault}{\color[rgb]{0,0,0}$p$}%
}}}}
\put(9001,-5761){\makebox(0,0)[b]{\smash{{\SetFigFont{8}{9.6}{\rmdefault}{\mddefault}{\updefault}{\color[rgb]{0,0,0}$p$}%
}}}}
\put(10801,-5761){\makebox(0,0)[b]{\smash{{\SetFigFont{8}{9.6}{\rmdefault}{\mddefault}{\updefault}{\color[rgb]{0,0,0}$p$}%
}}}}
\put(12601,-5761){\makebox(0,0)[b]{\smash{{\SetFigFont{8}{9.6}{\rmdefault}{\mddefault}{\updefault}{\color[rgb]{0,0,0}$p$}%
}}}}
\put(601,-5761){\makebox(0,0)[b]{\smash{{\SetFigFont{8}{9.6}{\rmdefault}{\mddefault}{\updefault}{\color[rgb]{0,0,0}$p$}%
}}}}
\put(13201,-1561){\makebox(0,0)[b]{\smash{{\SetFigFont{8}{9.6}{\rmdefault}{\mddefault}{\updefault}{\color[rgb]{0,0,0}$P_{UN}$}%
}}}}
\put(11401,-1561){\makebox(0,0)[b]{\smash{{\SetFigFont{8}{9.6}{\rmdefault}{\mddefault}{\updefault}{\color[rgb]{0,0,0}$P_{SN}$}%
}}}}
\put(9601,-1561){\makebox(0,0)[b]{\smash{{\SetFigFont{8}{9.6}{\rmdefault}{\mddefault}{\updefault}{\color[rgb]{0,0,0}$P_{SU}$}%
}}}}
\put(6001,-1561){\makebox(0,0)[b]{\smash{{\SetFigFont{8}{9.6}{\rmdefault}{\mddefault}{\updefault}{\color[rgb]{0,0,0}$H$}%
}}}}
\put(3601,-1561){\makebox(0,0)[b]{\smash{{\SetFigFont{8}{9.6}{\rmdefault}{\mddefault}{\updefault}{\color[rgb]{0,0,0}$N$}%
}}}}
\put(1201,-1561){\makebox(0,0)[b]{\smash{{\SetFigFont{8}{9.6}{\rmdefault}{\mddefault}{\updefault}{\color[rgb]{0,0,0}$R$}%
}}}}
\put(1201,-10411){\makebox(0,0)[b]{\smash{{\SetFigFont{8}{9.6}{\rmdefault}{\mddefault}{\updefault}{\color[rgb]{0,0,0}chain-recurrent}%
}}}}
\put(3601,-10411){\makebox(0,0)[b]{\smash{{\SetFigFont{8}{9.6}{\rmdefault}{\mddefault}{\updefault}{\color[rgb]{0,0,0}neutral}%
}}}}
\put(6001,-10411){\makebox(0,0)[b]{\smash{{\SetFigFont{8}{9.6}{\rmdefault}{\mddefault}{\updefault}{\color[rgb]{0,0,0}chain-hyperbolic}%
}}}}
\put(11401,-10411){\makebox(0,0)[b]{\smash{{\SetFigFont{8}{9.6}{\rmdefault}{\mddefault}{\updefault}{\color[rgb]{0,0,0}parabolic}%
}}}}
\end{picture}%

%% file: twist.pstex_t
\begin{picture}(0,0)%
\includegraphics{twist.pstex}%
\end{picture}%
\setlength{\unitlength}{1579sp}%
\begingroup\makeatletter\ifx\SetFigFont\undefined%
\gdef\SetFigFont#1#2#3#4#5{%
  \reset@font\fontsize{#1}{#2pt}%
  \fontfamily{#3}\fontseries{#4}\fontshape{#5}%
  \selectfont}%
\fi\endgroup%
\begin{picture}(15311,5064)(43,-6463)
\put(5986,-5901){\makebox(0,0)[b]{\smash{{\SetFigFont{8}{9.6}{\rmdefault}{\mddefault}{\updefault}{}%
}}}}
\put(5056,-2991){\makebox(0,0)[b]{\smash{{\SetFigFont{8}{9.6}{\rmdefault}{\mddefault}{\updefault}{}%
}}}}
\put(6301,-5181){\makebox(0,0)[b]{\smash{{\SetFigFont{8}{9.6}{\rmdefault}{\mddefault}{\updefault}{}%
}}}}
\put(5386,-2151){\makebox(0,0)[b]{\smash{{\SetFigFont{8}{9.6}{\rmdefault}{\mddefault}{\updefault}{}%
}}}}
\put(1306,-4476){\makebox(0,0)[b]{\smash{{\SetFigFont{12}{14.4}{\rmdefault}{\mddefault}{\updefault}{$p$}%
}}}}
\put(8581,-4371){\makebox(0,0)[b]{\smash{{\SetFigFont{12}{14.4}{\rmdefault}{\mddefault}{\updefault}{$q$}%
}}}}
\put(12346,-2691){\makebox(0,0)[b]{\smash{{\SetFigFont{12}{14.4}{\rmdefault}{\mddefault}{\updefault}{$E^c$}%
}}}}
\put(14431,-5046){\makebox(0,0)[b]{\smash{{\SetFigFont{12}{14.4}{\rmdefault}{\mddefault}{\updefault}{$E^{u}$}%
}}}}
\put(14161,-3321){\makebox(0,0)[b]{\smash{{\SetFigFont{12}{14.4}{\rmdefault}{\mddefault}{\updefault}{$E^{s}$}%
}}}}
\end{picture}%